\documentclass[11pt]{article}

\makeatletter
\disable@package@load{mathabx}{}
\disable@package@load{esint}{}
\makeatother

\usepackage[utf8]{inputenc}
\usepackage[T1]{fontenc}

\DeclareFontShape{T1}{cmr}{m}{scit}{<->ssub*cmr/m/sc}{}

\usepackage[margin=1in]{geometry}
\usepackage{amsmath,amsthm,amssymb,amsfonts}
\usepackage{mathrsfs}
\usepackage{mathtools}
\usepackage{esint}
\usepackage{xfrac}
\usepackage{nicefrac}
\usepackage{dsfont}
\usepackage{upgreek}
\usepackage{accents}

\let\olddiamond\diamond
\let\oldsquare\square
\usepackage{mathabx}
\renewcommand{\square}{\oldsquare}
\renewcommand{\diamond}{\olddiamond}

\usepackage{booktabs}
\usepackage{array}
\usepackage{paralist}
\usepackage{enumitem}
\usepackage{verbatim}
\usepackage[normalem]{ulem}
\usepackage{cancel}

\usepackage[dvipsnames]{xcolor}
\usepackage{graphicx}
\usepackage{subcaption}

\usepackage{tikz}
\usepackage{pgf}
\usetikzlibrary{calc}
\usetikzlibrary{external}
\usepackage[
colorlinks=true,
pdfstartview=FitV,
linkcolor=blue,
citecolor=blue,
urlcolor=blue
]{hyperref}

\usepackage{titlesec}

\newcommand{\addperiod}[1]{#1.}
\titleformat{\section}{\centering\normalfont\Large}{\thesection.}{0.5em}{}
\titleformat*{\subsection}{\bfseries}
\titleformat{\subsubsection}[runin]{\normalfont\bfseries}{\thesubsubsection.}{0.5em}{\addperiod}
\titleformat*{\paragraph}{\bfseries}
\titleformat*{\subparagraph}{\large\bfseries}

\usepackage[nottoc,notlot,notlof]{tocbibind}
\usepackage{fancyhdr}
\usepackage{extramarks}

\numberwithin{equation}{section}
\numberwithin{figure}{section}

\newtheorem{theorem}{Theorem}[section]

\newtheorem{proposition}[theorem]{Proposition}
\newtheorem{lemma}[theorem]{Lemma}

\newtheorem{theoremA}{Theorem}

\theoremstyle{definition}

\DeclareMathOperator{\dist}{dist}

\newcommand*{\tr}{\ensuremath{\mathrm{trace\,}}}

\DeclareMathSymbol{\shortminus}{\mathbin}{AMSa}{"39}

\let\originalleft\left
\let\originalright\right
\renewcommand{\left}{\mathopen{}\mathclose\bgroup\originalleft}
\renewcommand{\right}{\aftergroup\egroup\originalright}

\newcommand\thickbar[1]{\accentset{\rule{.45em}{.6pt}}{#1}}
\renewcommand{\bar}{\thickbar}

\makeatletter
\let\oldoint\oint
\renewcommand{\oint}{\oldoint\nolimits}
\makeatother

\newcommand*{\N}{\ensuremath{\mathbb{N}}}
\newcommand*{\Z}{\ensuremath{\mathbb{Z}}}

\newcommand*{\R}{\ensuremath{\mathbb{R}}}

\newcommand*{\Zd}{\ensuremath{\mathbb{Z}^d}}
\newcommand*{\Rd}{\ensuremath{\mathbb{R}^d}}
\newcommand{\E}{\mathbb{E}}
\renewcommand{\P}{\ensuremath{\mathbb{P}}}

\newcommand{\X}{\ensuremath{\mathcal{X}}}

\renewcommand{\S}{\mathcal{S}}

\newcommand{\RSZ}{\mathcal{R}}

\renewcommand{\a}{\mathbf{a}}

\newcommand{\g}{\mathbf{g}}
\newcommand{\h}{\mathbf{h}}
\renewcommand{\k}{\mathbf{k}}
\newcommand{\m}{\mathbf{m}}
\newcommand{\s}{\mathbf{s}}

\newcommand{\qq}{\mathbf{q}}

\newcommand{\bfA}{\mathbf{A}}

\newcommand{\bfE}{\mathbf{E}_0}

\newcommand{\ahom}{\bar{\a}}

\newcommand{\shom}{\bar{\mathbf{s}}}
\newcommand{\khom}{\bar{\mathbf{k}}}

\newcommand{\bfAhom}{\overline{\mathbf{A}}}

\newcommand{\uhom}{u_{\mathrm{hom}}}

\newcommand*{\Id}{\ensuremath{\mathrm{I}_d}}

\newcommand*{\Itwod}{\ensuremath{\mathrm{I}_{2d}}}

\newcommand{\Rskew}{\R^{d\times d}_{\mathrm{skew}}}

\newcommand{\nf}{\nicefrac}
\newcommand{\qand}{\quad \mbox{and} \quad}
\newcommand{\qqand}{\qquad \mbox{and} \qquad}

\makeatletter
\newcommand{\negphantom}{\v@true\h@true\negph@nt}
\newcommand{\neghphantom}{\v@false\h@true\negph@nt}
\newcommand{\negph@nt}{\ifmmode\expandafter\mathpalette
	\expandafter\mathnegph@nt\else\expandafter\makenegph@nt\fi}
\newcommand{\makenegph@nt}[1]{%
	\setbox\z@\hbox{\color@begingroup#1\color@endgroup}\finnegph@nt}
\newcommand{\finnegph@nt}{%
	\setbox\tw@\null
	\ifv@ \ht\tw@\ht\z@\dp\tw@\dp\z@\fi
	\ifh@\wd\tw@-\wd\z@\fi\box\tw@}
\newcommand{\mathnegph@nt}[2]{%
	\setbox\z@\hbox{$\m@th #1{#2}$}\finnegph@nt}
\makeatother

\newcommand{\Besov}[3]{\mathring{\phantom{B}}\negphantom{B}\underline{B}^{#1}_{#2,#3}}

\def\Xint#1{\mathchoice
	{\XXint\displaystyle\textstyle{#1}}%
	{\XXint\textstyle\scriptstyle{#1}}%
	{\XXint\scriptstyle\scriptscriptstyle{#1}}%
	{\XXint\scriptscriptstyle\scriptscriptstyle{#1}}%
	\!\int}
\def\XXint#1#2#3{{\setbox0=\hbox{$#1{#2#3}{\int}$}
		\vcenter{\hbox{$#2#3$}}\kern-.5\wd0}}

\def\fint{\Xint-}

\newcommand{\avsum}{\mathop{\mathpalette\avsuminner\relax}\displaylimits}
\makeatletter
\newcommand\avsuminner[2]{%
	{\sbox0{$\m@th#1\sum$}%
		\vphantom{\usebox0}%
		\ooalign{%
			\hidewidth
			\smash{\,\rule[.23em]{8.8pt}{1.1pt} \relax}%
			\hidewidth\cr
			$\m@th#1\sum$\cr
		}%
	}%
}
\makeatother

\makeatletter
\newcommand\avsuminnerr[2]{%
	{\sbox0{$\m@th#1\sum$}%
		\vphantom{\usebox0}%
		\ooalign{%
			\hidewidth
			\smash{\,\rule[.23em]{6pt}{0.7pt} \relax}%
			\hidewidth\cr
			$\m@th#1\sum$\cr
		}%
	}%
}
\makeatother

\newcommand{\cu}{\square}
\newcommand{\cus}{%
	\mathchoice
	{\mathrel{\rotatebox[origin=c]{45}{$\square$}\!\!}}
	{\mathrel{\rotatebox[origin=c]{45}{$\textstyle\square$}\!\!}}
	{\mathrel{\rotatebox[origin=c]{45}{$\scriptstyle{\square}$}}}
	{\mathrel{\rotatebox[origin=c]{45}{$\scriptscriptstyle{\square}$}}}}

\DeclareMathAlphabet{\mathmybb}{U}{bbold}{m}{n}
\newcommand{\indc}{{\mathmybb{1}}}

\makeatletter
\reenable@package@load{mathabx}
\reenable@package@load{esint}
\makeatother

\renewcommand{\cus}{\mathord{\lozenge}}

\usepackage{microtype}
\definecolor{paperlink}{RGB}{40,127,184}
\hypersetup{colorlinks=true,linktoc=all,linkcolor=paperlink,citecolor=paperlink,urlcolor=paperlink,pdfstartview=Fit,pdftitle={Homogenization at a polynomial scale in high contrast},pdfauthor={Scott Armstrong, Tuomo Kuusi, Am\'elie Loher}}
\tikzexternaldisable

\makeatletter
\renewcommand\section{\@startsection{section}{1}{\z@}%
	{-3.5ex \@plus -1ex \@minus -.2ex}%
	{2.3ex \@plus .2ex}%
	{\normalfont\Large\bfseries}}
\renewcommand\subsection{\@startsection{subsection}{2}{\z@}%
	{-3.25ex \@plus -1ex \@minus -.2ex}%
	{1.5ex \@plus .2ex}%
	{\normalfont\large\bfseries}}
\makeatother

\newcommand*{\history}{\mathscr H}
\newcommand*{\meanpenalty}{\mathfrak h}
\newcommand*{\cmet}{\mathfrak g}
\newcommand*{\Rpos}{\R^{d\times d}_{+}}

\setlist[itemize]{leftmargin=2em,itemsep=0.25em,topsep=0.35em}
\setlist[enumerate]{leftmargin=2.25em,itemsep=0.25em,topsep=0.35em}
\title{Homogenization at a polynomial scale in high contrast}
\author{
  Scott Armstrong\thanks{CNRS and Laboratoire Jacques-Louis Lions, Sorbonne Universit\'e; Courant Institute of Mathematical Sciences, New York University. Email: \texttt{scottnarmstrong@gmail.com}.}
  \and
  Tuomo Kuusi\thanks{Department of Mathematics and Statistics, University of Helsinki. Email: \texttt{tuomo.kuusi@helsinki.fi}.}
  \and
  Am\'elie Loher\thanks{All Souls College, University of Oxford. Email: \texttt{amelie.loher@all-souls.ox.ac.uk}.}
}
\date{23 September 2026}

\begin{document}

\maketitle

\begin{abstract}
We prove an upper bound on the homogenization length scale which is a fixed power of the coarse ellipticity ratio for coefficient fields that satisfy a finite range of dependence assumption. 
\end{abstract}

\setcounter{tocdepth}{1}
\tableofcontents

\section{Introduction}
\label{s.introduction}

We study the large-scale behavior of solutions of the elliptic equation
\begin{equation*}
	-\nabla\cdot\a(x)\nabla u=0\qquad\text{in }U\subseteq\Rd\,,
\end{equation*}
where~$d\geq2$ and~$\a(\cdot)$ is a matrix-valued, elliptic, stationary random coefficient field. Under a suitable decorrelation assumption, the equation homogenizes: on large scales its solutions are well-approximated by those of an equation with a deterministic constant coefficient matrix. The question we consider here is how far we need to zoom out before we see homogenization. Given a fixed error tolerance, the \emph{homogenization length} is the random scale above which the effective equation gives an accurate description. The main result of this paper is an estimate for the homogenization length as a function of the ellipticity contrast, in the case of a coefficient field satisfying a finite range of dependence. We measure lengths in units of the range of dependence, and prove the homogenization length is at most a power of the ellipticity contrast. 

\smallskip

Similar questions arise for electrical conduction in random resistor networks and mixtures of conducting materials; see~\cite{Clerc}. Consider a network whose bonds have conductances~$1$ and~$\lambda$, with~$0<\lambda\ll1$. If~$\lambda=0$, current can flow only along connected paths of bonds with conductance~$1$. A small positive~$\lambda$ also allows current to cross the poorly conducting bonds between these paths. Both the arrangement of the bonds and the ratio of their conductances therefore matter, especially near the percolation threshold, where the well-conducting bonds are close to forming a connected network across the material. The effect of a small positive conductance near this threshold was modeled in~\cite{Straley1976} and studied through random walks and numerical simulations in~\cite{HongEtAl1986}. This leads to our question: how far must we zoom out before the material behaves approximately like a homogeneous conductor, and is a length polynomial in the contrast enough? See~\cite{AK.ICM} for further discussion.

\smallskip

The continuum example from~\cite{AK.HC} makes the question concrete. Take the union of unit balls centered at the points of a homogeneous Poisson process, put~$\a=\Id$ in this union, and put~$\a=\lambda\Id$ outside it. Choose the critical percolation intensity and hold it fixed as~$\lambda$ decreases. The geometry of the conducting phase and the finite dependence range of the coefficient field remain unchanged, while the increasing contrast makes transport more sensitive to the connectivity of that phase. Related models of overlapping disks and spheres, including mixtures in which both phases have finite conductivity, were studied numerically in~\cite{TobochnikLaingWilson1990}. Our main result implies an upper bound of $O(\lambda^{-C})$ on the homogenization length, with~$C$ depending only on the dimension, and thus the estimate is uniform over coefficient laws satisfying the ellipticity and decorrelation assumptions. 

\smallskip

The quantitative theory of stochastic homogenization gives detailed error estimates when the ellipticity ratio is held fixed; see~\cite{GO1,GO2,AKMBook}. The dependence of the homogenization length on that ratio was left implicit in these works, but tracking the constants in their proofs gives bounds at least exponential in a positive power of the contrast. The reason can be seen in the fact that these works all use an iteration-in-the-scales whose gain at each step is an inverse power of the ellipticity ratio. A fixed reduction of the error then requires a number of steps which grows like a power of the contrast. Since the length grows geometrically with the number of steps, this polynomial cost in the iteration count necessarily becomes an exponential cost in the length bound; see the discussion in~\cite[Section~1.2]{AK.HC}.

\smallskip

The first two authors obtained a much better bound in~\cite{AK.HC} by introducing an iteration based on intrinsic coarse-grained coefficients and coarse-grained elliptic estimates. Since the notion of ``coarse ellipticity'' is scale-dependent and therefore changing, improvements feed into subsequent steps of the iteration. For uniformly elliptic coefficients with ellipticity ratio~$\frac\Lambda\lambda$ and unit range of dependence, the  results of~\cite{AK.HC} give an upper bound for the homogenization length scale of 
\begin{equation}
\exp\biggl(C\log^2\left(1+\frac{\Lambda}{\lambda}\right)\biggr)
=
\left(1+\frac{\Lambda}{\lambda}\right)^{C\log\left(1+\frac{\Lambda}{\lambda}\right)}\,,
\label{e.exp.log.squared}
\end{equation}
with stretched-exponential control of the tails. Here~$C$ depends only on the dimension and the requested accuracy. 

\smallskip

In this paper, we replace the logarithmically growing exponent in~\eqref{e.exp.log.squared} by a fixed dimensional constant, obtaining a polynomial upper bound for the length scale,
\begin{equation*}
C\left(2+\frac{\Lambda}{\lambda}\right)^C\,. 
\end{equation*}
This answers the polynomial upper-bound question we raised in~\cite{AK.HC,AK.ICM}. 
The main result, Theorem~\ref{t.polynomial.entry}, is formulated in terms of the (annealed) coarse ellipticity contrast. The quenched, quantitative homogenization estimates are consequence of the main result; these are stated below in Theorem~\ref{t.uniform.homogenization}.

\subsection{Main result: polynomial entry into small contrast}
\label{ss.main.result}

Throughout we fix~$d\geq2$ and consider measurable coefficient fields~$\a:\Rd\to\R^{d\times d}$ whose symmetric parts are positive definite, almost everywhere, and which satisfy the qualitative, local uniform ellipticity condition
\begin{equation}
	\s,\ \s^{-1},\ \k^t\s^{-1}\k\in L^\infty_{\mathrm{loc}}(\Rd;\R^{d\times d})\,,
	\qquad
	\s\coloneqq\frac12(\a+\a^t)\,,\qquad\k\coloneqq\frac12(\a-\a^t)\,.
	\label{e.qualitative.ellipticity}
\end{equation}
The fields need not be symmetric or globally bounded, or globally uniformly elliptic. Since our quantitative estimates do not depend on the implicit bounds in~\eqref{e.qualitative.ellipticity}, the assumption can be relaxed to coefficient fields which are not locally bounded and/or uniformly elliptic under conditions which ensure uniqueness of solutions (by approximation). 

\smallskip

Let~$\Omega$ denote the set of such coefficient fields. For a Borel set~$U\subseteq\Rd$, denote by~$\mathcal F(U)$ the~$\sigma$-algebra on~$\Omega$ generated by the random variables
\begin{equation*}
	\a\mapsto\int_{\Rd}e'\cdot\a(x)e\,\varphi(x)\,dx\,,\qquad e,e'\in\Rd,\ \varphi\in C^\infty_c(U)\,,
\end{equation*}
and set~$\mathcal F\coloneqq\mathcal F(\Rd)$. Translations act on~$\Omega$ by~$T_z\a\coloneqq\a(\cdot+z)$. We write~$\cu_m\coloneqq\big(\!-\frac12 3^m,\frac12 3^m\big)^d$ for~$m\in\Z$.

\smallskip

For a bounded Lipschitz domain~$U\subseteq\Rd$, the coarse-grained matrix~$\bfA(U;\a)$ is defined by the variational problem on~$U$ with gradient and flux data in~\cite[(2.13)--(2.14)]{AK.HC}. It is a symmetric positive definite~$2d \times 2d$ matrix which depends only on the restriction of~$\a$ to~$U$. Given a probability measure~$\P$ on~$(\Omega,\mathcal F)$, its annealed value is~$\bfAhom(U)\coloneqq\E[\bfA(U;\a)]$.

\smallskip

We consider a probability measure~$\P$ on~$(\Omega,\mathcal F)$ satisfying the following three assumptions.
\begin{enumerate}[label=(\textrm{P\arabic*}),leftmargin=3em,itemsep=0.5em]
\item \emph{Stationarity with respect to~$\Zd$-translations:}
\label{a.stationarity}
\begin{equation*}
	\P\circ T_z=\P\qquad(z\in\Zd)\,.
\end{equation*}
\item \emph{Unit range of dependence:}
\label{a.finite.range}
for all Borel sets~$U,V\subseteq\Rd$ with~$\dist(U,V)\geq1$, the~$\sigma$-algebras~$\mathcal F(U)$ and~$\mathcal F(V)$ are independent under~$\P$.
\item \emph{Coarse ellipticity above a random scale:}
\label{a.coarse.ellipticity}
there exist a symmetric positive definite matrix~$\bfE\in\R^{2d\times2d}$, an exponent~$\gamma\in[0,1)$, an increasing function~$\Psi_{\S}:\R_+\to[1,\infty)$ with a constant~$K_{\Psi_{\S}}\in(1,\infty)$, and a nonnegative random variable~$\S$ on~$\Omega$ such that
\begin{equation}
	\P[\S>t]\leq\Psi_{\S}(t)^{-1}\quad(t>0)\qand t\Psi_{\S}(t)\leq\Psi_{\S}(K_{\Psi_{\S}}t)\quad(t\geq1)\,,
	\label{e.source.tail}
\end{equation}
and such that, almost surely, for every~$m\in\Z$,
\begin{equation}
	3^m\geq\S\quad\Longrightarrow\quad\bfA(y+\cu_k;\a)\leq3^{\gamma(m-k)}\bfE\qquad\bigl(k\in\Z,\ k\leq m,\ y\in3^k\Zd\cap\cu_m\bigr)\,.
	\label{e.coarse.ellipticity}
\end{equation}
\end{enumerate}

\smallskip

Assumption~\ref{a.coarse.ellipticity} is the random-scale coarse ellipticity condition of~\cite{AK.HC}. In a cube of side length at least~$\S$, the coarse-grained matrices of its subcubes satisfy the displayed bound, which allows them to grow as the subcube size decreases, including at negative scales. The restriction~$\gamma<1$ makes the fine-scale boundary contributions summable when we compare cube geometries. The tail condition~\eqref{e.source.tail} supplies the moments of~$\S$ needed in the proof; its growth constant~$K_{\Psi_{\S}}$ enters the polynomial length bound.

\smallskip

The reference block measures the size of the coarse ellipticity bound. Write it in the block form
\begin{equation}
	\bfE=\begin{pmatrix}\s_0+\k_0^t\s_{*,0}^{-1}\k_0&-\k_0^t\s_{*,0}^{-1}\\-\s_{*,0}^{-1}\k_0&\s_{*,0}^{-1}\end{pmatrix}\,,\qquad\s_0,\s_{*,0}>0\,.
	\label{e.reference.block}
\end{equation}
The reference ellipticity ratio compares the largest upper bound with the smallest lower bound supplied by~$\bfE$, in the fixed Euclidean coordinates. We define it by
\begin{equation}
	\Pi\coloneqq|\s_{*,0}^{-1}|\min_{\h\in\Rskew}\bigl|\s_0+(\k_0-\h)^t\s_{*,0}^{-1}(\k_0-\h)\bigr|\geq1\,.
	\label{e.reference.aspect.ratio}
\end{equation}
The minimization removes a constant skew-symmetric part~$\h$, which does not affect the equation. The reference-block comparison in~\cite{AK.HC} gives~$\s_{*,0}\leq\s_0$ and hence~$\Pi\geq1$. The ratio~$\Pi$ is a fixed parameter of our assumptions and takes the place of the microscopic ellipticity ratio in the length bound. It can be large even for a constant, anisotropic coefficient.

\smallskip

Write the annealed coarse-grained matrix in the block form
\begin{equation}
	\bfAhom(U)=\begin{pmatrix}\shom+\khom^t\shom_*^{-1}\khom&-\khom^t\shom_*^{-1}\\-\shom_*^{-1}\khom&\shom_*^{-1}\end{pmatrix}(U)\,.
	\label{e.annealed.schur}
\end{equation}
We measure the contrast at a Euclidean scale~$m$ by
\begin{equation}
	\Theta_m\coloneqq\min_{\h\in\Rskew}\bigl|\shom_*^{-\nf12}\bigl(\shom+(\khom-\h)^t\shom_*^{-1}(\khom-\h)\bigr)\shom_*^{-\nf12}\bigr|(\cu_m)\,.
	\label{e.Theta.m}
\end{equation}
The contrast $\Theta_m$ measures how far the annealed coarse-grained matrix at scale $3^m$ is from agreeing with its dual. As homogenization takes place, these two matrices approach the same limit and $\Theta_m$ approaches one. Smallness of $\Theta_m-1$ therefore expresses agreement of the primal and dual matrices, not closeness of the effective coefficient to a scalar matrix. In particular, a constant coefficient has $\Theta_m=1$, however anisotropic it may be.

\smallskip
The following theorem is the main result of the paper.

\begin{theoremA}[Polynomial entry into small contrast]
\label{t.polynomial.entry}
Assume~$\P$ satisfies~\ref{a.stationarity},~\ref{a.finite.range}, and~\ref{a.coarse.ellipticity}. 
For every~$\sigma\in(0,1]$, there exists~$C(\sigma,d,\gamma)<\infty$ such that
\begin{equation}
\Theta_{m}\leq1+\sigma
\,, \qquad 
\mbox{for every} \quad 
m\geq 
m_{\rm ent}\coloneqq\bigl\lceil C\log_3(2+\Pi K_{\Psi_{\S}})\bigr\rceil
\,.
\label{e.polynomial.entry}
\end{equation}
\end{theoremA}
Since scale~$m$ corresponds to length~$3^m$, the theorem gives the polynomial length bound
\begin{equation*}
	3^{m_{\rm ent}}\leq3(2+\Pi K_{\Psi_{\S}})^{C(\sigma,d,\gamma)}\,.
\end{equation*}
The constant depends only on the tolerance~$\sigma$, the dimension~$d$, and the exponent~$\gamma$; the coefficient law and the reference block enter the displayed length only through~$\Pi$ and~$K_{\Psi_{\S}}$. The deterministic scale~$m_{\rm ent}$ is a sufficient scale for reaching small annealed contrast; it need not be the first such scale, and no bound uniform as~$\gamma\uparrow1$ is asserted.

\smallskip

The proof retains estimates for fluctuations and changes of the mean as the scale increases and the averaging geometry is gradually adjusted. By counting error reduction, geometric progress, and determinant loss together, we find an interval on which a comparison of the primal and dual optimizer energies closes the gap. This takes only a logarithmic number of scales; Section~\ref{ss.outline} explains how we avoid multiplying the two logarithmic costs in~\cite{AK.HC}.

\subsection{Consequences of the main result}
\label{ss.consequences}

Theorem~\ref{t.polynomial.entry} supplies the starting scale for the small-contrast iteration of~\cite{AK.HC}. Theorem~\ref{t.algebraic.convergence} gives the resulting algebraic convergence, and Theorems~\ref{t.uniform.homogenization} and~\ref{t.random.homogenization} combine it with the deterministic homogenization estimates of~\cite{AK.HC} to obtain PDE estimates above a random radius.

\paragraph{Algebraic convergence.}
A fixed small-contrast bound does not yet give a rate of convergence to the homogenized matrix. The next theorem supplies this rate and identifies the limit. Its proof uses the moments of~$\S$ to initialize the small-contrast iteration beyond the entry scale of Theorem~\ref{t.polynomial.entry}.

\begin{theoremA}[Algebraic convergence at a polynomial scale]
\label{t.algebraic.convergence}
There exist constants~$C(d,\gamma)<\infty$ and~$\kappa(d,\gamma)>0$ such that, if~$\P$ satisfies~\ref{a.stationarity},~\ref{a.finite.range}, and~\ref{a.coarse.ellipticity}, then there is a deterministic scale~$m_0\in\N$ satisfying
\begin{equation}
	m_0\leq\bigl\lceil C\log_3(2+\Pi K_{\Psi_{\S}})\bigr\rceil\,,
	\label{e.algebraic.entry}
\end{equation}
and
\begin{equation}
	\Theta_{m_0+j}-1\leq3^{-\kappa j}\qquad(j\in\N)\,.
	\label{e.algebraic.contrast.decay}
\end{equation}
There is a deterministic symmetric positive definite matrix~$\bfAhom\in\R^{2d\times2d}$ such that
\begin{equation}
	0\leq\bfAhom(\cu_{m_0+j})-\bfAhom\leq6 \cdot 3^{-\kappa j}\bfAhom\qquad(j\in\N)\,.
	\label{e.algebraic.block.decay}
\end{equation}
The matrices in the block representation of~$\bfAhom$, defined as in~\eqref{e.annealed.schur}, satisfy~$\shom_* = \shom>0$ and~$\khom=-\khom^t$. We define the associated coefficient matrix by~$\ahom\coloneqq\shom+\khom$.
\end{theoremA}

The decay exponent depends only on~$d$ and~$\gamma$; the original contrast affects the starting scale. The moment initialization and iteration are given in Section~\ref{ss.algebraic.convergence}.

\paragraph{Uniform ellipticity.}
The next consequence gives the polynomial improvement over~\eqref{e.exp.log.squared}, in terms of the usual ellipticity ratio~$\nf{\Lambda}{\lambda}$. The rescaled random radius retains the tail~$\exp(-t^d)$ of~\cite[Theorem~A]{AK.HC}, and the Dirichlet error decays algebraically above that radius.

\begin{theoremA}[Quantitative homogenization under uniform ellipticity]
\label{t.uniform.homogenization}
Assume that~$\P$ satisfies~\ref{a.stationarity} and~\ref{a.finite.range}, and that, almost surely, for almost every~$x\in\Rd$ and every~$e\in\Rd$,
\begin{equation}
	\lambda|e|^2\leq e\cdot\a(x)e\qand\Lambda^{-1}|e|^2\leq e\cdot\a(x)^{-1}e\,,\qquad 0<\lambda\leq1\leq\Lambda<\infty\,.
	\label{e.uniform.ellipticity}
\end{equation}
Let~$\ahom$ be the coefficient matrix identified in Theorem~\ref{t.algebraic.convergence}, set~$\overline\lambda\coloneqq\lambda_{\min}(\shom)$, and define
\begin{equation}
	E_r\coloneqq\bigl\{x\in\Rd:x\cdot\shom^{-1}x<\overline\lambda^{-1}r^2\bigr\}\qquad(r>0)\,.
	\label{e.homogenized.ellipsoids}
\end{equation}
There exist constants~$\kappa(d)>0$ and~$C_0(d)<\infty$ such that, for every~$\delta>0$, there exist~$C(\delta,d)\geq1$ and a random variable~$\X\geq1$ satisfying
\begin{equation}
	\P\Bigl[\X\geq C\Big(2+\frac{\Lambda}{\lambda}\Big)^C t\Bigr]\leq\exp\big(\!-t^d\big)\qquad(t\geq1)\,,
	\label{e.uniform.scale.tail}
\end{equation}
such that, almost surely, the following hold.
\begin{itemize}
\item \emph{Dirichlet problem.} For every~$r\geq\X$,~$f\in L^2(E_r)$, and~$g\in H^1(E_r)$, let~$u,\uhom\in g+H^1_0(E_r)$ be the weak solutions of
\begin{equation*}
	-\nabla\cdot\a\nabla u=f\qand-\nabla\cdot\ahom\nabla\uhom=f\qquad\text{in }E_r\,.
\end{equation*}
Then
\begin{equation}
	\|u-\uhom\|_{L^2(E_r)}\leq\delta\left(\frac{r}{\X}\right)^{-\kappa}\bigl(r\|\nabla g\|_{L^2(E_r)}+r^2\|f\|_{L^2(E_r)}\bigr)\,.
	\label{e.uniform.dirichlet}
\end{equation}
\item \emph{Large-scale energy estimate.} For every~$R\geq\X$ and every weak solution~$u\in H^1(E_R)$ of~$-\nabla\cdot\a\nabla u=0$ in~$E_R$,
\begin{equation}
	\sup_{r\in[\X,R]}\fint_{E_r}\nabla u\cdot\a\nabla u\leq C_0\fint_{E_R}\nabla u\cdot\a\nabla u\,.
	\label{e.uniform.energy}
\end{equation}
\end{itemize}
\end{theoremA}

The radius~$\X$ is common to all observation radii and data in the statement, and to the energy estimate. Changing the tolerance changes the starting radius, not the decay exponent~$\kappa(d)$. In the proof, given in Section~\ref{ss.uniform.homogenization}, the scale~$\S$ can be taken deterministic under~\eqref{e.uniform.ellipticity}; it is not imposed on the general coefficient fields considered here.

\paragraph{The random scale.}
For the general coefficient fields of Theorem~\ref{t.polynomial.entry}, the random scale~$\S$ remains visible in the stochastic estimate. The next theorem gives Dirichlet homogenization, correctors, and large-scale regularity above one random radius. Its tail bound~\eqref{e.random.scale.tail} combines a stretched-exponential term with a rescaled term from the tail of~$\S$, while retaining a polynomial deterministic length factor. The coefficient matrix is~$\ahom=\shom+\khom$ from Theorem~\ref{t.algebraic.convergence}, and~$E_r$ denotes the ellipsoid in~\eqref{e.homogenized.ellipsoids} formed with this matrix.

\smallskip

The homogenization and corrector errors will be measured in negative Sobolev norms of gradients and centered fluxes. We use volume-normalized norms so that their scaling is explicit. Write~$\|F\|_{\underline L^2(U)}^2\coloneqq\fint_U|F|^2$. For~$0<s<1$, the volume-normalized fractional norm is
\begin{equation}
	\|F\|_{\underline H^s(U)}^2\coloneqq |U|^{-\nf2d s}\fint_U|F|^2+\fint_U\int_U\frac{|F(x)-F(y)|^2}{|x-y|^{d+2s}}\,dy\,dx\,.
	\label{e.physical.fractional.norm}
\end{equation}
At the endpoint we use
\begin{equation}
	\|F\|_{\underline H^1(U)}^2\coloneqq |U|^{-\nf2d}\fint_U|F|^2+\fint_U|\nabla F|^2\,.
	\label{e.physical.endpoint.norm}
\end{equation}
For~$0<s\leq1$, the negative norm is the compact-test dual with averaged pairing:
\begin{equation}
	\|F\|_{\underline H^{-s}(U)}\coloneqq\sup\biggl\{\fint_U F\cdot\psi:\psi\in C^\infty_c(U;\Rd),\ \|\psi\|_{\underline H^s(U)}\leq1\biggr\}\,.
	\label{e.physical.negative.norm}
\end{equation}
These definitions also apply componentwise. Sobolev space membership without an underline has its usual meaning; in particular, the ordinary~$H^1$ norm in Theorem~\ref{t.uniform.homogenization} is unchanged.

\smallskip

We use the weighted solution spaces of~\cite[Section~2.1]{AK.HC}. The space~$H^1_\s(U)$ is the smooth completion for the norm
\begin{equation*}
	\|u\|_{H^1_\s(U)}^2\coloneqq\|u\|_{L^1(U)}^2+\int_U\nabla u\cdot\s\nabla u\,.
\end{equation*}
The space~$H^1_\a(U)$ is the smooth completion for the graph norm obtained by adding~$\|\nabla\cdot\k\nabla u\|_{H^{-1}_\s(U)}^2$, where~$H^{-1}_\s(U)$ is the dual of~$H^1_{\s,0}(U)$. The zero-boundary spaces are the respective closures of~$C^\infty_c(U)$. Weak equations are understood in distributions; condition~\eqref{e.qualitative.ellipticity} makes the gradients and fluxes locally integrable.

\begin{theoremA}[Homogenization with a random scale]
\label{t.random.homogenization}
Assume that~$\P$ satisfies~\ref{a.stationarity},~\ref{a.finite.range}, and~\ref{a.coarse.ellipticity}. There exist constants~$c(d)>0$,~$C(d,\gamma)\geq1$,~$c_{\rm src}(d,\gamma)>0$, and~$\kappa(d,\gamma)>0$, and a measurable random variable~$\X\geq\max\{1,\S\}$ such that
\begin{equation}
	\P\bigl[\X\geq C(2+\Pi K_{\Psi_{\S}})^C t\bigr]\leq\exp\bigl(-c(d)t^{d-2\gamma}\bigr)+\Psi_{\S}\bigl(c_{\rm src}(d,\gamma)t\bigr)^{-1}\qquad(t\geq1)\,.
	\label{e.random.scale.tail}
\end{equation}
Almost surely, the following hold simultaneously, with the same~$\X$.
\begin{enumerate}[label=\emph{(\alph*)}]
\item \emph{Dirichlet homogenization.} Let~$s\in[\nf{1+\gamma}{4},\nf{1}{2})$ and let~$U\subseteq E_1$ be a bounded Lipschitz domain.
We measure the geometry of the domain in coordinates in which the homogenized equation becomes the Laplace equation. 
There is a constant~$C_0(U,s,d)<\infty$, depending on the size and Lipschitz regularity of the transformed domain $\overline\lambda^{1/2}\overline{\mathbf s}^{-1/2}U$, with the following property. For every~$0<\varepsilon\leq\X^{-1}$ and~$g\in W^{1,\infty}(U)\cap H^{1+s}(U)$, set~$\a^\varepsilon(x)\coloneqq\a(\nf{x}{\varepsilon})$. The unique weak solutions~$u^\varepsilon\in g+H^1_{\a^\varepsilon,0}(U)$ and~$\uhom\in g+H^1_0(U)$ of
\begin{equation*}
	-\nabla\cdot\a^\varepsilon\nabla u^\varepsilon=0\qand-\nabla\cdot\ahom\nabla\uhom=0\qquad\text{in }U
\end{equation*}
satisfy
\begin{multline}
	\big\|\shom^{\nf12}(\nabla u^\varepsilon-\nabla\uhom)\big\|_{\underline H^{-s}(U)}+\big\|\shom^{-\nf12}((\a^\varepsilon-\khom)\nabla u^\varepsilon-\shom\nabla\uhom)\big\|_{\underline H^{-s}(U)}
	\\\leq C_0(\varepsilon\X)^\kappa\|\shom^{\nf12}\nabla g\|_{\underline H^s(U)}\,.
	\label{e.random.dirichlet}
\end{multline}
\item \emph{Correctors and large-scale regularity.} There is a slope-linear family of~$\Zd$-stationary gradient fields~$\{\nabla\phi_e:e\in\Rd\}$, with potentials~$\phi_e\in H^1_{\s,\mathrm{loc}}(\Rd)$ defined modulo constants, satisfying
\begin{equation*}
	-\nabla\cdot\a(e+\nabla\phi_e)=0\qquad\text{in }\Rd\,.
\end{equation*}
For every~$e\in\Rd$ and~$r\geq\X$,
\begin{equation}
	r^{-1}
	\big\|\shom^{\nf12}\nabla\phi_e\|_{\underline H^{-1}(E_r)}+r^{-1}
	\big\|\shom^{-\nf12}((\a-\khom)(e+\nabla\phi_e)-\shom e)\|_{\underline H^{-1}(E_r)}\leq C|\shom^{\nf12}e|\left(\frac{r}{\X}\right)^{-\kappa}.
	\label{e.random.corrector}
\end{equation}
For every~$\theta\in(0,1)$, the entire weak solutions~$v\in H^1_{\s,\mathrm{loc}}(\Rd)$ satisfying
\begin{equation}
	-\nabla\cdot\a\nabla v=0\quad\text{in }\Rd\qand\lim_{r\to\infty}r^{-(1+\theta)}\|v\|_{\underline L^2(B_r)}=0
	\label{e.random.liouville.growth}
\end{equation}
are exactly the functions~$x\mapsto e\cdot x+\phi_e(x)+c$, for~$e\in\Rd$ and~$c\in\R$, up to equality almost everywhere.

\smallskip

For every~$R\geq\X$ and weak solution~$u\in H^1_\a(E_R)$ of~$-\nabla\cdot\a\nabla u=0$ in~$E_R$,
\begin{equation}
	\sup_{r\in[\X,R]}\|\s^{\nf12}\nabla u\|_{\underline L^2(E_r)}\leq C\|\s^{\nf12}\nabla u\|_{\underline L^2(E_R)}\,.
	\label{e.random.energy}
\end{equation}
Moreover, for every~$\theta\in(0,1)$ there exists~$e\in\Rd$ such that, for every~$r\in[\X,R]$,
\begin{equation}
	\|\s^{\nf12}(\nabla u-e-\nabla\phi_e)\|_{\underline L^2(E_r)}\leq C(d,\gamma,\theta)\left(\frac{r}{R}\right)^\theta\|\s^{\nf12}\nabla u\|_{\underline L^2(E_R)}\,.
	\label{e.random.regularity}
\end{equation}
The radius~$\X$ and the corrector family are independent of~$s$ and~$\theta$.
\end{enumerate}
\end{theoremA}

The correctors describe the affine functions appropriate to the heterogeneous equation. They give both the entire solutions in the Liouville statement and the approximating functions in the last estimate. The proof is given in Sections~\ref{ss.random.dirichlet} and~\ref{ss.random.correctors}.

\subsection{Proof strategy}
\label{ss.outline}

Theorem~\ref{t.polynomial.entry} asserts that the annealed coarse-grained matrix and its dual become close in relative terms. We call their relative separation the \emph{duality gap}, following the variational approach of comparing a subadditive energy with its superadditive dual in~\cite{AS,AKM1,AKMBook}. Small random fluctuations alone do not close this gap: the random coarse-grained matrices can concentrate around their mean without that mean being close to its dual.

Our proof first finds a range of scales on which the coarse matrices fluctuate little, their means change little, and the cube shapes balance the estimates for gradients and fluxes. We then compare the energies of the primal and dual optimizers to show that the gap is small. This last comparison is made only once. The preceding iteration prepares the estimates and geometry it needs; it does not require the duality gap to contract at every step. The main task is to find these conditions within a logarithmic number of scales.

\smallskip

The squared logarithm in~\eqref{e.exp.log.squared} comes from two costs in~\cite{AK.HC}. At fixed accuracy and stochastic parameters, each improvement stage occupies a number of scales logarithmic in the reference ellipticity ratio. On these scales, the proof renormalizes the coarse ellipticity data, controls fluctuations and changes of the mean in a suitable geometry, and compares the primal and dual optimizer energies. The stage gives either an improvement of the contrast or a definite decrease of the determinant, and logarithmically many such stages suffice. Multiplying the number of stages by their length gives the squared logarithm.

As in~\cite{AK.HC}, our method is based on the variational coarse-grained matrices in adapted cubes but here we carry the fluctuation and mean estimates of these quantities forward as the scale and geometry change. The initial reduction of these errors is made once. Thereafter, each step advances a number of scales independent of the reference ellipticity ratio~$\Pi$, with the accuracy fixed. A single count includes all these steps, even those on which the errors grow and have to be reduced again. Once the required estimates and geometry have been found, the final energy comparison closes the gap. The point is to avoid repeating the logarithmic preparation at every improvement stage.

\paragraph{Averaging.}
Consider a large cube partitioned into smaller cubes on a fixed grid. Each smaller-cube matrix depends only on the coefficients there. Matrices from sufficiently separated cubes are therefore independent, so their average fluctuates less than an individual matrix. The large-cube matrix need not equal this average, but subadditivity places it below the average. Their difference is positive semidefinite, and its expectation is exactly the decrease of the mean matrix between the two scales. This sign, together with the moment bound on the average, controls the discrepancy between the large-cube matrix and the average. Averaging therefore reduces fluctuations up to an error measured by the decrease of the mean.

We record that decrease through the loss of the logarithm of the determinant, called the \emph{determinant loss}. On a fixed grid the mean matrices are ordered, so a small determinant loss makes them close in relative terms. It also controls the cost of changing the mean matrix used to normalize the fluctuations. The losses add across scales, and their total on the Euclidean grid is at most~$C\log(2+\Pi)$. Starting from bounds polynomial in~$\Pi$, the fixed-grid estimates therefore reduce the fluctuations and recent mean changes to a fixed small tolerance within a logarithmic number of scales. This initializes the iteration; the duality gap may still be large.

We must retain more than the estimates at the current scale. A change of grid cuts across cubes of many sizes, and the final energy comparison uses averages of gradients and fluxes on smaller cubes. We keep the earlier fluctuation and mean estimates with decreasing weights. As the scale increases, these weights reduce the contribution from older scales, while averaging and mean comparisons control the new contributions. Thus fluctuations are estimated at every step, using the bounds already obtained, rather than only on the steps at which contraction succeeds. Section~\ref{s.fixed.geometry} develops these estimates.

\paragraph{Changing the grid.}
In an anisotropic medium, gradients of the same Euclidean size can have very different energy costs in different directions. Estimates on Euclidean cubes can consequently have large constants. Even for a constant-coefficient equation, regularity estimates will be sub-optimal if the ``wrong geometry'' is used---for instance, imagine iterating estimates for the Laplacian in very eccentric ellipsoids instead of balls. Therefore, as in~\cite{AK.HC}, we deform the Euclidean cubes by an affine change of variables in order minimize these errors and to balance the gradient and flux estimates. The appropriate stretching is determined by the annealed matrix at the reference scale. Since this matrix changes with scale, cubes suited to an earlier mean need not be suited to a later one, and their shape must be adjusted during the iteration.

\smallskip

We change the shape gradually so that the estimates on the old grid can be used on the new one. For nearby grids, most of a new cube can be filled by large old-grid cubes, with successively smaller cubes near its boundary. Their volume fractions decrease fast enough to compensate for the deterioration of coarse ellipticity at small scales, because the coarse ellipticity exponent in assumption~\ref{a.coarse.ellipticity} satisfies~$\gamma<1$. The retained estimates control the contributions from all these sizes. Changing the grid may increase the errors, so we make the change only when the incoming errors and the variation of the mean are small enough to leave a bounded error on the new grid. A fixed number of startup scales then makes the fixed-grid contraction estimate available again. Section~\ref{s.geometry.transport} proves the required comparisons. Neither the change of grid nor startup creates smallness; both preserve enough of the preceding control to continue the iteration.

\paragraph{Counting scales.}
We count both reductions in the errors and improvements in the grid geometry. For the errors, we use the logarithmic size of the error plus drift: an increase by a fixed factor then has a bounded cost. For the geometry, we measure the distance from the current grid to the cube shape that balances gradient and flux estimates for the annealed coarse-grained matrix at the grid's initialization scale. This reference shape stays fixed while we advance on the same grid. Averaging reduces the logarithmic error up to determinant loss. A partial geometry update brings the grid closer to the preferred shape for a more recent mean, and this progress compensates for the possible error increase during the update and startup. Determinant loss controls the difference between the old and new preferred shapes. By combining the logarithmic error and the geometric distance with suitable weights, we can therefore count both kinds of progress. If the determinant tests fail, the errors remain controlled by the fixed-grid estimates, and the definite determinant decrease is included in the same count.

\smallskip

The determinant is monotone only while the grid is fixed. Comparisons between nearby grids bound its possible upward jumps; the progress in adapting the shape also compensates for these jumps. The overlapping determinant losses count each decrease only a bounded number of times. After summing, the bound on the number of steps involves only the initial logarithmic error, the initial mismatch of shapes, and the initial logarithmic determinant, each at most~$C\log(2+\Pi)$. This count includes steps on which the errors grow and the further averaging needed to reduce them again. It forces the iteration to find the required geometry and small errors, together with small mean variation over the subsequent interval, within logarithmically many steps. Proposition~\ref{p.global.selection} proves this by iterating Proposition~\ref{p.scale.selection}. We do not begin a new logarithmic count each time the grid changes.

\paragraph{Closing the gap.}
On the selected interval, the small change of the mean makes an optimizer on the large cube close in energy to the optimizers on its smaller cubes. The fluctuation estimates control averages of their centered gradients and fluxes. The equation converts this control of averages into a bound on the centered optimizer energies that determine the duality gap. The remaining comparison errors are controlled by the change of the mean and by the ratio of the small-cube size to the large-cube size, which decreases as the selected interval becomes longer. The adapted cube shapes and the choice of primal and dual data balance the gradient and flux estimates, keeping the constants independent of the anisotropy. This gives a bound linear in the unknown contrast. Comparing the means at the two ends of the interval then bounds the terminal duality gap by a small multiple of one plus that same gap. Rearranging makes the gap small, even if it was initially large. Proposition~\ref{p.response.transfer} gives this argument; its required accuracy and interval length are chosen before the selection begins.

Once the gap is small, only a comparison of annealed matrices remains: we transfer the bound from the adapted cubes to larger Euclidean cubes, without transporting fluctuation estimates again. The geometric eccentricity of the selected grid is polynomially bounded in~$\Pi$, so this comparison costs another logarithmic number of scales. The relative primal--dual bound then persists at all larger Euclidean scales, since the annealed coarse-grained matrices decrease and their duals increase.

The tail bound~\eqref{e.source.tail} for~$\S$ determines the initial scale at which the coarse ellipticity bounds have the moments needed throughout the construction. Including this initial scale, the selection and final comparison require at most~$C\log_3(2+\Pi K_{\Psi_{\S}})$ scales. Since scale~$m$ corresponds to length~$3^m$, this gives the polynomial length bound in Theorem~\ref{t.polynomial.entry}. Section~\ref{s.convergence.homogenization} then starts from this fixed small-contrast threshold and applies the distinct small-contrast iteration of~\cite{AK.HC} to obtain the decay rate in Theorem~\ref{t.algebraic.convergence}. The deterministic homogenization estimates of~\cite{AK.HC} yield Theorems~\ref{t.uniform.homogenization} and~\ref{t.random.homogenization}.

\subsection{Lean Formalization}
\label{ss.lean.formalization}

A formalization of this paper in the Lean~4 proof assistant is available at
\begin{center}
\url{https://github.com/scottnarmstrong/HighContrastHomogenization}\,.
\end{center}
It includes Theorems~\ref{t.polynomial.entry} and~\ref{t.algebraic.convergence} and the supporting lemmas and propositions. The formal Dirichlet estimate corresponding to Theorem~\ref{t.uniform.homogenization} treats~$f=0$ and uses the negative-Sobolev formulation of Theorem~\ref{t.random.homogenization}; the formal version of the latter is stated for cubes in adapted coordinates rather than general bounded Lipschitz domains. The formal proofs follow the arguments given here and build on \texttt{mathlib} and the \texttt{CoarseGraining} library. They contain no unproved placeholders or custom axioms. The main formal statements are also restated using only \texttt{mathlib} definitions, with separate Lean-checked proofs connecting them to the development.

The formalization was written by AI agents under the close supervision of the authors. The development was orchestrated by Claude Fable 5.1, with lean code written by subagents - primarily Opus 5 and GPT 5.6-Sol (low effort).

\subsection{Use of artificial intelligence}
\label{ss.AI}

The authors developed the main conceptual ideas of this work and directed its mathematical development. Throughout the project, we used large language models as research assistants, guiding them toward specific questions arising in the proof.

\smallskip

Our work on this project goes back several years. We initially pursued an approach based on a cluster expansion, using the BKAR forest formula. Attempts to formalize this approach in Lean brought several difficulties into focus and helped clarify the obstacles to completing the argument. The resolution of these difficulties allowed us to remove the use of cluster expansions completely, leading to the present method which is actually quite close to~\cite{AK.HC} at a high level. The authors believed, from the beginning of the project, that control of the fluctuations of the quantities should be separated from control of the annealed coarse-graining defect. Eventually, with the help of GPT~5.6-Sol, we found an iteration scheme, more intricate than the one of~\cite{AK.HC}, which was based on this idea. The subsequent Lean formalization, together with further work with LLMs, helped us refine and simplify the proof. The authors take responsibility for the mathematical content and presentation of the paper.

\section{Scale selection}
\label{s.scale.selection}

The proof of Theorem~\ref{t.polynomial.entry} is based on an iteration over scales. The iteration controls fluctuations and changes of the mean while adapting the grid, until we can compare the primal and dual optimizer energies and conclude that the duality gap is small. This section sets up the objects the iteration tracks and states the proposition describing one of its steps.

\paragraph{Errors and geometry.}
We keep track of both the random fluctuations of the coarse-grained matrices and the changes in their means as the scale increases. Estimates from earlier scales are retained with decreasing weights, because changing the grid and closing the duality gap both require information from smaller cubes. These bounds are organized into the error~$\mathcal P_{\qq}$ and the drift of the means. On a fixed grid, their sum contracts, apart from terms controlled by the decrease of the logarithmic determinant of the mean matrix. To keep the estimates effective in an anisotropic medium, we also adapt the shapes of the cubes to the mean matrix, balancing the primal and dual energy estimates as in~\cite{AK.HC}. We compare shapes independently of their overall size and change them gradually, so that the estimates already obtained can be transferred to the new grid. Rounding the defining matrices ensures compatibility with the integer translations under which the law is stationary: we snap the grid to the lattice.

\paragraph{Iteration.}
Proposition~\ref{p.scale.selection}, stated in Subsection~\ref{ss.scale.selection}, gives the rule for continuing this procedure. On a newly adopted grid, we first advance a few scales while keeping its shape fixed. This startup carries the incoming error bound to a scale where contraction becomes available; it need not reduce the error, and its cost is controlled by the incoming bound and the change of the mean. After startup, we keep averaging on the same grid while the errors are too large to permit a change of shape. This gives a contraction estimate with an additional error measured by determinant loss. When the errors are sufficiently small and the mean varies little, we can instead move toward a shape better adapted to the mean matrix, carrying the earlier estimates with us. If the mean changes too much for either of these steps, we still advance the scale on the same grid and record a definite determinant loss. Such a loss does not stop the procedure: its accumulated size will bound how often these steps can occur. Throughout, the grid's geometric eccentricity remains controlled, and each step preserves the initialization conditions~\eqref{e.renormalization.input} needed to take the next one.

The small-cube estimates throughout this procedure use a single random variable, supplied by the adapted cube bound~\eqref{e.source.adapted.bound}, even as the grid changes. We prove Proposition~\ref{p.scale.selection} in Subsection~\ref{s.scale.selection.proof} and carry out the iteration in Section~\ref{s.polynomial.entry.proof}.

\subsection{Notation}

Assumptions~\ref{a.stationarity}--\ref{a.coarse.ellipticity} are in force throughout Sections~\ref{s.scale.selection}--\ref{s.polynomial.entry.proof}. We use the stationary extension~$\S(z;\a)\coloneqq\S(T_z\a)$ for~$z\in\Zd$, suppressing~$\a$ and writing~$\S(z)$, so that~$\S(0)=\S$; by~\ref{a.stationarity}, the condition~\eqref{e.coarse.ellipticity} holds almost surely at every~$z\in\Zd$ with~$\S(z)$ in place of~$\S$ and~$z+y+\cu_k$ in place of~$y+\cu_k$.

\paragraph{Adapted cubes.}
The law is stationary under integer translations, whereas the geometry selected by a coarse block need not generate an integer grid. We therefore round the geometry. Fix~$j_*\in\N$ with~$3^{j_*}\geq2d$. We write~$\Rpos$ for the set of symmetric positive definite~$d\times d$ matrices. For every~$\m\in\Rpos$, define
\begin{equation*}
	\bigl(\mathcal Q(\m)\bigr)_{ab}
	\coloneqq3^{-j_*}\bigl\lceil3^{j_*}|\m^{-1}|^{\nf12}(\m^{\nf12})_{ab}\bigr\rceil
	\qquad(1\leq a,b\leq d)\,.
\end{equation*}
The matrix~$\qq=\mathcal Q(\m)$ is symmetric, and the rounding error has operator norm at most~$d3^{-j_*}\leq\frac12$. The unrounded matrix has smallest eigenvalue one and largest eigenvalue~$(|\m|\,|\m^{-1}|)^{\nf12}$, so
\begin{equation}
	\qq\geq\frac12\Id\,,\quad |\qq^{-1}|\leq2\qand |\qq|\leq2\bigl(|\m|\,|\m^{-1}|\bigr)^{\nf12}\,.
	\label{e.rounded.grid.bounds}
\end{equation}
Also,~$3^j\qq\Zd\subseteq\Zd$ whenever~$j\geq j_*$.

\paragraph{Normalized blocks.}
We measure the means and fluctuations relative to the annealed block at a later scale. For~$j\in\Z$, set
\begin{equation*}
	\cus_j^{\qq}\coloneqq\qq\cu_j
	\qquad\text{and}\qquad
	\bfAhom_{j,\qq}\coloneqq\bfAhom(\cus_j^{\qq})=\E\bigl[\bfA(\cus_j^{\qq})\bigr]\,.
\end{equation*}
For~$j\leq k$ and~$z\in3^j\qq\Zd$, define the normalized mean and fluctuation by
\begin{equation}
	P_{j,k}^{\qq}
	\coloneqq\bfAhom_{k,\qq}^{-\nf12}\bfAhom_{j,\qq}\bfAhom_{k,\qq}^{-\nf12}
	\qquad\text{and}\qquad
	V_{j,k}^{\qq}(z)
	\coloneqq\bfAhom_{k,\qq}^{-\nf12}
	\bigl(\bfA(z+\cus_j^{\qq})-\bfAhom_{j,\qq}\bigr)
	\bfAhom_{k,\qq}^{-\nf12}\,.
	\label{e.scale.selection.normalized.mean.fluctuation}
\end{equation}
Then
\begin{equation*}
	\bfAhom_{k,\qq}^{-\nf12}\bfA(z+\cus_j^{\qq})\bfAhom_{k,\qq}^{-\nf12}
	=P_{j,k}^{\qq}+V_{j,k}^{\qq}(z)\,.
\end{equation*}
For~$j\geq j_*$, stationarity makes~$V_{j,k}^{\qq}(z)$ centered at every grid point. The decomposition thus separates random fluctuations from the deterministic change of the mean. Write~$V_j^{\qq}\coloneqq V_{j,j}^{\qq}(0)$.

At these scales, aligned subdivision and stationarity make the annealed blocks decrease with the scale; we measure this decrease by the determinant loss. For~$j\leq k$, set
\begin{equation}
	\Delta_{j,k}^{\qq}\coloneqq\log\det\bfAhom_{j,\qq}-\log\det\bfAhom_{k,\qq}\geq0\,.
	\label{e.scale.selection.logdet.loss}
\end{equation}
For an integer~$h\geq1$, set
\begin{equation}
	\widehat\Delta_h^{\qq}(m)\coloneqq\sum_{a=m+1}^{m+h}\Delta_{a-h,a}^{\qq}\,.
	\label{e.scale.selection.synchronized.loss}
\end{equation}
Then
\begin{equation}
	\Delta_{m,m+h}^{\qq}\leq\widehat\Delta_h^{\qq}(m)\,.
	\label{e.scale.selection.synchronized.loss.lower}
\end{equation}
\paragraph{The error.}
A change of grid cuts across cubes at many scales. Controlling only the current scale would therefore lose the information needed after an update. We retain the earlier errors with decaying weights. Using the coarse ellipticity exponent~$\gamma$ from assumption~\ref{a.coarse.ellipticity}, choose the even moment exponent~$Q$ and the decay exponent~$\rho_{\max}$ by
\begin{equation}
	Q\coloneqq2\biggl\lceil\frac{2(d+1)}{1-\gamma}\biggr\rceil
	\qquad\text{and}\qquad
	\rho_{\max}\coloneqq\gamma+\frac1Q\biggl(d+\frac14(1-\gamma)\biggr)<1\,.
	\label{e.scale.selection.Q.choice}
\end{equation}
We write~$|M|$ for the operator norm of a matrix~$M$. For~$N\geq1$ and a deterministic symmetric~$2d \times 2d$ matrix~$M$ with eigenvalues~$\lambda_1(M),\ldots,\lambda_{2d}(M)$, set
\begin{equation*}
	|M|_{S_N}\coloneqq\biggl(\sum_{j=1}^{2d}|\lambda_j(M)|^N\biggr)^{\nf1N}\,.
\end{equation*}
For a random symmetric~$2d \times 2d$ matrix~$H$, set
\begin{equation*}
	\|H\|_{L^N(S_N)}\coloneqq\bigl(\E[|H|_{S_N}^N]\bigr)^{\nf1N}\,.
\end{equation*}
For~$P\geq\Itwod$, set
\begin{equation*}
	\meanpenalty_Q(P)\coloneqq\bigl(1+\tr(P-\Itwod)\bigr)^Q-1\,.
\end{equation*}
For~$m\geq j_*$, define the fluctuation history by
\begin{equation}
	\history_{\qq}^{\rm fluc}(m)
	\coloneqq\E\Biggl[
	\sup_{j_*\leq j\leq m}3^{-Q\rho_{\max}(m-j)}
	\max_{z\in3^j\qq\Zd\cap\cus_m^{\qq}}
	\bigl|V_{j,m}^{\qq}(z)\bigr|^Q
	\Biggr]\,.
	\label{e.scale.selection.fluctuation.history}
\end{equation}
For~$j_*\leq n\leq m$, define the mean history by
\begin{equation}
	\history_{\qq}^{\rm mean}(m;n)
	\coloneqq\sum_{j=n}^{m-1}
	3^{-\frac14(1-\gamma)(m-1-j)}
	\meanpenalty_Q(P_{j,m}^{\qq})\,.
	\label{e.scale.selection.mean.history}
\end{equation}
Set~$\history_{\qq}^{\rm mean}(m)\coloneqq\history_{\qq}^{\rm mean}(m;j_*)$ and~$\history_{\qq}(m)\coloneqq\history_{\qq}^{\rm fluc}(m)+\history_{\qq}^{\rm mean}(m)$. We also track the successive increments of the mean, rather than their cumulative difference from the final mean, through the drift
\begin{equation}
	D_{\qq,j_*}(m)
	\coloneqq\sum_{j=j_*+1}^{m}
	3^{-\frac18(1-\gamma)(m-j)}
	\tr\bigl(P_{j-1,m}^{\qq}-P_{j,m}^{\qq}\bigr)\,.
	\label{e.scale.selection.determinant.drift}
\end{equation}

The determinant loss~$\Delta_{j,k}^{\qq}$ adds exactly across scale intervals. The mean history compares each earlier normalized mean~$P_{j,m}^{\qq}$ with~$\Itwod$, using the nonlinear penalty~$\meanpenalty_Q$ needed for the fluctuation estimates. The drift instead sums the one-scale decreases~$\tr(P_{j-1,m}^{\qq}-P_{j,m}^{\qq})$ linearly, with weights that decay more slowly than those of the mean history. Advancing the current scale discounts the old increments and adds the new ones; determinant loss controls the change of normalization. We can therefore estimate the new drift directly in terms of the old drift and the intervening determinant loss.

For a change of grid, the useful quantities are instead the cumulative differences~$\tr(P_{j,m}^{\qq}-\Itwod)$ between an earlier normalized mean and the current one. Each such difference is the sum of its intervening increments. Summation by parts expresses the drift as a weighted sum of these differences; see~\eqref{e.two.grid.drift.abel}. We can then estimate its terms using Whitney partitions to fill new-grid cubes with old-grid cubes and compare the corresponding coarse-grained matrices.

To restart the estimates at a scale~$n$, we carry a weighted copy of the full history up to~$n$ and add the mean and fluctuation errors from the subsequent scales. For~$j_*\leq n\leq m$, set
\begin{align}
	\mathcal P_{\qq}(m;n) &
	\coloneqq3^{-\frac14(1-\gamma)(m-n)}
	\bigl(1+\meanpenalty_Q(P_{n,m}^{\qq})\bigr)
	\history_{\qq}(n)
	+\history_{\qq}^{\rm mean}(m;n)
	\notag \\ & \qquad
	+\sum_{j=n+1}^{m}
	3^{-\frac14(1-\gamma)(m-j)}
	e^{Q\Delta_{j,m}^{\qq}}\E\bigl[|V_j^{\qq}|_{S_Q}^Q\bigr]\,.
	\label{e.scale.selection.complete.profile}
\end{align}
In particular,~$\mathcal P_{\qq}(n;n)=\history_{\qq}(n)$.

\paragraph{Canonical geometry.}
To choose a new grid, we associate a metric with each positive coarse block and compare metrics up to positive rescaling. Set
\begin{equation*}
	\mathbf R\coloneqq\begin{pmatrix}0&\Id\\ \Id&0\end{pmatrix}\,.
\end{equation*}
For~$\m_0,\m_1\in\Rpos$, define
\begin{equation}
	d_{\rm pr}([\m_0],[\m_1])
	\coloneqq\frac12\log
	\frac{\lambda_{\max}(\m_0^{-\nf12}\m_1\m_0^{-\nf12})}
	{\lambda_{\min}(\m_0^{-\nf12}\m_1\m_0^{-\nf12})}\,.
	\label{e.scale.selection.projective.distance}
\end{equation}
This distance measures the anisotropy of one metric relative to the other; an overall dilation has no effect. It satisfies the triangle inequality. Indeed, if~$a_{ij}\m_i\leq\m_j\leq b_{ij}\m_i$ with the sharp constants, then~$a_{01}a_{12}\m_0\leq\m_2\leq b_{01}b_{12}\m_0$. Taking half the logarithm of the ratio of the upper and lower constants gives~$d_{\rm pr}([\m_0],[\m_2])\leq d_{\rm pr}([\m_0],[\m_1])+d_{\rm pr}([\m_1],[\m_2])$.

\smallskip

For a positive~$2d \times 2d$ block~$F$, define
\begin{equation}
	\cmet(F)\coloneqq\Bigl(F^{\nf12}\bigl(F^{-\nf12}\mathbf RF^{-1}\mathbf RF^{-\nf12}\bigr)^{\nf12}F^{\nf12}\Bigr)_{22}^{-1}>0\,.
	\label{e.scale.selection.canonical.metric}
\end{equation}
The matrix inside the outer parentheses in the previous display can also be written as
\begin{equation*}
M(F) \coloneqq F\#(\mathbf RF^{-1}\mathbf R)\,,
\end{equation*}
where the geometric mean~$A\#B$ of two positive matrices~$A$ and~$B$ is defined by 
\begin{equation*}
	A\#B=A^{\nf12}(A^{-\nf12}BA^{-\nf12})^{\nf12}A^{\nf12} \,.
\end{equation*}
It is the unique positive solution of~$XA^{-1}X=B$, since conjugation by~$A^{-\nf12}$ reduces this equation to the positive square-root equation. This characterization gives symmetry in~$A,B$, covariance under congruence, and~$(aA)\#(bB)=\sqrt{ab}(A\#B)$ for~$a,b>0$. It also gives~$(A\#B)^{-1}=A^{-1}\#B^{-1}$ by inversion. In particular,~$M(cF)=M(F)$ for~$c>0$.

The mean~$M(F)$ is self-dual, meaning~$\mathbf RM(F)^{-1}\mathbf R=M(F)$. Its lower-right block is~$\cmet(F)^{-1}$, with Schur complement~$\cmet(F)$, as verified in the proof of Proposition~\ref{p.response.transfer}. This common metric balances the gradient and flux estimates for the primal and dual optimizers.

\smallskip

The mean is increasing in both arguments. To see this without any commutativity assumption, spectral calculus gives
\begin{equation*}
	H^{\nf12}=\frac1\pi\int_0^\infty r^{-\nf12}H(rI+H)^{-1}\,dr\qquad(H>0)\,.
\end{equation*}
Inversion reverses order, so~$H(rI+H)^{-1}=I-r(rI+H)^{-1}$ increases with~$H$. Integration proves monotonicity of the square root. The defining formula for~$A\#B$ therefore makes it increasing in~$B$, and symmetry gives the same conclusion for~$A$. In particular,
\begin{equation}
	\mathbf RF^{-1}\mathbf R\leq F
	\quad\Longrightarrow\quad
	\mathbf RF^{-1}\mathbf R\leq M(F)\leq F\,.
	\label{e.matrix.canonical.order}
\end{equation}

\subsection{Scale selection}
\label{ss.scale.selection}

To control all adapted cubes in~$\cu_{2j_*}$ with one random variable, define
\begin{equation}
	\RSZ\coloneqq\min\Bigl\{3^{\gamma r}:r\in\Z_{\geq0},\ \max_{z\in3^{j_*+r}\Zd\cap\cu_{2j_*+r}}\S(z)\leq3^{j_*+r}\Bigr\}\,.
	\label{e.source.multiplier}
\end{equation}
For each~$r$, the centers~$z$ in the maximum give the partition of~$\cu_{2j_*+r}$ into its~$3^{dj_*}$ standard subcubes~$z+\cu_{j_*+r}$. The minimum is attained almost surely, as proved below.

\smallskip

There is a constant~$C_{\rm src}(d,\gamma)<\infty$ such that, whenever
\begin{equation}
	j_*\geq\bigl\lceil C_{\rm src}(d,\gamma)\log_3(2K_{\Psi_{\S}})\bigr\rceil\,,
	\label{e.source.lower.scale}
\end{equation}
the random variable~$\RSZ$ satisfies~$\|\RSZ\|_{L^Q}\leq2$ and
\begin{equation}
	\bfA(z+\cus_j^{\qq})\leq C(d,\gamma)\bigl(|\m|\,|\m^{-1}|\bigr)^{\nf12}\RSZ3^{\gamma(j_*-j)_+}\bfE\qquad\bigl(z+\cus_j^{\qq}\subseteq\cu_{2j_*}\bigr)\,.
	\label{e.source.adapted.bound}
\end{equation}
The same~$\RSZ$ works for every adapted cube in this region, including scales below~$j_*$. For standard aligned cubes the factor~$C(d,\gamma)(|\m|\,|\m^{-1}|)^{\nf12}$ can be omitted. We prove these facts in Section~\ref{ss.source.estimate}, below. From now on~$j_*$ satisfies~\eqref{e.source.lower.scale}; the global argument will choose it large enough to contain all the cubes used in the iteration.

Fix the grid~$\qq=\mathcal Q(\m)$, and let~$k$ be the scale at which its error was initialized. The contraction estimate compares coarse-grained matrices on cubes separated by~$h$ scales, all later than~$k$. We therefore first advance~$h$ scales on the same grid. This startup step begins at the current scale~$n=k$ and requires~$\mathcal P_{\qq}(n;n)+D_{\qq,j_*}(n)\leq1$. Subsequent steps in the proposition start at~$n\geq k+h$, but do not require the error plus drift to remain at most one.

\begin{proposition}[Scale-selection alternatives]
\label{p.scale.selection}
There exist~$h(d,\gamma)\in\N$ and~$\varepsilon_0(d,\gamma),c(d,\gamma)\in(0,1)$ such that, for every~$\varepsilon\in(0,\varepsilon_0]$ and~$\sigma\in(0,\varepsilon]$, there exist~$L(\varepsilon,\sigma,d,\gamma)\in\N$ and~$B_0(\varepsilon,\sigma,d,\gamma)\geq1$ such that the following holds for every~$B\geq B_0$. Set~$\eta\coloneqq c\varepsilon\sigma$. Let
\begin{equation*}
	\m\in\Rpos\,,
	\qquad
	\qq=\mathcal Q(\m)\,,
	\qquad
	k,n\in\Z\,,
	\qquad
	j_*\leq k\leq n\,,
\end{equation*}
and assume
\begin{equation}
	\frac12\log\bigl(|\m|\,|\m^{-1}|\bigr)
	\leq\frac{\varepsilon}{L}
	\big(k-j_*-\left\lceil B\log_3(2+\Pi)\right\rceil\big)\,.
	\label{e.renormalization.eccentricity}
\end{equation}
Assume also one of the following two conditions:
\begin{equation}
\left\{ 
	\begin{aligned}
	&k=n\quad\text{and}\quad\mathcal P_{\qq}(n;n)+D_{\qq,j_*}(n)\leq1\,, \quad \mbox{or}\\
	&n\geq k+h\,.
	\end{aligned}
\right.
	\label{e.renormalization.input}
\end{equation}
For the following geometry update, write~$\m_*\coloneqq \cmet\left(\bfAhom_{n+2L,\qq}\right)$, 
and, with the convention~$\theta\coloneqq1$ when~$[\m]=[\m_*]$, set
\begin{equation}
	\m_+\coloneqq
	\m^{\nf12}\bigl(\m^{-\nf12}\m_*\m^{-\nf12}\bigr)^{\theta}\m^{\nf12}
	\,, \qquad \theta \coloneqq \min\Big\{ 1, \frac{\varepsilon}{d_{\rm pr}([\m],[\m_*])}\Big\}\,.
	\label{e.renormalization.geometry.update}
\end{equation}
Set~$\qq_+\coloneqq\mathcal Q(\m_+)$. Assume also that
\begin{equation}
	\cus_{n+2L}^{\qq}\cup\cus_{n+L}^{\qq_+}\subseteq\cu_{2j_*}\,.
	\label{e.renormalization.containment}
\end{equation}
Exactly one of the following five cases occurs:
\begin{itemize}
\item \emph{Case 1 (Startup).}
\begin{equation*}
k=n.
\end{equation*}

\item \emph{Case 2 (Contraction).}
\begin{equation*}
k<n, \quad 
\mathcal P_{\qq}(n;k)+D_{\qq,j_*}(n)>\eta,
\quad \mbox{and} \quad 
d^{-1}\widehat\Delta_h^{\qq}(n)\leq\sigma.
\end{equation*}

\item \emph{Case 3 (Geometry update).}
\begin{equation*}
k<n, \quad 
\mathcal P_{\qq}(n;k)+D_{\qq,j_*}(n)\leq\eta,
\quad \mbox{and} \quad 
d^{-1}\Delta_{n,n+2L}^{\qq}\leq\varepsilon\sigma.
\end{equation*}

\item \emph{Case 4 (Determinant obstruction while seeking contraction).}
\begin{equation*}
k<n, \quad \mathcal P_{\qq}(n;k)+D_{\qq,j_*}(n)>\eta, \quad \mbox{and} \quad d^{-1}\widehat\Delta_h^{\qq}(n)>\sigma.
\end{equation*}
\item \emph{Case 5 (Determinant obstruction while seeking geometry update).}
\begin{equation*}
k < n, \quad \mathcal P_{\qq}(n;k)+D_{\qq,j_*}(n)\leq\eta, \quad \mbox{and} \quad 
d^{-1}\Delta_{n,n+2L}^{\qq}>\varepsilon\sigma
\end{equation*}
\end{itemize}
In the startup case, we have 
\begin{equation}
	\mathcal P_{\qq}(n+h;k)+D_{\qq,j_*}(n+h)
	\leq C(d,\gamma)
	\left(
	\mathcal P_{\qq}(n;k)+D_{\qq,j_*}(n)
	+e^{Q\Delta_{n,n+h}^{\qq}}-1
	\right)\,.
	\label{e.renormalization.startup}
\end{equation}
In the contraction case, we have
\begin{equation}
	\mathcal P_{\qq}(n+h;k)+D_{\qq,j_*}(n+h)
	\leq\frac14\bigl(\mathcal P_{\qq}(n;k)+D_{\qq,j_*}(n)\bigr)
	+C(d,\gamma)\widehat\Delta_h^{\qq}(n)\,.
	\label{e.renormalization.service}
\end{equation}
In the geometry update case, we have
\begin{equation}
	\mathcal P_{\qq_+}(n+L;n+L)+D_{\qq_+,j_*}(n+L)
	\leq C(d,\gamma)\sigma^{\frac18(1-\gamma)}\,.
	\label{e.renormalization.transport.output}
\end{equation}
Moreover, if we define
\begin{equation}
	(\m',\qq',k',n')\coloneqq
	\begin{cases}
		(\m,\qq,k,n+h)&\text{in cases 1, 2, and 4,}\\
		(\m_+,\qq_+,n+L,n+L)&\text{in case 3,}\\
		(\m,\qq,k,n+2L)&\text{in case 5,}
	\end{cases}
	\label{e.renormalization.output.data}
\end{equation}
then we have 
\begin{equation}
	\frac12\log\bigl(|\m'|\,|(\m')^{-1}|\bigr)
	\leq\frac{\varepsilon}{L}
	\left(k'-j_*-\left\lceil B\log_3(2+\Pi)\right\rceil\right)\,,
	\label{e.renormalization.eccentricity.output}
\end{equation}
and
\begin{equation}
\left\{ 
	\begin{aligned}
	&k'=n'\quad\text{and}\quad\mathcal P_{\qq'}(n';n')+D_{\qq',j_*}(n')\leq1\,, \quad \mbox{or}\\
	&n'\geq k'+h\,.
	\end{aligned}
\right.
	\label{e.renormalization.output}
\end{equation}
\end{proposition}

\paragraph{The five cases.}
Case~1 prepares a newly initialized error for the contraction estimate. Starting from~$n=k$, we advance to~$n+h$ without changing~$\m$,~$\qq$, or~$k$. The startup bound~\eqref{e.renormalization.startup} controls the error at this later scale, allowing growth from the determinant loss; it does not assert contraction.

For~$k<n$, we either continue reducing the error on the current grid or use its smallness to change geometry. If~$\mathcal P_{\qq}(n;k)+D_{\qq,j_*}(n)>\eta$, we seek contraction. The relevant determinant quantity is~$\widehat\Delta_h^{\qq}(n)=\sum_{a=n+1}^{n+h}\Delta_{a-h,a}^{\qq}$: it collects the losses in the parent--child comparisons used to reach the next~$h$ scales. When~$\widehat\Delta_h^{\qq}(n)\leq\sigma d$, Case~2 advances to~$n+h$, keeping~$\m$,~$\qq$, and~$k$ fixed. Estimate~\eqref{e.renormalization.service} reduces the error plus drift by a factor~$\nf14$, up to an additive term controlled by this determinant loss.

If~$\mathcal P_{\qq}(n;k)+D_{\qq,j_*}(n)\leq\eta$, the error is small enough to attempt a geometry update. Here the comparisons between grids use the next~$2L$ scales, so we require~$\Delta_{n,n+2L}^{\qq}\leq d\varepsilon\sigma$. Under this condition, Case~3 moves~$\m$ toward~$\m_*=\cmet(\bfAhom_{n+2L,\qq})$ and transfers the error to~$\qq_+$. We restart at the intermediate scale~$k'=n'=n+L$. The transferred bound~\eqref{e.renormalization.transport.output} is at most one, though it need not remain below~$\eta$, and thus permits another startup step. The error is reinitialized even if rounding leaves~$\qq_+=\qq$.

Cases~4 and~5 apply when the respective determinant tests fail. We retain~$\m$,~$\qq$, and~$k$ and advance by~$h$ or~$2L$, respectively. No contraction or geometry improvement is asserted: the error may grow, but the fixed-grid estimates control this growth in terms of the determinant loss, whose accumulated contribution is bounded in the global argument. Equality in the tests belongs to Case~2 or Case~3, so the five cases are disjoint and exhaustive. In every case, the resulting grid and scales satisfy the eccentricity bound~\eqref{e.renormalization.eccentricity.output} and initialization conditions~\eqref{e.renormalization.output}. Before applying the proposition again, we must also check that the cubes used in the next step lie inside~$\cu_{2j_*}$.

\paragraph{Bounded geometry changes.}
We approach~$\m_*$ gradually because a large change in the shapes of the cubes makes it costly to compare the two grids. The Whitney partitions, the comparison of annealed matrices, and the transfer of errors in Section~\ref{s.geometry.transport} all depend on the distortion between the grids. The first target can be a projective distance of order~$\log(2+\Pi)$ from the Euclidean metric. Moving there in one step could therefore give a distortion polynomial in~$\Pi$, requiring a scale separation of order~$\log(2+\Pi)$ to control the boundary layers. Instead, we impose~$d_{\rm pr}([\m],[\m_+])\leq\varepsilon$; Lemma~\ref{l.projective.step} then bounds the distortion independently of~$\Pi$.

The interpolation in~\eqref{e.renormalization.geometry.update} makes this bounded move explicit. It follows the affine-invariant geodesic from~$\m$ to~$\m_*$: raising the eigenvalues of~$\m^{-\nf12}\m_*\m^{-\nf12}$ to the power~$\theta$ multiplies their logarithmic spread by~$\theta$. Hence~$d_{\rm pr}([\m],[\m_+])=\theta\,d_{\rm pr}([\m],[\m_*])$. The chosen~$\theta$ moves the metric exactly~$\varepsilon$ toward~$\m_*$, unless the target is already within that distance, in which case~$\m_+=\m_*$.

The logarithmic grid eccentricity increases by at most~$\varepsilon$, while the initialization scale advances by at least~$L$, from~$k$ to~$n+L$. Thus~\eqref{e.renormalization.eccentricity} remains valid. Unless the update reaches~$\m_*$, the distance to this target decreases by~$\varepsilon$; the global argument uses this decrease to compensate for the possible growth of the error when it is reinitialized and passes through startup again.

The proof of the proposition is given in Subsection~\ref{s.scale.selection.proof}, after Section~\ref{s.fixed.geometry} establishes the fixed-grid estimates and Section~\ref{s.geometry.transport} proves the estimates for changing geometry.

\section{Fixed-grid estimates}
\label{s.fixed.geometry}

In this section, we prove the fixed-grid estimates needed for Proposition~\ref{p.scale.selection}. These give contraction up to determinant loss, together with bounds for startup and for the advance preceding a change of geometry.

\paragraph{Averaging and mean changes.}
We fix an adapted grid~$\qq=\mathcal Q(\m)$ and ask why passing to larger cubes should reduce the errors. A large cube contains many smaller cubes, which we call its children. The coarse matrix of each child depends only on the coefficients inside it. Although neighboring children need not be independent, they can be grouped into a fixed number of separated families within which finite-range dependence gives independence. Averaging therefore reduces their fluctuations by the square root of the number of children, uniformly in the adapted geometry. This is the probabilistic gain established in Lemma~\ref{l.fixed.geometry.matrix.averaging}.

\smallskip

The difficulty is that the parent matrix need not equal the average of its children. Subadditivity nevertheless puts it below that average, so the difference is positive semidefinite. Its expectation cannot be small through cancellation: in every direction, the difference is nonnegative. Together with control of the child average, this sign allows us to bound the discrepancy by the decrease of the mean (Lemma~\ref{l.fixed.geometry.positive.gap}). Determinant loss measures this decrease because the two means are ordered: after normalizing the parent mean to the identity, every eigenvalue of the child mean is at least one, so a determinant close to one forces all of them to be close to one. Thus averaging controls one part of the parent fluctuation, and determinant loss controls the remaining discrepancy and the change of normalization. Proposition~\ref{p.fixed.geometry.parent.child.recurrence} makes this precise, without assuming contraction in advance.

\paragraph{Retaining earlier estimates.}
To contract the error plus drift~$\mathcal P_{\qq}(m;n)+D_{\qq,j_*}(m)$, we must also carry along the earlier scales. These do not need to improve again: their weights decrease as the current scale advances. The new fluctuations are controlled by averaging, while determinant loss controls the new mean errors, the new drift increments, and the cost of measuring old errors against the current mean. Over a sufficiently long fixed span, these two gains outweigh the constants in the estimates. When the determinant loss is small, Proposition~\ref{p.fixed.geometry.one.grid.propagation} therefore gives a contraction factor for the error plus drift, together with an additive determinant error. The histories remain controlled throughout; this is essential for the later change of grid and the comparison of the annealed coarse-grained matrix with its dual.

\paragraph{Startup and propagation.}
This argument explains the distinction between startup and contraction in Proposition~\ref{p.scale.selection}. Contraction compares each new fluctuation with one a fixed number of scales earlier. At startup, there is not yet a full span of scales after initialization to compare with, so we bound the new fluctuations directly from the initialization scale. This gives a bound by a fixed multiple of the small input and a term controlled by determinant loss, but not necessarily a decrease. It is the startup in Case~1. After~$h$ scales, the comparisons needed for contraction are available. If the error plus drift still exceeds the threshold and the determinant test passes, we obtain Case~2. Even then, the additive determinant error remains: the estimate need not decrease the error at every step.

\smallskip

When the error plus drift is already small and the means vary little, we use these estimates to preserve smallness over the advance needed for Case~3. Section~\ref{s.geometry.transport} then transfers it to the new grid. If either determinant test fails, Cases~4 and~5 still advance on the old grid. The estimates allow growth, but this growth is controlled by determinant loss, whose total is bounded in the global argument. The comparisons in successive contraction steps overlap, but Proposition~\ref{p.fixed.geometry.one.grid.propagation} bounds how often they count each one-scale decrease. The local losses can therefore be summed without losing control.

\smallskip

The three subsections establish the adapted cube bound~\eqref{e.source.adapted.bound}, the parent--child recurrence, and propagation of the error and drift. The first supplies the moment control needed throughout; the last also gives the mean comparisons used in the change of grid. Subsection~\ref{s.scale.selection.proof} assembles the scale-selection alternatives.

\subsection{The adapted cube bound}
\label{ss.source.estimate}

The coarse ellipticity assumption~\ref{a.coarse.ellipticity} controls the coarse-grained matrices on standard (Euclidean) cubes. Its exponent~$\gamma\in[0,1)$ measures the permitted deterioration of this bound as the cubes shrink. To obtain a bound on an adapted cube, we partition it into standard cubes and apply subadditivity, as in~\cite{AK.HC}. The difficulty is that cubes near the boundary can be arbitrarily small. Their total volume compensates for the weaker matrix bounds: inside a scale-$j$ adapted cube, the relative volume of the scale-$r$ pieces is at most~$C3^{r-j}$. Each decrease of one scale multiplies this volume bound by~$\nf13$, while the matrix bound grows by at most~$3^\gamma$. Since~$\gamma<1$, the volume-weighted sum converges. The next lemma supplies the partition and its volume estimate.

\begin{samepage}
\begin{lemma}[Standard cubes inside an adapted cube]
\label{l.source.whitney}
Let~$\m\in\Rpos$, set~$\qq=\mathcal Q(\m)$, and let~$y\in\Rd$ and~$j\in\Z$. The maximal standard aligned cubes contained in~$y+\cus_j^{\qq}$ with scales at most~$j$ partition~$y+\cus_j^{\qq}$ up to a null set. If~$\mathcal Z_r(y+\cus_j^{\qq})$ denotes their centers at scale~$r$, then
\begin{equation}
	\sum_{z\in\mathcal Z_r(y+\cus_j^{\qq})}\frac{|\cu_r|}{|\cus_j^{\qq}|}\leq12d^{\nf32}3^{r-j}\qquad(r\leq j)\,.
	\label{e.source.whitney.volumes}
\end{equation}
\end{lemma}
\end{samepage}

\begin{proof}
Aligned triadic cubes are nested or disjoint, so the selected cubes are disjoint. Every interior point of~$y+\cus_j^{\qq}$ outside the grid boundaries lies in a sufficiently small aligned cube contained in~$y+\cus_j^{\qq}$, and hence in a maximal one of scale at most~$j$. This proves the partition property. At scale~$j$, disjointness bounds the relative volume by one.

\smallskip

For~$s>0$, the part of~$y+\cus_j^{\qq}$ within Euclidean distance~$s$ of its boundary is contained in the union of the strips of width~$s$ along its~$2d$ faces. Under~$x=y+\qq v$, each strip becomes a coordinate strip in~$\cu_j$ of width at most~$s|\qq^{-1}|$. Volume ratios are preserved, so
\begin{equation*}
	\frac{\bigl|\{x\in y+\cus_j^{\qq}:\dist(x,\partial(y+\cus_j^{\qq}))\leq s\}\bigr|}{|\cus_j^{\qq}|}\leq2d|\qq^{-1}|s3^{-j}\leq4ds3^{-j}\,.
\end{equation*}
If~$r<j$, the parent of each selected scale-$r$ cube meets the complement of~$y+\cus_j^{\qq}$. Every point of that cube is therefore within distance~$\sqrt d\,3^{r+1}$ of~$\partial(y+\cus_j^{\qq})$. Apply the preceding estimate with~$s=\sqrt d\,3^{r+1}$ and use disjointness to obtain~\eqref{e.source.whitney.volumes}.
\end{proof}

\begin{proof}[Proof of the adapted cube bound~\eqref{e.source.adapted.bound}]
We adapt the subdivision argument of~\cite[Lemma~2.8]{AK.HC}, using a moment bound rather than retaining the full tail estimate. Recall the moment exponent~$Q$ from~\eqref{e.scale.selection.Q.choice} and the tail growth constant~$K_{\Psi_{\S}}$ from assumption~\ref{a.coarse.ellipticity}. Fix~$p=2(d+Q)$ and put~$\overline K=\max\{2,K_{\Psi_{\S}}\}$. The growth condition in~\eqref{e.source.tail} remains valid with~$\overline K$, and~$\overline K\leq2K_{\Psi_{\S}}$. Thus~\cite[Lemma~C.1]{AK.HC} and~\eqref{e.source.tail} give~$\E[\S(0)^p]\leq C(p)K_{\Psi_{\S}}^{C(p)}$.

Let~$\ell_*$ be the first subdivision depth at which every value of~$\S$ appearing in~\eqref{e.source.multiplier} is at most the corresponding cube size. At a fixed depth~$r$, there are~$3^{dj_*}$ integer centers. The variables~$\S(z)$ have the same law as~$\S(0)$, so the union bound and Markov's inequality give
\begin{equation}
	\P[\ell_*>r]\leq3^{dj_*}\P[\S(0)>3^{j_*+r}]\leq C(p)K_{\Psi_{\S}}^{C(p)}3^{-(p-d)j_*}3^{-pr}\qquad(r\geq0)\,.
	\label{e.source.subdivision.tail}
\end{equation}
In particular,~$\ell_*$ is finite almost surely, so the minimum in~\eqref{e.source.multiplier} is attained and~$\RSZ=3^{\gamma\ell_*}$. If~$\gamma>0$, summing~\eqref{e.source.subdivision.tail} yields
\begin{equation*}
	\E\bigl[ \RSZ^Q \bigr]
	\leq
	1+C(d,Q)K_{\Psi_{\S}}^{C(d,Q)}3^{-(p-d)j_*}\sum_{r=0}^{\infty}3^{-(p-Q\gamma)r}\leq2^Q\,.
\end{equation*}
The last inequality follows from~\eqref{e.source.lower.scale} after choosing~$C_{\rm src}$ large enough. If~$\gamma=0$, then~$\RSZ=1$ and the same conclusion holds.

\smallskip

At every center~$z$ in the subdivision with~$r=\ell_*$, we have~$\S(z)\leq3^{j_*+\ell_*}$. Apply~\eqref{e.coarse.ellipticity} there with~$m=j_*+\ell_*$. Every finer standard aligned cube belongs to one of these subcubes; coarser cubes are unions of them. Subadditivity therefore gives, simultaneously for all standard aligned cubes contained in~$\cu_{2j_*}$,
\begin{equation*}
	\bfA(z+\cu_j)\leq3^{\gamma(j_*+\ell_*-j)_+}\bfE\leq\RSZ3^{\gamma(j_*-j)_+}\bfE\,.
\end{equation*}
For an adapted cube, apply Lemma~\ref{l.source.whitney}. Its relative-volume bound is independent of~$\m$. Subadditivity weights each standard-cube bound by its relative volume in the partition. The resulting volume decay dominates the permitted growth at small scales because~$\gamma<1$: indeed,
\begin{equation*}
	\sum_{r\leq j}3^{r-j}3^{\gamma(j_*-r)_+}\leq\frac{3^{\gamma(j_*-j)_+}}{1-3^{-(1-\gamma)}}\,,
\end{equation*}
the standard-cube bound proves~\eqref{e.source.adapted.bound}, using~$(|\m|\,|\m^{-1}|)^{\nf12}\geq1$. This also proves convergence in~$L^Q(S_Q)$ of these volume-weighted sums. Finite exhaustion and local integrability of the coefficient block justify subadditivity, as in~\cite{AK.HC}.
\end{proof}

\subsection{The parent--child recurrence}

We compare the normalized fluctuation~$V_{j+h}^{\qq}$ on a parent cube with~$V_j^{\qq}$ on its children, the aligned cubes~$h$ scales below. Averaging over the~$3^{dh}$ children reduces fluctuations, but the parent matrix is not their average, and the two scales use different normalizations. Proposition~\ref{p.fixed.geometry.parent.child.recurrence} controls both discrepancies through the determinant loss~$\Delta_{j,j+h}^{\qq}$. Summing this estimate over scales in Subsection~\ref{ss.fixed.geometry.contraction} gives the fixed-grid estimates behind Cases~1--3 of Proposition~\ref{p.scale.selection}.

\smallskip

Finite-range dependence gives the averaging gain: the children can be sorted into a bounded number of classes of well separated cubes, within which the matrices are independent. Their centered average gains the square root of the number of children, giving the factor~$3^{-\nf d2h}$ in Lemma~\ref{l.fixed.geometry.matrix.averaging}. Subadditivity gives the other ingredient. The difference between the child average and the parent matrix is positive semidefinite, and its expectation is the decrease of the annealed matrix between the two scales. Positivity, together with the moment bound on the child average, controls this random difference through its mean, as made precise in Lemma~\ref{l.fixed.geometry.positive.gap}. Thus the averaging estimate transfers to the parent, at a cost determined by the decrease of the mean.

\smallskip

This decrease is what the determinant loss measures. After normalization by the annealed matrix at the coarser scale, the mean~$P_{j,j+h}^{\qq}$ has eigenvalues at least one whose product is~$e^{\Delta_{j,j+h}^{\qq}}$, so its trace exceeds that of the identity by at most~$e^{\Delta_{j,j+h}^{\qq}}-1$ and its operator norm is at most~$e^{\Delta_{j,j+h}^{\qq}}$. The first bound controls the gap, the second the change of normalization. Both are therefore charged to the logarithmic determinant, whose total decrease along the iteration is bounded as described in Section~\ref{ss.outline}.

\smallskip

Fix~$\qq=\mathcal Q(\m)$ and~$j\geq j_*$. The cell~$\cus_{j+h}^{\qq}$ is the union of~$3^{hd}$ aligned scale-$j$ cells. Stationarity and subadditivity~\cite[(2.24)]{AK.HC} give
\begin{equation}
	0\leq\bfA(\cus_{j+h}^{\qq})
	\leq\avsum_{z\in3^j\qq\Zd\cap\cus_{j+h}^{\qq}}
	\bfA(z+\cus_j^{\qq})\,.
	\label{e.fixed.geometry.parent.child}
\end{equation}

For~$N=2$, the following estimate is the usual variance reduction for an average. We need higher moments because the fluctuation history involves maxima over many cells and scales; these moments control the cost of taking a maximum. The bound on the number of classes of independent matrices, and hence the averaging gain, is uniform in the adapted geometry.

\begin{lemma}[Finite-range matrix averaging]
\label{l.fixed.geometry.matrix.averaging}
Let~$N\geq2$ be even, let~$\qq=\mathcal Q(\m)$, and let~$j\geq j_*$. Let~$Z$ be a nonempty finite subset of~$3^j\qq\Zd$, and let~$R$ be a deterministic positive definite~$2d \times 2d$ matrix. Then
\begin{equation}
	\biggl\|\avsum_{z\in Z}
	R^{-\nf12} \bigl(\bfA(z+\cus_j^{\qq})-\bfAhom_{j,\qq}\bigr)
	R^{-\nf12}\biggr\|_{L^N(S_N)}
 \leq \frac{N 3^{\nf d2}}{(\# Z)^{\nf12}}
	\biggl\|R^{-\nf12}
	\bigl(\bfA(\cus_j^{\qq})-\bfAhom_{j,\qq}\bigr)
	R^{-\nf12}\biggr\|_{L^N(S_N)}\,.
	\label{e.fixed.geometry.matrix.averaging}
\end{equation}
\end{lemma}

\begin{proof}
Write~$z=3^j\qq w$ and color the adapted cubes by the residue class of~$w$ modulo three. There are at most~$3^d$ classes. Before applying~$3^j\qq$, distinct cubes in one class are separated by Euclidean distance at least two. By~\eqref{e.rounded.grid.bounds}, their images have Euclidean separation at least~$3^j$, and hence~$\ell^\infty$-separation at least~$\frac{3^j}{\sqrt d}\geq1$. The range-of-dependence assumption~\ref{a.finite.range} therefore makes the cubes in each class jointly independent: the union of any preceding cubes in the class remains at distance at least one from the next cube. Applying the two-set independence property successively proves joint independence.

\smallskip

Let~$(Y_i)_{i\in I}$ be a finite family of independent centered symmetric random matrices with finite~$N$-th moments. Since~$N$ is even, the trace of the~$N$-th power of a symmetric matrix is nonnegative. Expanding the power gives
\begin{equation*}
	\tr\E\Biggl[ \biggl(\sum_{i\in I}Y_i\biggr)^{\! N} \Biggr]
	=\sum_{i_1,\ldots,i_N\in I}\tr\E[Y_{i_1}\cdots Y_{i_N}]\,.
\end{equation*}
If one index occurs exactly once, conditioning on all the other matrices and using centering shows that the corresponding term vanishes. For every remaining term, trace H\"older followed by H\"older in probability gives
\begin{equation*}
	\bigl|\tr\E[Y_{i_1}\cdots Y_{i_N}]\bigr|
	\leq\prod_{k=1}^N\left(\tr\E[Y_{i_k}^N]\right)^{\nf1N}\,.
\end{equation*}
Fix an equality partition~$\pi$ of~$\{1,\ldots,N\}$ with no singleton blocks. Since every block~$B\in\pi$ has at least two elements, summing the preceding estimate over the terms inducing~$\pi$ gives
\begin{align*}
	\sum_{\text{terms inducing }\pi}
	\bigl|\tr\E[Y_{i_1}\cdots Y_{i_N}]\bigr|
\leq\prod_{B\in\pi}\left(\sum_{i\in I}\left(\tr\E[Y_i^N]\right)^{\nf2N}\right)^{\! \nf{|B|}2}
=\left(\sum_{i\in I}\left(\tr\E[Y_i^N]\right)^{\nf2N}\right)^{\! \nf N2}\,.
\end{align*}
There are at most~$N^N$ equality partitions. Summing over them gives
\begin{equation*}
	\tr\E\Biggl[\biggl(\sum_{i\in I}Y_i\biggr)^{\!N}\Biggr]
	\leq N^N\Biggl(\sum_{i\in I}\left(\tr\E[Y_i^N]\right)^{\nf2N}\Biggr)^{\!\nf N2}\,.
\end{equation*}
Since~$N$ is even,~$\tr(Y^N)=|Y|_{S_N}^N$ for every symmetric matrix~$Y$. Taking the~$N$-th root therefore gives
\begin{equation*}
	\biggl\|\sum_{i\in I}Y_i\biggr\|_{L^N(S_N)}
	\leq N\biggl(\sum_{i\in I}\|Y_i\|_{L^N(S_N)}^2\biggr)^{\!\nf12}\,.
\end{equation*}
Apply this estimate to the normalized centered matrices in each color class~$\mathcal C$. Stationarity gives
\begin{equation*}
	\biggl\|R^{-\nf12}\avsum_{z\in\mathcal C}
	\bigl(\bfA(z+\cus_j^{\qq})-\bfAhom_{j,\qq}\bigr)
	R^{-\nf12}\biggr\|_{L^N(S_N)}
	\leq \frac{N}{(\#\mathcal C)^{\nf12}}
	\biggl\|R^{-\nf12}
	\bigl(\bfA(\cus_j^{\qq})-\bfAhom_{j,\qq}\bigr)
	R^{-\nf12}\biggr\|_{L^N(S_N)}\,.
\end{equation*}
The class sizes satisfy
\begin{equation*}
	\sum_{\mathcal C}(\#\mathcal C)^{\nf12}
	\leq3^{\nf d2}(\# Z)^{\nf12}\,.
\end{equation*}
The average over~$Z$ is the convex combination of these class averages with weights~$\nf{\#\mathcal C}{\# Z}$. Minkowski's inequality and the preceding bound prove~\eqref{e.fixed.geometry.matrix.averaging}.
\end{proof}

Averaging controls the upper matrix in~\eqref{e.fixed.geometry.parent.child}. The next lemma transfers that control to the smaller matrix, with an error measured by the difference of their means.

\begin{lemma}[Positive gap]
\label{l.fixed.geometry.positive.gap}
Let~$N\geq2$, and let~$F$ and~$G$ be positive semidefinite random~$2d \times 2d$ matrices in~$L^N(S_N)$ such that~$F\leq G$. Then
\begin{align}
\lefteqn{ 
	\bigl\|F-\E[F]\bigr\|_{L^N(S_N)}
} \quad & 
\notag \\ & 
\leq\bigl(1+(2d)^{\nf1N}\bigr)
\bigl\|G-\E[G]\bigr\|_{L^N(S_N)}
+2\bigl(1+d^{1-\nf1N}\bigr)
|\E[G]|^{1-\nf1N}
\bigl(\tr(\E[G]-\E[F])\bigr)^{\nf1N}\,.
\label{e.fixed.geometry.positive.gap}
\end{align}
\end{lemma}

\begin{proof}
Set~$D\coloneqq G-F$, so that~$0\leq D\leq G$ and~$\E[D]=\E[G]-\E[F]$. Every eigenvalue of~$D$ is at most~$|G|$, hence~$\tr(D^N)\leq |G|^{N-1}\tr D$. Taking the~$N$-th root gives
\begin{equation*}
	|D|_{S_N}\leq|G|^{1-\nf1N}(\tr D)^{\nf1N}\,.
\end{equation*}
Since the power~$1-\frac1N$ is subadditive, H\"older's inequality and~$\tr D\leq(2d)^{1-\nf1N}|D|_{S_N}$ give
\begin{equation*}
	\|D\|_{L^N(S_N)}
	\leq|\E[G]|^{1-\nf1N}
	\bigl(\tr(\E[G]-\E[F])\bigr)^{\nf1N}
	+(2d)^{\frac{N-1}{N^2}}
	\bigl\|G-\E[G]\bigr\|_{L^N(S_N)}^{1-\nf1N}
	\|D\|_{L^N(S_N)}^{\nf1N}\,.
\end{equation*}
Young's inequality with exponents~$\frac{N}{N-1}$ and~$N$ lets us absorb the resulting multiple of~$\|D\|_{L^N(S_N)}$ into the left side, giving
\begin{equation*}
	\|D\|_{L^N(S_N)}
	\leq(2d)^{\nf1N}\bigl\|G-\E[G]\bigr\|_{L^N(S_N)}
	+\frac{N}{N-1}|\E[G]|^{1-\nf1N}
	\bigl(\tr(\E[G]-\E[F])\bigr)^{\nf1N}\,.
\end{equation*}
Also,
\begin{equation*}
	F-\E[F]
	=G-\E[G]-\bigl(D-(\E[G]-\E[F])\bigr)\,.
\end{equation*}
Since~$\E[G]-\E[F]\geq0$, its Schatten norm is at most its trace. Since also~$\tr(\E[G]-\E[F])\leq2d\,|\E[G]|$,
\begin{equation*}
	\tr(\E[G]-\E[F])
	\leq(2d)^{1-\nf1N}|\E[G]|^{1-\nf1N}
	\bigl(\tr(\E[G]-\E[F])\bigr)^{\nf1N}\,.
\end{equation*}
The triangle inequality, the preceding two estimates, and the bounds~$\frac{N}{N-1}\leq2$ and~$(2d)^{1-\nf1N}\leq2d^{1-\frac1N}$ give~\eqref{e.fixed.geometry.positive.gap}.
\end{proof}

Combining the two lemmas gives a recurrence whose error is measured by the determinant loss. Recall~$V_{j,k}^{\qq}$ from~\eqref{e.scale.selection.normalized.mean.fluctuation} and~$\Delta_{j,k}^{\qq}$ from~\eqref{e.scale.selection.logdet.loss}; in particular,~$V_j^{\qq}=V_{j,j}^{\qq}(0)$.

\begin{proposition}[Parent--child recurrence]
\label{p.fixed.geometry.parent.child.recurrence}
Let~$N\geq2$ be even, let~$\m\in\Rpos$, set~$\qq=\mathcal Q(\m)$, and let~$j\geq j_*$ and~$h\geq1$ be integers. Then we have
\begin{align}
\lefteqn{ 
	\|V_{j+h}^{\qq}\|_{L^N(S_N)}
} \quad & 
\notag \\ &
\leq N3^{\nf d2}\bigl(1+(2d)^{\nf1N}\bigr)
	3^{-\frac{d}{2}h}e^{\Delta_{j,j+h}^{\qq}}
	\|V_j^{\qq}\|_{L^N(S_N)}
	+2\bigl(1+d^{1-\nf1N}\bigr)
	e^{(1-\nf1N)\Delta_{j,j+h}^{\qq}}
	\bigl(e^{\Delta_{j,j+h}^{\qq}}-1\bigr)^{\nf1N}\,.
	\label{e.fixed.geometry.parent.child.recurrence}
\end{align}
\end{proposition}

The first term is the averaging gain, with the factor~$e^{\Delta_{j,j+h}^{\qq}}$ accounting for the change of normalization. The second term controls the positive gap between the child average and the parent through the decrease of their means. If the determinant loss vanishes, that positive gap has zero expectation and hence vanishes almost surely; only the averaging term remains. A nonzero loss permits both a normalization cost and an additive error. The recurrence bounds the fluctuations in either case, without assuming that they already contract.

\begin{proof}
The aligned scale-$j$ cells partition~$\cus_{j+h}^{\qq}$ into~$3^{hd}$ equal pieces. Set
\begin{equation*}
	G\coloneqq \avsum_{z\in3^j\qq\Zd\cap\cus_{j+h}^{\qq}}
	\bfA(z+\cus_j^{\qq})\,.
\end{equation*}
Subadditivity and stationarity give
\begin{equation*}
	0\leq\bfA(\cus_{j+h}^{\qq})\leq G\,,
	\quad \E[G]=\bfAhom_{j,\qq}
	\qqand \E\bigl[\bfA(\cus_{j+h}^{\qq})\bigr]=\bfAhom_{j+h,\qq}\,.
\end{equation*}
Taking expectations gives~$\bfAhom_{j+h,\qq}\leq\bfAhom_{j,\qq}$. Conjugating by~$\bfAhom_{j+h,\qq}^{-\nf12}$ gives
\begin{equation*}
	P_{j,j+h}^{\qq}\geq\Itwod\,,
	\quad \det P_{j,j+h}^{\qq}=e^{\Delta_{j,j+h}^{\qq}}\,,
	\quad \Delta_{j,j+h}^{\qq}\geq0
	\qqand |P_{j,j+h}^{\qq}|\leq e^{\Delta_{j,j+h}^{\qq}}\,.
\end{equation*}
The summands defining~$G$ are stationary copies of~$\bfA(\cus_j^{\qq})$. Lemma~\ref{l.fixed.geometry.matrix.averaging}, with~$R=\bfAhom_{j,\qq}$, gives
\begin{equation*}
	\bigl\|\bfAhom_{j,\qq}^{-\nf12}
	\bigl(G-\bfAhom_{j,\qq}\bigr)
	\bfAhom_{j,\qq}^{-\nf12}\bigr\|_{L^N(S_N)}
	\leq N\,3^{\nf d2}3^{-\frac{d}{2}h}\|V_j^{\qq}\|_{L^N(S_N)}\,.
\end{equation*}
Changing normalization conjugates the matrix in this norm by~$\bfAhom_{j,\qq}^{\nf12}\bfAhom_{j+h,\qq}^{-\nf12}$. Since the square of its operator norm is~$|P_{j,j+h}^{\qq}|$, the Schatten ideal property gives
\begin{equation*}
	\bigl\|\bfAhom_{j+h,\qq}^{-\nf12}
	\bigl(G-\bfAhom_{j,\qq}\bigr)
	\bfAhom_{j+h,\qq}^{-\nf12}\bigr\|_{L^N(S_N)}
	\leq N\,3^{\nf d2}3^{-\frac{d}{2}h}e^{\Delta_{j,j+h}^{\qq}}
	\|V_j^{\qq}\|_{L^N(S_N)}\,.
\end{equation*}
The bounded-window estimate gives the required~$L^N(S_N)$-integrability.

Apply Lemma~\ref{l.fixed.geometry.positive.gap} to~\mbox{$\bfA(\cus_{j+h}^{\qq})$} and~$G$, after conjugating both by~$\bfAhom_{j+h,\qq}^{-\nf12}$. Their expectations are~$\Itwod$ and~$P_{j,j+h}^{\qq}$. The preceding estimate gives
\begin{align*}
	\|V_{j+h}^{\qq}\|_{L^N(S_N)}
	&\leq N\,3^{\nf d2}\bigl(1+(2d)^{\nf1N}\bigr)
	3^{-\frac{d}{2}h}e^{\Delta_{j,j+h}^{\qq}}
	\|V_j^{\qq}\|_{L^N(S_N)}
	\notag \\ & \qquad
	+2\bigl(1+d^{1-\nf1N}\bigr)
	|P_{j,j+h}^{\qq}|^{1-\nf1N}
	\bigl(\tr(P_{j,j+h}^{\qq}-\Itwod)\bigr)^{\nf1N}\,.
\end{align*}
Let~$\lambda_1,\ldots,\lambda_{2d}$ be the eigenvalues of~$P_{j,j+h}^{\qq}$. Since~$\lambda_i\geq1$,
\begin{equation*}
	\tr(P_{j,j+h}^{\qq}-\Itwod)
	=\sum_{i=1}^{2d}(\lambda_i-1)
	\leq\prod_{i=1}^{2d}\lambda_i-1
	=e^{\Delta_{j,j+h}^{\qq}}-1\,.
\end{equation*}
The preceding estimates prove~\eqref{e.fixed.geometry.parent.child.recurrence}.
\end{proof}

\subsection{Contraction of the error on a fixed grid}
\label{ss.fixed.geometry.contraction}

To reduce~$\mathcal P_{\qq}(m;n)+D_{\qq,j_*}(m)$, we separate the errors already recorded at~$m$ from those entering between~$m$ and~$m+h$. The old errors receive smaller weights as the scale advances. For each new scale~$m<j\leq m+h$, we compare the coarse matrix on a scale-$j$ cube with the average over its scale-$(j-h)$ subcubes. Finite-range dependence reduces the fluctuations of this average. Subadditivity makes the average minus the large-cube matrix nonnegative, with expectation~$\bfAhom_{j-h,\qq}-\bfAhom_{j,\qq}$. Thus averaging controls the moments of~$V_j^{\qq}$ by those of~$V_{j-h}^{\qq}$, up to an error from the decrease of the mean. Proposition~\ref{p.fixed.geometry.parent.child.recurrence} quantifies this comparison using the determinant loss~$\Delta_{j-h,j}^{\qq}$.

\smallskip

The normalizing mean also changes with scale. Passing from~$\bfAhom_{m,\qq}$ to~$\bfAhom_{m+h,\qq}$, and adding the new mean errors and drift increments, incurs a cost controlled by~$\Delta_{m,m+h}^{\qq}$. Some fluctuation comparisons reach back to~$j-h<m$, however, so this forward loss alone is insufficient. All these contributions are controlled by the sum from~\eqref{e.scale.selection.synchronized.loss},
\begin{equation*}
	\widehat\Delta_h^{\qq}(m)=\sum_{j=m+1}^{m+h}\Delta_{j-h,j}^{\qq}\,.
\end{equation*}
Along successive length-$h$ steps, each one-scale loss~$\Delta_{r,r+1}^{\qq}$ is counted at most~$h$ times, which gives~\eqref{e.fixed.geometry.synchronized.multiplicity}.

\smallskip

When~$m\geq n+h$, the earlier scales satisfy~$n<j-h\leq m$, so their fluctuation moments are already included in~$\mathcal P_{\qq}(m;n)$. With~$Q$ chosen as in~\eqref{e.scale.selection.Q.choice}, taking~$h\geq2Q$ makes the decay of the old weights and the averaging gain strong enough to absorb the constants. This yields contraction up to determinant loss, as in~\eqref{e.fixed.geometry.synchronized.propagation}, and supplies Case~2 of Proposition~\ref{p.scale.selection}.

\smallskip

At startup,~$m=n$, the past is still retained in~$\mathcal P_{\qq}(n;n)=\history_{\qq}(n)$, but there are no post-initialization scales to use in these comparisons. Starting from~$\mathcal P_{\qq}(n;n)+D_{\qq,j_*}(n)\leq1$, we compare~$V_{n+r}^{\qq}$ directly with~$V_n^{\qq}$ for~$1\leq r\leq h$. These short spans need not give contraction, but~\eqref{e.fixed.geometry.fixed.span.propagation} with~$L=h$ carries the initial bound to~$n+h$, up to a fixed factor and determinant loss. Startup therefore preserves control until contraction is available; it does not create smallness. The estimate for general~$L$ also controls the old-grid advance before a geometry update.

\begin{proposition}[Propagation on one adapted geometry]
\label{p.fixed.geometry.one.grid.propagation}
There exists~$C=C(d,\gamma)<\infty$ such that the following holds. Let~$h\in\N$ satisfy~$h\geq2Q$, let~$L\geq1$ be an integer, let~$\m\in\Rpos$, set~$\qq=\mathcal Q(\m)$, and let~$j_*\leq n\leq m$ be integers.

\smallskip

\noindent\emph{(i) Control of the histories.} We have
\begin{equation}
	\history_{\qq}(m)\leq C\mathcal P_{\qq}(m;n)\,.
	\label{e.fixed.geometry.profile.majorization}
\end{equation}
\noindent\emph{(ii) Contraction.} If~$m\geq n+h$, then
\begin{equation}
	\mathcal P_{\qq}(m+h;n)+D_{\qq,j_*}(m+h)\leq\frac18e^{Q\widehat\Delta_h^{\qq}(m)}\Bigl(\mathcal P_{\qq}(m;n)+D_{\qq,j_*}(m)\Bigr)+C\Bigl(e^{Q\widehat\Delta_h^{\qq}(m)}-1\Bigr)\,.
	\label{e.fixed.geometry.synchronized.propagation}
\end{equation}
\noindent\emph{(iii) Propagation from a small input.} Suppose that~$\mathcal P_{\qq}(m;n)+D_{\qq,j_*}(m)\leq1$. If either~$m=n$ or~$m\geq n+h$, then
\begin{equation}
	\mathcal P_{\qq}(m+L;n)+D_{\qq,j_*}(m+L)\leq C L\Bigl(\mathcal P_{\qq}(m;n)+D_{\qq,j_*}(m)+e^{Q\Delta_{m,m+L}^{\qq}}-1\Bigr)\,.
	\label{e.fixed.geometry.fixed.span.propagation}
\end{equation}
\noindent\emph{(iv) Summation of determinant losses.} If~$m_0\geq j_*+h$, then, for every~$K\geq1$,
\begin{equation}
	\sum_{k=0}^{K-1}\widehat\Delta_h^{\qq}(m_0+kh)
	\leq h\Delta_{m_0+1-h,m_0+Kh}^{\qq}\,.
	\label{e.fixed.geometry.synchronized.multiplicity}
\end{equation}
\end{proposition}

By definition,~$\E[|V_m^{\qq}|^Q]\leq\mathcal P_{\qq}(m;n)$, and thus the error bounds the current fluctuation as well as its history. The retained history supplies the earlier-scale input to the parent--child recurrence, which controls the new fluctuations whether the determinant loss is small or large. The contraction estimate requires no smallness assumption on this input or on the determinant loss; the condition~$m\geq n+h$ required for~\eqref{e.fixed.geometry.synchronized.propagation} ensures that the earlier scales needed for comparison are already included in the error. When the determinant test in Case~2 of Proposition~\ref{p.scale.selection} passes, the estimate gives a contraction factor plus an additive determinant error, not necessarily a decrease at every step. When a test fails, its intended conclusion is no longer guaranteed, but the estimates still bound possible growth of the fluctuations. The iteration can continue with a larger error; Section~\ref{s.global.selection} sums its logarithmic cost. The propagation estimate instead assumes a small input, but also applies when~$m=n$; it supplies both the startup bound and the advance used before a geometry update.

\begin{proof}
We first verify that the error controls the histories. For contraction and propagation, we separate the errors already recorded from those entering at new scales. The old contributions decay through their weights, while the new fluctuations are controlled by the parent--child recurrence. Determinant losses control the changes of normalization, the new mean errors, and the new drift increments. We first estimate the error alone and then add the drift; treating the error on a fresh grid separately gives the startup bound. Finally, we count the overlaps between determinant losses in successive contraction steps. Throughout the proof,~$C$ may change from line to line and depends only on~$d$ and~$\gamma$.

\smallskip

\emph{Step 1: Control of the histories.} We estimate the fluctuation and mean histories separately. A change of normalization gives
\begin{equation*}
	V_{j,m}^{\qq}(z)=\bigl(\bfAhom_{n,\qq}^{\nf12}\bfAhom_{m,\qq}^{-\nf12}\bigr)^t V_{j,n}^{\qq}(z)\bigl(\bfAhom_{n,\qq}^{\nf12}\bfAhom_{m,\qq}^{-\nf12}\bigr)\,
	\qquad j\leq n\,.
\end{equation*}
The square of the conjugating matrix's norm is~$|P_{n,m}^{\qq}|$. Thus the first term of the error controls fluctuations from scales up to~$n$, while its last sum controls the remaining scales. More precisely, for~$j\leq n$, partition~$\cus_m^{\qq}$ into scale-$n$ cells, bound the maximum by a sum, and use stationarity and the change of normalization. For~$n<j\leq m$, bound the maximum by a sum over the scale-$j$ cells. Since~$Q\rho_{\max}-d\geq\frac14(1-\gamma)$,
\begin{align*}
	\history_{\qq}^{\rm fluc}(m)
	&\leq3^{-\frac14(1-\gamma)(m-n)}
	\bigl(1+\meanpenalty_Q(P_{n,m}^{\qq})\bigr)
	\history_{\qq}^{\rm fluc}(n)
	+\sum_{j=n+1}^{m}
	3^{-\frac14(1-\gamma)(m-j)}e^{Q\Delta_{j,m}^{\qq}}
	\E\bigl[|V_j^{\qq}|_{S_Q}^Q\bigr]
	\notag \\ &
	\leq\mathcal P_{\qq}(m;n)\,.
\end{align*}

\smallskip

To obtain the analogous mean estimate, we use a trace inequality that will also control changes of normalization in Step~2. If~$P,K\geq\Itwod$, cyclicity gives
\begin{align*}
	1+\tr\bigl(K^{\nf12}PK^{\nf12}-\Itwod\bigr)
	&=1+\tr(P-\Itwod)+\tr(K-\Itwod)
	+\tr\bigl((P-\Itwod)(K-\Itwod)\bigr) \\
	&\leq\bigl(1+\tr(P-\Itwod)\bigr)
	\bigl(1+\tr(K-\Itwod)\bigr)\,.
\end{align*}
The inequality follows from positivity and~$\tr\bigl((P-\Itwod)(K-\Itwod)\bigr)\leq\tr(P-\Itwod)\tr(K-\Itwod)$.

\smallskip

Split the mean history at~$n$. The terms with~$j\geq n$ give~$\history_{\qq}^{\rm mean}(m;n)$. For~$j_*\leq j<n$, apply the trace inequality after changing normalization, raise it to the~$Q$-th power, and subtract one. This gives
\begin{equation*}
	\meanpenalty_Q(P_{j,m}^{\qq})\leq\bigl(1+\meanpenalty_Q(P_{n,m}^{\qq})\bigr)\bigl(1+\meanpenalty_Q(P_{j,n}^{\qq})\bigr)-1=\bigl(1+\meanpenalty_Q(P_{n,m}^{\qq})\bigr)\meanpenalty_Q(P_{j,n}^{\qq})+\meanpenalty_Q(P_{n,m}^{\qq})\,.
\end{equation*}
Multiplying by~$3^{-\frac14(1-\gamma)(m-1-j)}$ and summing over~$j_*\leq j<n$ gives
\begin{align*}
	\sum_{j=j_*}^{n-1}3^{-\frac14(1-\gamma)(m-1-j)}
	\meanpenalty_Q(P_{j,m}^{\qq})
	&
	\leq3^{-\frac14(1-\gamma)(m-n)}
	\bigl(1+\meanpenalty_Q(P_{n,m}^{\qq})\bigr)
	\history_{\qq}^{\rm mean}(n)
	\notag \\ & \qquad
	+3^{-\frac14(1-\gamma)(m-n)}
	\meanpenalty_Q(P_{n,m}^{\qq})
	\sum_{j=j_*}^{n-1}3^{-\frac14(1-\gamma)(n-1-j)}\,.
\end{align*}
The geometric sum is bounded by~$C$. The first term is bounded by the first term in~$\mathcal P_{\qq}(m;n)$. If~$m=n$, the second vanishes; if~$m>n$, it is bounded by the~$j=n$ term in~$\history_{\qq}^{\rm mean}(m;n)$ up to~$C$. Hence
\begin{equation*}
	\history_{\qq}^{\rm mean}(m)
	\leq C\mathcal P_{\qq}(m;n)\,.
\end{equation*}
Combining the two history estimates proves~\eqref{e.fixed.geometry.profile.majorization}. Adding~$D_{\qq,j_*}(m)$ gives the useful consequence
\begin{equation}
	\history_{\qq}^{\rm fluc}(m)+\history_{\qq}^{\rm mean}(m)+D_{\qq,j_*}(m)
	\leq C\bigl(\mathcal P_{\qq}(m;n)+D_{\qq,j_*}(m)\bigr)\,.
	\label{e.fixed.geometry.carried.majorization}
\end{equation}

\smallskip

\emph{Step 2: Contraction.} Assume~$m\geq n+h$. We first prove the error estimate
\begin{equation}
	\mathcal P_{\qq}(m+h;n)
	\leq\frac18e^{Q\widehat\Delta_h^{\qq}(m)}\mathcal P_{\qq}(m;n)
	+C\bigl(e^{Q\widehat\Delta_h^{\qq}(m)}-1\bigr)\,.
	\label{e.fixed.geometry.profile.contraction}
\end{equation}
The terms already present at scale~$m$ gain the factor~$3^{-\frac14(1-\gamma)h}$, up to the change of normalization. The new fluctuations are controlled by the parent--child recurrence. These two contributions give
\begin{align}
	\mathcal P_{\qq}(m+h;n)
	&\leq\biggl(3^{-\frac14(1-\gamma)h}
	+2^{Q-1}(3^dQ)^Q3^{-\frac{Qd}{2}h}\biggr)
	e^{Q\widehat\Delta_h^{\qq}(m)}\mathcal P_{\qq}(m;n)
	+C
	\bigl(e^{Q\widehat\Delta_h^{\qq}(m)}-1\bigr) \notag \\
	&\leq\frac18e^{Q\widehat\Delta_h^{\qq}(m)}\mathcal P_{\qq}(m;n)
	+C\bigl(e^{Q\widehat\Delta_h^{\qq}(m)}-1\bigr)\,.
	\label{e.fixed.geometry.profile.contraction.reduction}
\end{align}
We first check the coefficient bound in the second inequality, then estimate the terms of the error to prove the first.

\smallskip

Since~$d\geq2$ and~$Q\geq2$,
\begin{equation*}
	Q\,3^{\nf d2}\bigl(1+(2d)^{\nf1Q}\bigr)\leq3^dQ\,.
\end{equation*}
Since~$e^{(Q-1)t}(e^t-1)\leq e^{Qt}-1$ for~$t\geq0$, applying~\eqref{e.fixed.geometry.parent.child.recurrence} with~$N=Q$ and raising to the~$Q$-th power gives
\begin{equation}
	\E\bigl[|V_{j+h}^{\qq}|_{S_Q}^Q\bigr]\leq2^{Q-1}(3^dQ)^Q3^{-\frac{Qd}{2}h}e^{Q\Delta_{j,j+h}^{\qq}}\E\bigl[|V_j^{\qq}|_{S_Q}^Q\bigr]+2^{2Q-1}\bigl(1+d^{1-\nf1Q}\bigr)^Q\bigl(e^{Q\Delta_{j,j+h}^{\qq}}-1\bigr)\,.
	\label{e.fixed.geometry.parent.child.powered}
\end{equation}
By~\eqref{e.scale.selection.Q.choice},
\begin{equation*}
	(1-\gamma)h\geq2(1-\gamma)Q\geq8(d+1)\,.
\end{equation*}
Moreover,~$d\geq2$,~$Q\geq12$, and
\begin{equation*}
	2Q\,3^{-d(Q-1)}\leq\frac13\,.
\end{equation*}
Using~$h\geq2Q$, the second term in the following display is at most~$\frac12\bigl(2Q\,3^{-d(Q-1)}\bigr)^Q\leq\frac12 3^{-12}$. Therefore
\begin{equation*}
	3^{-\frac18(1-\gamma)h}\leq3^{-3}<\frac18
	\qand
	3^{-\frac14(1-\gamma)h}
	+2^{Q-1}(3^dQ)^Q3^{-\frac{Qd}{2}h}
	\leq3^{-6}+\frac12 3^{-12}<\frac18\,.
\end{equation*}
These inequalities give the required coefficient bounds; it remains to estimate the error terms.

\smallskip

We begin with the new fluctuation terms. Each new scale~$m<j\leq m+h$ is compared with~$j-h$. The assumption~$m\geq n+h$ required for~\eqref{e.fixed.geometry.synchronized.propagation} places~$j-h$ in~$[n+1,m]$, so its fluctuation moment is already represented in the error. These comparisons reach back before~$m$: the forward loss~$\Delta_{m,m+h}^{\qq}$ alone does not control their determinant errors. This is why we use the synchronized loss~$\widehat\Delta_h^{\qq}(m)$, which sums the losses~$\Delta_{j-h,j}^{\qq}$ over the new scales.

\smallskip

Apply~\eqref{e.fixed.geometry.parent.child.powered} at~$j-h$ and multiply by the weight in~$\mathcal P_{\qq}(m+h;n)$. Additivity and the definition of~$\widehat\Delta_h^{\qq}(m)$ give~$\Delta_{j,m+h}^{\qq}+\Delta_{j-h,j}^{\qq}=\Delta_{j-h,m+h}^{\qq}\leq\widehat\Delta_h^{\qq}(m)$. Moreover,
\begin{equation*}
e^{Q\Delta_{j,m+h}^{\qq}}\bigl(e^{Q\Delta_{j-h,j}^{\qq}}-1\bigr)\leq e^{Q\widehat\Delta_h^{\qq}(m)}-1
\,.
\end{equation*}
Summing over~$j$ gives
\begin{align*}
	\lefteqn{
	\sum_{j=m+1}^{m+h}3^{-\frac14(1-\gamma)(m+h-j)}
	e^{Q\Delta_{j,m+h}^{\qq}}\E\bigl[|V_j^{\qq}|_{S_Q}^Q\bigr]
	} \qquad &
	\notag \\ &
	\leq2^{Q-1}(3^dQ)^Q3^{-\frac{Qd}{2}h}
	e^{Q\widehat\Delta_h^{\qq}(m)}\mathcal P_{\qq}(m;n)
	+C\bigl(e^{Q\widehat\Delta_h^{\qq}(m)}-1\bigr)\,.
\end{align*}
We next estimate the mean terms and the history from scale~$n$. Applying the trace inequality from Step~1 after changing the normalization, and using~$1+\tr(P_{m,m+h}^{\qq}-\Itwod)\leq e^{\Delta_{m,m+h}^{\qq}}$, gives
\begin{equation*}
	\meanpenalty_Q(P_{j,m+h}^{\qq})
	\leq e^{Q\Delta_{m,m+h}^{\qq}}\meanpenalty_Q(P_{j,m}^{\qq})
	+e^{Q\Delta_{m,m+h}^{\qq}}-1\,.
\end{equation*}
Thus
\begin{equation*}
	\sum_{j=n}^{m-1}3^{-\frac14(1-\gamma)(m+h-1-j)}\meanpenalty_Q(P_{j,m+h}^{\qq})\leq3^{-\frac14(1-\gamma)h}e^{Q\Delta_{m,m+h}^{\qq}}\history_{\qq}^{\rm mean}(m;n)+C\bigl(e^{Q\Delta_{m,m+h}^{\qq}}-1\bigr)\,.
\end{equation*}
The same inequality gives
\begin{equation*}
	1+\meanpenalty_Q(P_{n,m+h}^{\qq})
	\leq e^{Q\Delta_{m,m+h}^{\qq}}
	\bigl(1+\meanpenalty_Q(P_{n,m}^{\qq})\bigr)\,.
\end{equation*}
After multiplying by the weight, the first term in~$\mathcal P_{\qq}(m+h;n)$ is at most~$3^{-\frac14(1-\gamma)h}e^{Q\Delta_{m,m+h}^{\qq}}$ times the first term in~$\mathcal P_{\qq}(m;n)$. For~$m\leq j<m+h$, the eigenvalues of~$P_{j,m+h}^{\qq}$ are at least one, and hence
\begin{equation*}
	1+\tr(P_{j,m+h}^{\qq}-\Itwod)
	\leq\det P_{j,m+h}^{\qq}
	=e^{\Delta_{j,m+h}^{\qq}}
	\leq e^{\widehat\Delta_h^{\qq}(m)}\,.
\end{equation*}
Therefore
\begin{equation*}
	\sum_{j=m}^{m+h-1}3^{-\frac14(1-\gamma)(m+h-1-j)}
	\meanpenalty_Q(P_{j,m+h}^{\qq})
	\leq C\bigl(e^{Q\widehat\Delta_h^{\qq}(m)}-1\bigr)\,.
\end{equation*}
For~$n<j\leq m$, additivity gives the factor~$3^{-\frac14(1-\gamma)h}e^{Q\Delta_{m,m+h}^{\qq}}$. Adding the new fluctuations, the mean terms, and the initial-history term proves the first inequality in~\eqref{e.fixed.geometry.profile.contraction.reduction}. The coefficient bounds proved above give its second inequality, and hence~\eqref{e.fixed.geometry.profile.contraction}.

\smallskip

We now include the drift. We prove the following bound for every~$m\geq j_*$ and integer~$L\geq1$, so that it will also apply to the error on a fresh grid in Step~3:
\begin{equation}
	D_{\qq,j_*}(m+L)\leq3^{-\frac18(1-\gamma)L}e^{\Delta_{m,m+L}^{\qq}}D_{\qq,j_*}(m)+e^{\Delta_{m,m+L}^{\qq}}-1\,.
	\label{e.fixed.geometry.drift.advance}
\end{equation}
We split the increments at~$m$. Changing normalization multiplies the trace of each positive increment with~$j\leq m$ by at most~$e^{\Delta_{m,m+L}^{\qq}}$, while the increments above~$m$ telescope. Thus
\begin{equation*}
	\sum_{j=j_*+1}^{m}3^{-\frac18(1-\gamma)(m+L-j)}\tr(P_{j-1,m+L}^{\qq}-P_{j,m+L}^{\qq})\leq3^{-\frac18(1-\gamma)L}e^{\Delta_{m,m+L}^{\qq}}D_{\qq,j_*}(m)\,,
\end{equation*}
and
\begin{equation*}
	\sum_{j=m+1}^{m+L}3^{-\frac18(1-\gamma)(m+L-j)}\tr(P_{j-1,m+L}^{\qq}-P_{j,m+L}^{\qq})\leq\tr(P_{m,m+L}^{\qq}-\Itwod)\leq e^{\Delta_{m,m+L}^{\qq}}-1\,.
\end{equation*}
Adding these two estimates proves~\eqref{e.fixed.geometry.drift.advance}. With~$L=h$, the bound~$3^{-\frac18(1-\gamma)h}<\frac18$ proved above and~$\Delta_{m,m+h}^{\qq}\leq\widehat\Delta_h^{\qq}(m)$ give
\begin{equation*}
	D_{\qq,j_*}(m+h)\leq\frac18e^{Q\widehat\Delta_h^{\qq}(m)}D_{\qq,j_*}(m)+C\bigl(e^{Q\widehat\Delta_h^{\qq}(m)}-1\bigr)\,.
\end{equation*}
Adding this to~\eqref{e.fixed.geometry.profile.contraction} proves~\eqref{e.fixed.geometry.synchronized.propagation}.

\smallskip

\emph{Step 3: Propagation from a small input.} We prove separate error estimates for continuing and fresh inputs, then add the drift. Fix~$L\geq1$. For a continuing input, assume~$m\geq n+h$ and~$\mathcal P_{\qq}(m;n)\leq1$. We will show
\begin{equation}
	\mathcal P_{\qq}(m+L;n)
	\leq C L\bigl(\mathcal P_{\qq}(m;n)+e^{Q\Delta_{m,m+L}^{\qq}}-1\bigr)\,.
	\label{e.fixed.geometry.profile.fixed.span}
\end{equation}
We again separate the scales up to~$m$ from the new mean and fluctuation terms. The initial-history term in~$\mathcal P_{\qq}(m+L;n)$ is at most~$3^{-\frac14(1-\gamma)L}e^{Q\Delta_{m,m+L}^{\qq}}$ times the corresponding term in~$\mathcal P_{\qq}(m;n)$. The fluctuation sum over~$n<j\leq m$ is exactly the same factor times the corresponding sum in~$\mathcal P_{\qq}(m;n)$. The mean-history contribution satisfies
\begin{equation*}
	\sum_{j=n}^{m-1}3^{-\frac14(1-\gamma)(m+L-1-j)}\meanpenalty_Q(P_{j,m+L}^{\qq})\leq3^{-\frac14(1-\gamma)L}e^{Q\Delta_{m,m+L}^{\qq}}\history_{\qq}^{\rm mean}(m;n)+C\bigl(e^{Q\Delta_{m,m+L}^{\qq}}-1\bigr)\,.
\end{equation*}
Since~$\mathcal P_{\qq}(m;n)\leq1$, we have
\begin{equation*}
	e^{Q\Delta_{m,m+L}^{\qq}}\mathcal P_{\qq}(m;n)
	\leq\mathcal P_{\qq}(m;n)+e^{Q\Delta_{m,m+L}^{\qq}}-1\,.
\end{equation*}
Thus these terms are bounded by~$\mathcal P_{\qq}(m;n)+C\bigl(e^{Q\Delta_{m,m+L}^{\qq}}-1\bigr)$. For~$m\leq j<m+L$,
\begin{equation*}
	\meanpenalty_Q(P_{j,m+L}^{\qq})
	\leq e^{Q\Delta_{j,m+L}^{\qq}}-1
	\leq e^{Q\Delta_{m,m+L}^{\qq}}-1\,.
\end{equation*}
Since their weights are at most one and there are~$L$ such indices,
\begin{equation*}
	\sum_{j=m}^{m+L-1}3^{-\frac14(1-\gamma)(m+L-1-j)}
	\meanpenalty_Q(P_{j,m+L}^{\qq})
	\leq C L\bigl(e^{Q\Delta_{m,m+L}^{\qq}}-1\bigr)\,.
\end{equation*}
It remains to bound the new fluctuation terms. We will prove
\begin{equation}
	\sum_{s=m+1}^{m+L}3^{-\frac14(1-\gamma)(m+L-s)}e^{Q\Delta_{s,m+L}^{\qq}}\E\bigl[|V_s^{\qq}|_{S_Q}^Q\bigr]\leq C L\bigl(\mathcal P_{\qq}(m;n)+e^{Q\Delta_{m,m+L}^{\qq}}-1\bigr)\,.
	\label{e.fixed.geometry.new.fluctuations}
\end{equation}
We use the parent--child recurrence with span~$2Q$, so that the constants are independent of the possibly larger~$h$. If~$m+1-2Q\leq r\leq m$, then~$r\geq n+1$ and~$m-r<2Q$. The scale-$r$ term in~$\mathcal P_{\qq}(m;n)$ therefore gives
\begin{equation*}
	e^{Q\Delta_{r,m}^{\qq}}\E\bigl[|V_r^{\qq}|_{S_Q}^Q\bigr]
	\leq C\mathcal P_{\qq}(m;n)\,.
\end{equation*}
For~$r<m$, the scale-$r$ term in~$\history_{\qq}^{\rm mean}(m;n)$ gives
\begin{equation*}
	\meanpenalty_Q(P_{r,m}^{\qq})
	\leq C\mathcal P_{\qq}(m;n)\,.
\end{equation*}
Since~$\mathcal P_{\qq}(m;n)\leq1$, the inequalities~$\Delta_{r,m}^{\qq}\leq\tr(P_{r,m}^{\qq}-\Itwod)$ and comparison of~$e^{Qt}-1$ with~$(1+t)^Q-1$ give
\begin{equation*}
	e^{Q\Delta_{r,m}^{\qq}}-1
	\leq C\mathcal P_{\qq}(m;n)\qquad(r<m)\,.
\end{equation*}
At~$r=m$, the first estimate follows from the scale-$m$ term, and the second left side is zero. For~$m<s\leq m+L$, apply~\eqref{e.fixed.geometry.parent.child.powered} with span~$2Q$ and~$j=s-2Q$, and multiply by~$e^{Q\Delta_{s,m+L}^{\qq}}$:
\begin{equation*}
	e^{Q\Delta_{s,m+L}^{\qq}}\E\bigl[|V_s^{\qq}|_{S_Q}^Q\bigr]\leq2^{Q-1}(3^dQ)^Q3^{-dQ^2}e^{Q\Delta_{s-2Q,m+L}^{\qq}}\E\bigl[|V_{s-2Q}^{\qq}|_{S_Q}^Q\bigr]+C\bigl(e^{Q\Delta_{s-2Q,m+L}^{\qq}}-1\bigr)\,.
\end{equation*}
If~$s-2Q\leq m$, additivity gives
\begin{equation*}
	e^{Q\Delta_{s-2Q,m+L}^{\qq}}\E\bigl[|V_{s-2Q}^{\qq}|_{S_Q}^Q\bigr]=e^{Q\Delta_{m,m+L}^{\qq}}e^{Q\Delta_{s-2Q,m}^{\qq}}\E\bigl[|V_{s-2Q}^{\qq}|_{S_Q}^Q\bigr]\,.
\end{equation*}
Moreover,
\begin{equation*}
	e^{Q\Delta_{s-2Q,m+L}^{\qq}}-1=e^{Q\Delta_{m,m+L}^{\qq}}\bigl(e^{Q\Delta_{s-2Q,m}^{\qq}}-1\bigr)+e^{Q\Delta_{m,m+L}^{\qq}}-1\,.
\end{equation*}
The preceding two bounds and~$e^{Q\Delta_{m,m+L}^{\qq}}\mathcal P_{\qq}(m;n)\leq\mathcal P_{\qq}(m;n)+e^{Q\Delta_{m,m+L}^{\qq}}-1$ yield
\begin{equation*}
	e^{Q\Delta_{s,m+L}^{\qq}}\E\bigl[|V_s^{\qq}|_{S_Q}^Q\bigr]
	\leq C\mathcal P_{\qq}(m;n)
	+C\bigl(e^{Q\Delta_{m,m+L}^{\qq}}-1\bigr)\,.
\end{equation*}
If~$s-2Q>m$, then~$e^{Q\Delta_{s-2Q,m+L}^{\qq}}-1\leq e^{Q\Delta_{m,m+L}^{\qq}}-1$, and the second contraction inequality gives
\begin{equation*}
	2^{Q-1}(3^dQ)^Q3^{-dQ^2}\leq\frac18\,.
\end{equation*}
Induction along each residue class modulo~$2Q$ gives the same bound for every~$m<s\leq m+L$. Summing over~$m<s\leq m+L$ proves~\eqref{e.fixed.geometry.new.fluctuations}. Combining this with the bounds for the mean terms and the scales up to~$m$ proves~\eqref{e.fixed.geometry.profile.fixed.span}.

\smallskip

For a fresh input, let~$m=n$ and~$\history_{\qq}(n)\leq1$. Since~$\mathcal P_{\qq}(n;n)=\history_{\qq}(n)$, the corresponding estimate is
\begin{equation}
	\mathcal P_{\qq}(n+L;n)
	\leq C L\bigl(\history_{\qq}(n)+e^{Q\Delta_{n,n+L}^{\qq}}-1\bigr)\,.
	\label{e.fixed.geometry.profile.startup}
\end{equation}
The mean comparison in Step~2 and~$e^x a\leq a+e^x-1$ for~$0\leq a\leq1$ give
\begin{equation*}
	3^{-\frac14(1-\gamma)L}\bigl(1+\meanpenalty_Q(P_{n,n+L}^{\qq})\bigr)\history_{\qq}(n)\leq e^{Q\Delta_{n,n+L}^{\qq}}\history_{\qq}(n)\leq\history_{\qq}(n)+e^{Q\Delta_{n,n+L}^{\qq}}-1\,.
\end{equation*}
For~$n\leq j<n+L$,
\begin{equation*}
	\meanpenalty_Q(P_{j,n+L}^{\qq})
	\leq e^{Q\Delta_{j,n+L}^{\qq}}-1
	\leq e^{Q\Delta_{n,n+L}^{\qq}}-1\,.
\end{equation*}
Their weights are at most one, so their sum is bounded by~$L\bigl(e^{Q\Delta_{n,n+L}^{\qq}}-1\bigr)$. For~$1\leq r\leq L$, apply~\eqref{e.fixed.geometry.parent.child.powered} with span~$r$ and multiply by~$e^{Q\Delta_{n+r,n+L}^{\qq}}$. Additivity gives
\begin{equation*}
	e^{Q\Delta_{n+r,n+L}^{\qq}}\E\bigl[|V_{n+r}^{\qq}|_{S_Q}^Q\bigr]\leq C L\biggl(e^{Q\Delta_{n,n+L}^{\qq}}\E\bigl[|V_n^{\qq}|_{S_Q}^Q\bigr]+e^{Q\Delta_{n,n+L}^{\qq}}-1\biggr)\,.
\end{equation*}
At scale~$n$, comparison of the Schatten and operator norms gives
\begin{equation*}
	\E\bigl[|V_n^{\qq}|_{S_Q}^Q\bigr]\leq2d\history_{\qq}^{\rm fluc}(n)\leq2d\history_{\qq}(n)\,.
\end{equation*}
Applying~$e^x a\leq a+e^x-1$ for~$0\leq a\leq1$ to~$a=\history_{\qq}(n)$ therefore yields
\begin{equation*}
	e^{Q\Delta_{n+r,n+L}^{\qq}}
	\E\bigl[|V_{n+r}^{\qq}|_{S_Q}^Q\bigr]
	\leq C L\bigl(\history_{\qq}(n)+e^{Q\Delta_{n,n+L}^{\qq}}-1\bigr)
	\qquad(1\leq r\leq L)\,.
\end{equation*}
The constant~$C$ is uniform in~$1\leq r\leq L$. Summing gives
\begin{equation*}
	\sum_{r=1}^{L}3^{-\frac14(1-\gamma)(L-r)}e^{Q\Delta_{n+r,n+L}^{\qq}}\E\bigl[|V_{n+r}^{\qq}|_{S_Q}^Q\bigr]\leq C L\bigl(\history_{\qq}(n)+e^{Q\Delta_{n,n+L}^{\qq}}-1\bigr)\,.
\end{equation*}
Together with the preceding estimates, this proves~\eqref{e.fixed.geometry.profile.startup}.

\smallskip

It remains to include the drift in these bounds. If~$\mathcal P_{\qq}(m;n)+D_{\qq,j_*}(m)\leq1$, then~$D_{\qq,j_*}(m)\leq1$. Applying~$e^x a\leq a+e^x-1$ in~\eqref{e.fixed.geometry.drift.advance} gives
\begin{equation*}
	D_{\qq,j_*}(m+L)\leq D_{\qq,j_*}(m)+2\bigl(e^{\Delta_{m,m+L}^{\qq}}-1\bigr)\,.
\end{equation*}
For~$m\geq n+h$, add this estimate to~\eqref{e.fixed.geometry.profile.fixed.span}; for~$m=n$, add it to~\eqref{e.fixed.geometry.profile.startup} and use~$\mathcal P_{\qq}(n;n)=\history_{\qq}(n)$. Since~$e^{\Delta_{m,m+L}^{\qq}}-1\leq e^{Q\Delta_{m,m+L}^{\qq}}-1$, both cases give~\eqref{e.fixed.geometry.fixed.span.propagation}.

\smallskip

\emph{Step 4: Summation of determinant losses.} Let~$m_0\geq j_*+h$ and~$K\geq1$. We expand each loss into one-scale increments:
\begin{equation*}
	\sum_{k=0}^{K-1}\widehat\Delta_h^{\qq}(m_0+kh)
	=\sum_{a=m_0+1}^{m_0+Kh}\sum_{r=a-h}^{a-1}\Delta_{r,r+1}^{\qq}\,.
\end{equation*}
The smallest index is at least~$j_*$, and each~$\Delta_{r,r+1}^{\qq}$ occurs at most~$h$ times. Hence
\begin{equation*}
	\sum_{k=0}^{K-1}\widehat\Delta_h^{\qq}(m_0+kh)
	\leq h\sum_{r=m_0+1-h}^{m_0+Kh-1}\Delta_{r,r+1}^{\qq}
	=h\Delta_{m_0+1-h,m_0+Kh}^{\qq}\,.
\end{equation*}
This proves~\eqref{e.fixed.geometry.synchronized.multiplicity}.
\end{proof}

\section{Geometry update}
\label{s.geometry.transport}
Section~\ref{s.fixed.geometry} shows how averaging reduces the errors on a fixed grid. We now ask whether these gains survive when we change the shapes of the cubes to better suit the medium. The cubes of the two grids do not fit together, so their estimates cannot be compared directly. We show that the transfer is possible when the change of shape is bounded, the old-grid errors are sufficiently small, and the annealed matrix changes little over the scales used in the comparison. Changing the grid may enlarge both the error and the drift. The point is to keep their sum small enough to start on the new grid and resume the fixed-grid argument, as required in Case~3 of Proposition~\ref{p.scale.selection}.

\paragraph{Comparing the means.}
The geometric idea is to fill a new-grid cube with old-grid cubes, using large cubes in the interior and progressively smaller ones near the boundary. Subadditivity bounds the coarse-grained matrix of the whole cube by the volume-weighted average of the matrices on these pieces. Reversing the roles of the grids gives the other side of the comparison. Together, these constructions place the new-grid mean between old-grid means at nearby scales, up to boundary errors. Small determinant loss ensures that those means are close. The smaller boundary cubes require estimates from earlier scales: the drift controls the changes of their means, and at the finest scales their shrinking volume outweighs the deterioration of the coarse ellipticity bound. As in Subsection~\ref{ss.source.estimate}, this uses the restriction~$\gamma<1$ on the small-scale growth exponent in assumption~\ref{a.coarse.ellipticity}. Lemma~\ref{l.two.grid.whitney} constructs the partitions, and Proposition~\ref{p.successful.short.bridge} proves the resulting comparison of the annealed matrices. Only the relative change of shape enters the partition constants; Lemma~\ref{l.projective.step} keeps these bounded even when the old grid has large geometric eccentricity.

\paragraph{Transferring the errors.}
Closeness of the means lets us change the matrix against which the errors are measured, but we must still control the fluctuations and earlier errors on the new grid. Here we use the same principle as in Section~\ref{s.fixed.geometry}: compare the coarse matrix of a cube with the average of the matrices on its pieces. The fluctuations of this average and the difference of the two means together control the new-grid error. Finite-range dependence reduces the fluctuations in averages over scales above the old initialization scale; at and below that scale, we use the estimates retained in the history. The mean history and drift account for the remaining discrepancies. Proposition~\ref{p.two.grid.transport} puts these estimates together. The transfer has a cost, depending on how far we advance in scale, so we begin with a smaller error than we need on the new grid. Restarting does not create smallness or erase the past: it carries the earlier estimates to the new geometry. Subsection~\ref{s.scale.selection.proof} combines this transfer with fixed-grid propagation and checks the grid and error conditions needed to continue the iteration.

\smallskip

For~$\m\in\Rpos$, set
\begin{equation*}
	\mathfrak e(\m)\coloneqq\bigl(|\m|\,|\m^{-1}|\bigr)^{\nf12}\,.
\end{equation*}
For two rounded geometries, set
\begin{equation*}
	K(\qq,\qq')\coloneqq\left(1+|\qq^{-1}\qq'|+|(\qq')^{-1}\qq|\right)^{2d}\,.
\end{equation*}

The geometric eccentricity of the current grid is measured by~$\mathfrak e(\m)$, which describes its distortion relative to Euclidean coordinates up to the fixed rounding constants. This depends on the chosen grid, whereas the reference ellipticity ratio~$\Pi$ is fixed by the coefficient bound. In contrast,~$K$ measures the distortion between two rounded grids and remains bounded for a bounded projective change, even when both grids are highly eccentric.

\subsection{Whitney partitions}

Subadditivity bounds the coarse-grained matrix of a cube by the volume-weighted average of the matrices on its partition elements. To obtain comparisons in both directions, the first construction below fills a new-grid cube with old-grid cubes, and the second packs new-grid cubes into an old-grid cube and fills the remaining boundary region. The largest permitted cubes fill the bulk; smaller cubes are needed only near the boundary. Their decreasing volume fractions will make the sums over smaller scales converge.

\begin{lemma}[Whitney partitions between adapted grids]
\label{l.two.grid.whitney}
For every~$K_0\geq1$, there exists~$C(d,K_0)<\infty$ with the following property. Let~$\m,\m'\in\Rpos$, let~$\qq=\mathcal Q(\m)$ and~$\qq'=\mathcal Q(\m')$ satisfy~$K(\qq,\qq')\leq K_0$, and let~$j\in\Z$ and~$\ell\in\N$ with~$\ell\geq1$. Then the following constructions and estimates hold.

\smallskip

\noindent (i) Let~$y\in3^j\qq'\Zd$. Select the maximal adapted cubes of the~$\qq$-grid contained in~$y+\cus_j^{\qq'}$ whose scales are at most~$j-\ell$. Denote their centers at scale~$r$ by~$\mathcal Z_r(y+\cus_j^{\qq'})$. These cubes partition~$y+\cus_j^{\qq'}$ up to a null set, and for $r<j-\ell$, 
\begin{align}
	\frac{|\cus_r^{\qq}|}{|\cus_j^{\qq'}|}\leq C3^{-d(j-r)}, \qquad
	\#\mathcal Z_{j-\ell}(y+\cus_j^{\qq'})\leq C3^{d\ell}\,,\qquad \#\mathcal Z_r(y+\cus_j^{\qq'})\leq C3^{(d-1)(j-r)}\,.
	\label{e.two.grid.whitney.counts}
\end{align}
In particular,
\begin{equation}
	\sum_{r\leq j-\ell}\sum_{z\in\mathcal Z_r(y+\cus_j^{\qq'})}\frac{|\cus_r^{\qq}|}{|\cus_j^{\qq'}|}=1\qand \sum_{z\in\mathcal Z_r(y+\cus_j^{\qq'})}\frac{|\cus_r^{\qq}|}{|\cus_j^{\qq'}|}\leq C3^{r-j}\quad(r<j-\ell)\,.
	\label{e.two.grid.whitney.volumes}
\end{equation}

\smallskip

\noindent (ii) Let~$y\in3^j\qq\Zd$. First take all the adapted cubes~$z+\cus_{j-\ell}^{\qq'}$, with~$z\in3^{j-\ell}\qq'\Zd$, contained in~$y+\cus_j^{\qq}$. In the uncovered part, select the maximal adapted cubes of the~$\qq$-grid whose scales are at most~$j-\ell$, and again denote their centers by~$\mathcal Z_r(y+\cus_j^{\qq})$. Together the two families partition~$y+\cus_j^{\qq}$ up to a null set, and
\begin{equation}
	\sum_{z\in\mathcal Z_r(y+\cus_j^{\qq})}\frac{| \cus_r^{\qq}|}{|\cus_j^{\qq}|}\leq C3^{r-j}\qquad(r\leq j-\ell)\,.
	\label{e.two.grid.whitney.reverse.volumes}
\end{equation}
\end{lemma}

\begin{proof}
We work in~$\qq$-coordinates, where the old-grid cubes are triadic cubes. All volume ratios are unchanged, and~$K(\qq,\qq')\leq K_0$ bounds the distortion of the other grid. Constants denoted by~$C$ depend only on~$d$ and~$K_0$.

\smallskip

For~(i), select all cubes of scale~$j-\ell$ contained in~$y+\cus_j^{\qq'}$, then continue at smaller scales in the uncovered part. Disjointness gives the bound for~$\#\mathcal Z_{j-\ell}(y+\cus_j^{\qq'})$. A selected cube of scale~$r<j-\ell$ has a parent meeting the complement of~$y+\cus_j^{\qq'}$, so it lies within distance~$C3^r$ of the boundary. In these coordinates the target cube has~$2d$ flat faces and diameter at most~$C3^j$, so this neighborhood has volume at most~$C3^r3^{(d-1)j}$. Dividing by the volume of a scale-$r$ cube gives~\eqref{e.two.grid.whitney.counts}, and dividing by the volume of the target cube gives the second estimate in~\eqref{e.two.grid.whitney.volumes}. Every interior point outside the grid boundaries belongs to a sufficiently small cube contained in~$y+\cus_j^{\qq'}$. Thus the selected cubes cover~$y+\cus_j^{\qq'}$ up to a null set, proving the first equality in~\eqref{e.two.grid.whitney.volumes}.

\smallskip

For~(ii), every point of the uncovered part lies within distance~$C3^{j-\ell}$ of the outer boundary: otherwise the new-grid cube containing it would be contained in~$y+\cus_j^{\qq}$ and would have been taken. Hence the uncovered volume is at most~$C3^{j-\ell}3^{(d-1)j}$, which proves~\eqref{e.two.grid.whitney.reverse.volumes} at scale~$j-\ell$.

\smallskip

We also need the boundary of the uncovered part at smaller scales. A packed cube with an exposed face lies in a strip of width~$C3^{j-\ell}$ along the outer boundary. The packed cubes are disjoint and each has volume at least~$C^{-1}3^{d(j-\ell)}$, so at most~$C3^{(d-1)\ell}$ of them have exposed faces. For~$r<j-\ell$, the width-$C3^r$ neighborhood of each such face has volume at most~$C3^r3^{(d-1)(j-\ell)}$. Summing over the exposed faces and the outer boundary gives at most~$C3^r3^{(d-1)j}$. Every selected old-grid cube of scale~$r<j-\ell$ has a parent meeting this boundary, so division by~$|\cus_j^{\qq}|$ proves~\eqref{e.two.grid.whitney.reverse.volumes}. If no new-grid cube is contained in~$y+\cus_j^{\qq}$, the old-grid cubes of scale~$j-\ell$ already partition~$y+\cus_j^{\qq}$, and there are no smaller selected cubes. In either case, the construction covers the uncovered part up to a null set.
\end{proof}

\smallskip

We record how the adapted cube bound applies to these partitions. In construction~(i), assume that~$y+\cus_j^{\qq'}\subseteq\cu_{2j_*}$ and~$j-\ell\geq j_*$. Then~\eqref{e.source.adapted.bound} and the volume bound give
\begin{equation}
	\sum_{r<j_*}\sum_{z\in\mathcal Z_r(y+\cus_j^{\qq'})}\frac{|\cus_r^{\qq}|}{|\cus_j^{\qq'}|}\bfA(z+\cus_r^{\qq})\leq C(d,\gamma,K_0)\mathfrak e(\m)\RSZ3^{-(j-j_*)}\bfE\,.
	\label{e.two.grid.whitney.source}
\end{equation}
The same estimate holds for the old-grid cubes in construction~(ii), with~$\mathcal Z_r(y+\cus_j^{\qq})$ and denominator~$|\cus_j^{\qq}|$, provided~$y+\cus_j^{\qq}\subseteq\cu_{2j_*}$ and~$j-\ell\geq j_*$. Below~$j_*$, the adapted cube bound grows as the cubes shrink, but their relative volume decreases faster. In both constructions the resulting series is
\begin{equation}
	\sum_{r<j_*}3^{r-j}3^{\gamma(j_*-r)}=3^{-(j-j_*)}\sum_{s=1}^{\infty}3^{-(1-\gamma)s}\leq C(d,\gamma)3^{-(j-j_*)}\,.
	\label{e.two.grid.whitney.fine.sum}
\end{equation}
Since~$\gamma<1$ and~$\|\RSZ\|_{L^Q}\leq2$, the matrix sum converges in~$L^Q(S_Q)$. The same argument gives convergence for any fixed upper scale, by treating the finitely many larger scales separately. Subadditivity on these countable partitions follows by finite exhaustion: the remaining set shrinks to a null set, and its contribution tends to zero by the pointwise-average bound for~$\bfA$ and local integrability of the coefficient block. The stronger adapted cube bound~\eqref{e.two.grid.whitney.source} controls the weighted sums used below.

\smallskip

\subsection{Comparison of the annealed matrices}

The comparison uses the old grid at scales~$n$ and~$n+2L$, with the new grid at the intermediate scale~$n+L$. Filling a new-grid cube with old-grid cubes starting at scale~$n$ gives an upper bound by~$\bfAhom_{n,\qq}$, up to boundary errors. Packing new-grid cubes into an old-grid cube at scale~$n+2L$ gives the reverse bound against~$\bfAhom_{n+2L,\qq}$. Small determinant loss makes these two old-grid means close. The two comparisons then place the new mean close to~$\bfAhom_{n+2L,\qq}$, as stated in Proposition~\ref{p.successful.short.bridge}. The intermediate scale leaves room for both partitions.

\smallskip

The contribution of the finest boundary cubes is bounded in~\eqref{e.two.grid.whitney.source} by a scalar multiple of the reference matrix~$\bfE$ from assumption~\ref{a.coarse.ellipticity}. We need a bound by a small multiple of the annealed matrix instead, because this is the matrix against which we measure the errors. Thus we must bound~$\bfE$ from above by a multiple of~$\bfAhom_{s,\qq}$, as in~\eqref{e.two.grid.source.normalization} below. The adapted cube bound gives the opposite inequality. Comparing a matrix with its dual will allow us to reverse this comparison, with a factor depending on~$\Pi$ and the grid geometry.

For a symmetric positive definite~$2d\times2d$ matrix~$F$, its dual is~$F_*\coloneqq\mathbf RF^{-1}\mathbf R$, where~$\mathbf R$ exchanges the two~$d$-dimensional components. Inversion reverses matrix order, so an upper bound for~$F$ gives a lower bound for~$F_*$. For coarse-grained matrices we also have~$\bfA_*(U)\leq\bfA(U)$ for each coefficient field, by~\cite[(2.22)]{AK.HC}, with~$\bfA_*(U)$ the dual coarse-grained matrix of~\cite{AK.HC}. A lower bound for the dual is therefore a lower bound for the original matrix as well.

The same argument applies to annealed matrices, but we must distinguish the dual of the mean,~$\bfAhom_*(U)=\mathbf R\bfAhom(U)^{-1}\mathbf R$, from the mean of the dual,~$\E[\bfA_*(U)]$. They need not agree. The variational formula~$v\cdot F^{-1}v=\sup_w(2v\cdot w-w\cdot Fw)$ gives~$(\E[F])^{-1}\leq\E[F^{-1}]$: taking the supremum after expectation is no larger than taking it before expectation. Combining this inequality with~$\bfA_*(U)\leq\bfA(U)$ gives
\begin{equation}
	\bfAhom_*(U)\leq\E[\bfA_*(U)]\leq\bfAhom(U)\,.
	\label{e.matrix.annealed.sharp.order}
\end{equation}
The reference matrix has the same order property. Write~$\mathbf{E}_{*,0}\coloneqq\mathbf R\bfE^{-1}\mathbf R$. Almost surely, we can choose an integer~$m$ with~$3^m\geq\S(0)$, so that~\eqref{e.coarse.ellipticity} gives~$\bfA(\cu_m)\leq\bfE$. Taking duals reverses this inequality, and hence~$\mathbf{E}_{*,0}\leq\bfA_*(\cu_m)\leq\bfA(\cu_m)\leq\bfE$. In particular,~$\mathbf{E}_{*,0}\leq\bfE$.

\smallskip

We next bound~$F$ from above by a multiple of~$F_*$. We keep a factor tending to one as the contrast tends to one, since this estimate will also be used to close the duality gap in Subsection~\ref{s.response.transfer}. Write~$F$ in the block form~\eqref{e.annealed.schur}, with matrices~$S,S_*,K$. Assume~$F_*\leq F$, and fix a skew matrix~$h$. Set
\begin{equation*}
	\theta=\bigl|S_*^{-\nf12}\bigl(S+(K-h)^tS_*^{-1}(K-h)\bigr)S_*^{-\nf12}\bigr|\,.
\end{equation*}
We claim that
\begin{equation}
	F\leq\bigl(1+6(\theta-1)\bigr)F_*\,.
	\label{e.matrix.block.comparison}
\end{equation}
Write~$T=S_*^{-1}$,~$L=K+K^t$, and~$|v|_T^2=v\cdot Tv$. The lower-right blocks in~$F_*\leq F$ give~$S^{-1}\leq S_*^{-1}$, hence~$S_*\leq S$. The definition of~$\theta$ then gives~$1\leq\theta$,~$S\leq\theta S_*$, and~$(K-h)^tT(K-h)\leq(\theta-1)S_*$. The last inequality also holds with~$K-h$ replaced by its transpose: after conjugation by~$S_*^{-\nf12}$, the two products have the same operator norm. Thus
\begin{equation*}
	|Lx|_T^2\leq4(\theta-1)x\cdot S_*x
	\qand |w|_T^2\leq\theta|w|_{S^{-1}}^2\,.
\end{equation*}
For the quadratic forms~$q_F(x,y)=\binom{x}{y}\cdot F\binom{x}{y}$ and~$q_{F_*}$, put~$w=y+K^tx$. Expanding the block quadratic forms gives
\begin{equation*}
	q_F(x,y)=x\cdot Sx+|w-Lx|_T^2
	\qand q_{F_*}(x,y)=x\cdot S_*x+|w|_{S^{-1}}^2\,.
\end{equation*}
To control the cross term without a fixed multiplicative loss, use the order~$F_*\leq F$ once more, now at~$(x,-5w-K^tx)$. This gives
\begin{equation*}
	x\cdot S_*x+25|w|_{S^{-1}}^2
	\leq x\cdot Sx+25|w|_T^2+10w\cdot TLx+|Lx|_T^2\,.
\end{equation*}
This controls the cross term in~$q_F$ without a loss at~$\theta=1$:
\begin{equation*}
	q_F(x,y)
	\leq\tfrac65x\cdot Sx-\tfrac15x\cdot S_*x
	+6|w|_T^2-5|w|_{S^{-1}}^2+\tfrac65|Lx|_T^2\leq(6\theta-5)\bigl(x\cdot S_*x+|w|_{S^{-1}}^2\bigr)\,.
\end{equation*}
This proves~\eqref{e.matrix.block.comparison}. Applying it to~$\bfE$ and minimizing over~$h$ gives
\begin{equation}
	\bfE\leq6\Pi\mathbf{E}_{*,0}\,,
	\label{e.matrix.reference.comparison}
\end{equation}
because~$\theta\leq|S_*^{-1}|\,|S+(K-h)^tS_*^{-1}(K-h)|$ for each~$h$. The minimum exists since the quadratic expression is coercive on the finite-dimensional space of skew matrices.

\smallskip

We will normalize~\eqref{e.two.grid.whitney.source} by annealed matrices on either grid. If~$s\geq j_*$ and~$\cus_s^{\qq}\subseteq\cu_{2j_*}$, then
\begin{equation}
	\bfE\leq C(d,\gamma)\Pi\mathfrak e(\m)\bfAhom_{s,\qq}\,.
	\label{e.two.grid.source.normalization}
\end{equation}
To prove this, take expectation in~\eqref{e.source.adapted.bound} and use~$\E[\RSZ]\leq2$. The order~$\bfAhom_*(U)\leq\bfAhom(U)$ then gives
\begin{equation*}
	\bfAhom_{s,\qq}\leq C(d,\gamma)\mathfrak e(\m)\bfE\qand
	\bfAhom_{s,\qq}\geq\bigl(C(d,\gamma)\mathfrak e(\m)\bigr)^{-1}\mathbf R\bfE^{-1}\mathbf R\,.
\end{equation*}
The reference-block comparison~$\bfE\leq6\Pi\mathbf R\bfE^{-1}\mathbf R$ proves~\eqref{e.two.grid.source.normalization}. Replacing~$(\m,\qq)$ by~$(\m',\qq')$ gives the same estimate on the other grid.

The interpolated update must also control the grid eccentricity and the relative distortion between the two rounded grids. The next lemma gives these geometric bounds.

\begin{lemma}[The projective step]
\label{l.projective.step}
Let~$\m,\m_*\in\Rpos$, let~$\varepsilon>0$, define~$\m_+$ by~\eqref{e.renormalization.geometry.update}, and set~$\qq=\mathcal Q(\m)$ and~$\qq_+=\mathcal Q(\m_+)$. Then
\begin{equation}
	d_{\rm pr}([\m],[\m_+])\leq\varepsilon\,.
	\label{e.projective.step}
\end{equation}
Moreover,
\begin{equation}
	\frac12\log\bigl(|\m_+|\,|\m_+^{-1}|\bigr)
	\leq\frac12\log\bigl(|\m|\,|\m^{-1}|\bigr)+\varepsilon\,.
	\label{e.projective.eccentricity}
\end{equation}
If~$0<\varepsilon\leq1$, then
\begin{equation}
	K(\qq,\qq_+)\leq C(d)\,.
	\label{e.projective.rounded.hop}
\end{equation}
\end{lemma}

\begin{proof}
Diagonalize~$\m^{-\nf12}\m_*\m^{-\nf12}$. The exponent~$\theta$ in~\eqref{e.renormalization.geometry.update} multiplies the logarithmic spread of its eigenvalues by~$\theta$, so~$d_{\rm pr}([\m],[\m_+])=\theta\,d_{\rm pr}([\m],[\m_*])\leq\varepsilon$ by the definition of~$\theta$. This proves~\eqref{e.projective.step}. The triangle inequality for~$d_{\rm pr}$ gives~\eqref{e.projective.eccentricity}. The rounded-grid comparison at scale~$j_*$, verified directly at the start of the proof of Proposition~\ref{p.two.grid.transport}, and~\eqref{e.projective.step} give~\eqref{e.projective.rounded.hop} when~$\varepsilon\leq1$.
\end{proof}

We now combine the geometric bounds with the adapted cube bound. Small error, drift, and determinant variation on the old grid then make the two annealed blocks close, as required to change the normalization of the errors.

\begin{proposition}[Comparison after a change of geometry]
\label{p.successful.short.bridge}
There exist~$L_0(d,\gamma)\in\N$ and~$c_0(d,\gamma)\in(0,1)$ such that, for every~$\sigma\in(0,1]$ and~$L\in\N$ satisfying
\begin{equation}
	L\geq L_0+\lceil\log_3(\sigma^{-1})\rceil\,,
	\label{e.bridge.length.input}
\end{equation}
there exists~$B_0(\sigma,L,d,\gamma)\geq1$ with the following property.

\smallskip

Let~$\m,\m_+\in\Rpos$, set~$\qq=\mathcal Q(\m)$ and~$\qq_+=\mathcal Q(\m_+)$, and let~$k,n\in\Z$ satisfy~$j_*\leq k\leq n$. Assume that
\begin{equation}
	\cus_{n+2L}^{\qq}\cup\cus_{n+L}^{\qq_+}\subseteq\cu_{2j_*}\,,
	\label{e.bridge.cubes.input}
\end{equation}
and
\begin{equation}
	\log\bigl(|\m|\,|\m^{-1}|\bigr)\leq\frac{c_0}{L}\bigl(k-j_*-\lceil B_0\log_3(2+\Pi)\rceil\bigr)\,.
	\label{e.bridge.eccentricity.input}
\end{equation}
Assume also that~$n\geq k+2Q$,
\begin{equation}
	d_{\rm pr}([\m],[\m_+])\leq1\,,
	\label{e.bridge.projective.input}
\end{equation}
and
\begin{equation}
	\mathcal P_{\qq}(n;k)+D_{\qq,j_*}(n)+\Delta_{n,n+2L}^{\qq}\leq c_0\sigma\,.
	\label{e.bridge.smallness.input}
\end{equation}
Then
\begin{equation}
	(1-\sigma)\bfAhom_{n+2L,\qq}\leq\bfAhom_{n+L,\qq_+}\leq(1+\sigma)\bfAhom_{n+2L,\qq}\,.
	\label{e.successful.short.near}
\end{equation}
\end{proposition}

\begin{proof}
We first compare the two annealed matrices by the Whitney partitions, then estimate the resulting sums. We keep~$c_0\in(0,1)$ unspecified until the end, when we choose~$c_0$ small,~$L_0$ large, and then~$B_0$ large. Constants denoted by~$C$ depend only on~$d$ and~$\gamma$, uniformly for~$0<c_0\leq1$.

\smallskip

\emph{Step 1: We prove~\eqref{e.bridge.upper.comparison} and~\eqref{e.bridge.lower.comparison}.} All the adapted cubes constructed below lie in~$\cu_{2j_*}$ by~\eqref{e.bridge.cubes.input}. We apply~\eqref{e.source.adapted.bound} with the same~$\RSZ$, taking expectations before estimating the annealed blocks.

\smallskip

The eccentricity bound~\eqref{e.bridge.eccentricity.input}, the distance bound~\eqref{e.bridge.projective.input}, and the triangle inequality for~$d_{\rm pr}$ give
\begin{equation}
	\mathfrak e(\m)\leq\exp\biggl(\frac{c_0}{L}\bigl(k-j_*-\lceil B_0\log_3(2+\Pi)\rceil\bigr)\biggr)\qand\mathfrak e(\m_+)\leq e\,\mathfrak e(\m)\,.
	\label{e.local.eccentricity.bounds}
\end{equation}
The rounded-grid comparison used in Lemma~\ref{l.projective.step} gives~$K(\qq,\qq_+)\leq K_0(d)$ from~\eqref{e.bridge.projective.input}. In Step~2, the upper and lower comparisons proved here will give
\begin{equation}
	\bigl|\bfAhom_{n+2L,\qq}^{-\nf12}\bfAhom_{n+L,\qq_+}\bfAhom_{n+2L,\qq}^{-\nf12}-\Itwod\bigr|\leq Cc_0\sigma+CK_0 3^{-L}+C(1+L)(2+\Pi)^{1-\frac1{16}(1-\gamma)B_0}\,.
	\label{e.successful.short.preliminary.comparison}
\end{equation}
The three terms account for the old-grid drift and determinant variation, the boundary volume, and the contribution of the scales below~$j_*$. At the end of the proof, we choose~$c_0$, then~$L_0$, and finally~$B_0$ to make each term at most~$\frac13\sigma$.

\smallskip

By~\eqref{e.bridge.smallness.input}, the determinant loss is at most~$c_0\sigma$. Set, only for this comparison,
\begin{equation*}
	\delta_0\coloneqq
	\exp\bigl(d^{-1}\Delta_{n,n+2L}^{\qq}\bigr)-1
	\leq Cc_0\sigma\,.
\end{equation*}
Since~$c_0\sigma\leq1$ and~$d\geq2$, we have~$\delta_0\leq1$. Applying~\eqref{e.fixed.geometry.drift.advance} with lengths~$L$ and~$2L$, and using~$D_{\qq,j_*}(n)\leq c_0\sigma$ and~$(1+x)^d-1\leq Cx$ for~$0\leq x\leq1$, gives
\begin{equation*}
	D_{\qq,j_*}(n+L)+D_{\qq,j_*}(n+2L)\leq Cc_0\sigma\,.
\end{equation*}

\smallskip

We first prove the upper comparison
\begin{align}
\lefteqn{ 
	\bfAhom_{n+L,\qq_+}-\bfAhom_{n,\qq}
} \qquad & 
\notag \\ & 
\leq CK_0\sum_{r=j_*}^{n}3^{r-n-L}\bfAhom_{r,\qq}
	+C\bigl(1+\Pi\bigl(\mathfrak e(\m)+\mathfrak e(\m_+)\bigr)^2\bigr)3^{-(1-\gamma)(n+L-j_*)}\bfAhom_{n+2L,\qq}\,.
	\label{e.bridge.upper.comparison}
\end{align}
Partition~$\cus_{n+L}^{\qq_+}$ into old-grid adapted cubes by Lemma~\ref{l.two.grid.whitney}(i), starting at scale~$n$. The scale-$n$ cubes have total relative volume at most one, so their expected contribution is at most~$\bfAhom_{n,\qq}$. By~\eqref{e.two.grid.whitney.volumes}, the smaller cubes at scale~$r\geq j_*$ have total relative volume at most~$CK_0 3^{r-n-L}$. Subadditivity and stationarity give the sum on the right side of~\eqref{e.bridge.upper.comparison}.

\smallskip

For~$r<j_*$, take expectation in~\eqref{e.two.grid.whitney.source} and use~$\E[\RSZ]\leq2$. By~\eqref{e.two.grid.source.normalization},~$\bfE\leq C\Pi\mathfrak e(\m)\bfAhom_{n+2L,\qq}$. The resulting bound is at most the last term in~\eqref{e.bridge.upper.comparison}. This proves the upper comparison.

\smallskip

For the lower comparison, we prove
\begin{align}
\lefteqn{
	\bfAhom_{n+2L,\qq}-\bfAhom_{n+L,\qq_+}
} \qquad & 
\notag \\ & 
	\leq CK_0\sum_{r=j_*}^{n+L}3^{r-n-2L}\bfAhom_{r,\qq}
	+C\bigl( 1+\Pi\bigl(\mathfrak e(\m)+\mathfrak e(\m_+)\bigr)^2\bigr)3^{-(1-\gamma)(n+2L-j_*)}\bfAhom_{n+2L,\qq}\,.
	\label{e.bridge.lower.comparison}
\end{align}
Apply Lemma~\ref{l.two.grid.whitney}(ii) to~$\cus_{n+2L}^{\qq}$ with upper scale~$n+L$. Thus we pack the contained new-grid cubes of scale~$n+L$ and partition the uncovered part by old-grid cubes. The bound~\eqref{e.two.grid.whitney.reverse.volumes} applies to every selected old-grid scale, including~$n+L$.

\smallskip

The packed new-grid cubes have total relative volume at most one, so their expected contribution is at most~$\bfAhom_{n+L,\qq_+}$. Subadditivity, stationarity for~$r\geq j_*$, and the same adapted cube bound for~$r<j_*$ give~\eqref{e.bridge.lower.comparison}.

\smallskip

\emph{Step 2: We estimate the right sides of~\eqref{e.bridge.upper.comparison} and~\eqref{e.bridge.lower.comparison} to prove~\eqref{e.successful.short.near}.} The boundary sums involve means at every earlier scale, not just the endpoints of the determinant test. Expanding these means into successive positive increments makes the sums accessible to the drift. For~$m=n+L$ and~$m=n+2L$, we have
\begin{equation}
	\sum_{r=j_*}^{m-L}3^{r-m}\bfAhom_{r,\qq}\leq C\bigl(3^{-L}+3^{-\left(1-\frac18(1-\gamma)\right)L}D_{\qq,j_*}(m)\bigr)\bfAhom_{m,\qq}\,.
	\label{e.bridge.boundary.sum}
\end{equation}
To see this, write~$\bfAhom_{r,\qq}=\bfAhom_{m,\qq}+\sum_{j=r+1}^{m}(\bfAhom_{j-1,\qq}-\bfAhom_{j,\qq})$. The constant part contributes at most~$C3^{-L}\bfAhom_{m,\qq}$. Each positive increment is bounded by~$\tr(P_{j-1,m}^{\qq}-P_{j,m}^{\qq})\bfAhom_{m,\qq}$, and its coefficient satisfies
\begin{equation*}
	\sum_{r=j_*}^{\min\{m-L,j-1\}}3^{r-m}\leq C3^{-\max\{L,m-j\}}\leq C3^{-\left(1-\frac18(1-\gamma)\right)L}3^{-\frac18(1-\gamma)(m-j)}\,.
\end{equation*}
Summing over~$j$ gives~\eqref{e.bridge.boundary.sum} by the definition of~$D_{\qq,j_*}(m)$.

\smallskip

The last terms in~\eqref{e.bridge.upper.comparison} and~\eqref{e.bridge.lower.comparison} are bounded by~$C\bfAhom_{n+2L,\qq}$ times the left side below. By~\eqref{e.local.eccentricity.bounds},
\begin{equation}
\begin{aligned}
	\lefteqn{
	\bigl(1+\Pi\bigl(\mathfrak e(\m)+\mathfrak e(\m_+)\bigr)^2\bigr) (1+n+L-j_*)3^{-\frac18(1-\gamma)(n-j_*)}
	} \quad &
	  \\  &
	\leq C(1+L)(2+\Pi)^{1-\frac1{16}(1-\gamma)B_0}\exp  \Bigl(-\bigl(k{-}j_*{-}\lceil B_0\log_3(2{+}\Pi)\rceil\bigr)\Bigl(\frac{1{-}\gamma} {16}\log3-\frac{2c_0}{L}\Bigr)\Bigr)\,.
\end{aligned}
		\label{e.bridge.source.bounds}
\end{equation}
Indeed,~$1+n+L-j_*\leq(1+L)(1+n-j_*)$, and weakening the decay exponent from~$\frac18(1-\gamma)$ to~$\frac1{16}(1-\gamma)$ absorbs~$1+n-j_*$. The squared grid eccentricities give the factor~$\nf{2c_0}{L}$ in the exponent.

\smallskip

Mean monotonicity and the determinant bound give
\begin{equation*}
	\bfAhom_{n+2L,\qq}\leq\bfAhom_{n+L,\qq}\leq\bfAhom_{n,\qq}\leq(1+\delta_0)^d\bfAhom_{n+2L,\qq}\,.
\end{equation*}
Thus both matrix inequalities can be normalized by~$\bfAhom_{n+2L,\qq}$. Apply~\eqref{e.bridge.boundary.sum} with the drift bounds above, and use~$(1+\delta_0)^d-1\leq Cc_0\sigma$. The eccentricity bound~\eqref{e.bridge.eccentricity.input} gives~$k-j_*-\lceil B_0\log_3(2+\Pi)\rceil\geq0$, so the exponential factor in~\eqref{e.bridge.source.bounds} is at most one provided that~$\nf{2c_0}{L_0}\leq\frac1{16}(1-\gamma)\log3$. Inserting these bounds into~\eqref{e.bridge.upper.comparison} and~\eqref{e.bridge.lower.comparison} proves~\eqref{e.successful.short.preliminary.comparison}.

\smallskip

We now choose the parameters in~\eqref{e.successful.short.preliminary.comparison}. First fix~$c_0(d,\gamma)\in(0,1)$ so that~$Cc_0\leq\frac13$. Next choose~$L_0(d,\gamma)$ so that~$\nf{2c_0}{L_0}\leq\frac1{16}(1-\gamma)\log3$ and~$CK_0 3^{-L_0}\leq\frac13$. The lower bound~\eqref{e.bridge.length.input} on~$L$ gives
\begin{equation*}
	3^{-L}\leq3^{-L_0}\sigma\,.
\end{equation*}
Thus the first two terms in~\eqref{e.successful.short.preliminary.comparison} are each at most~$\frac13\sigma$. Since~$2+\Pi\geq3$, choose~$B_0(\sigma,L,d,\gamma)\geq1$ so that
\begin{equation*}
	C(1+L)(2+\Pi)^{1-\frac1{16}(1-\gamma)B_0}\leq\frac13\sigma\,.
\end{equation*}
This choice is independent of~$\Pi$ and of the coefficient law. The three terms in~\eqref{e.successful.short.preliminary.comparison} now sum to at most~$\sigma$, proving the proposition.

\end{proof}

\subsection{Transport of the error and drift}

Closeness of the final annealed matrices controls the change of normalization, but does not by itself control the random coarse-grained matrices on the new grid. We must estimate the fluctuations and mean errors in every new-grid subcube, with the weights defining the histories. The next proposition performs this transfer and bounds the drift as well. Its input is measured on the old grid at~$n+2L$, while its output is a fresh error on the new grid at~$n+L$. Freshness does not discard earlier errors:~$\mathcal P_{\qq_+}(n+L;n+L)=\history_{\qq_+}(n+L)$ retains the full new-grid history.

\begin{proposition}[Transport of the error]
\label{p.two.grid.transport}
There exists~$C(d,\gamma)\geq1$ with the following property. Let~$\rho\in(0,1]$ and~$\delta\in[0,\nf14]$. Let~$\m,\m_+\in\Rpos$, set~$\qq=\mathcal Q(\m)$ and~$\qq_+=\mathcal Q(\m_+)$, and let~$k,n\in\Z$ satisfy~$j_*\leq k\leq n$. Let~$L\in\N$, and assume that
\begin{equation}
	\cus_{n+2L}^{\qq}\cup\cus_{n+L}^{\qq_+}\subseteq\cu_{2j_*}\,,
	\label{e.two.grid.cubes.input}
\end{equation}
\begin{equation}
	n-j_*\geq C\bigl( L+\log_3\bigl((2+\Pi)|\m|\,|\m^{-1}|\bigr)\bigr)\,.
	\label{e.two.grid.separation.input}
\end{equation}
\begin{equation}
	d_{\rm pr}([\m],[\m_+])\leq1\,,\qquad L\geq C+\log_3(\rho^{-1})\,,
	\label{e.two.grid.geometry.length.input}
\end{equation}
\begin{equation}
	\mathcal P_{\qq}(n+2L;k)+D_{\qq,j_*}(n+2L)\leq3^{-\frac12(1-\gamma)L}\rho\,,
	\label{e.two.grid.smallness.input}
\end{equation}
\begin{equation}
	(1-\delta)\bfAhom_{n+2L,\qq}\leq\bfAhom_{n+L,\qq_+}\leq(1+\delta)\bfAhom_{n+2L,\qq}\,.
	\label{e.two.grid.near}
\end{equation}
Then
\begin{equation}
	\mathcal P_{\qq_+}(n+L;n+L)+D_{\qq_+,j_*}(n+L)\leq C(\delta+\rho)\,.
	\label{e.successful.short.bridge}
\end{equation}
\end{proposition}

The old-grid error is measured at~$n+2L$, but the new history ends at the earlier scale~$n+L$. Earlier errors therefore receive larger weights in the new history. More precisely, a new-grid cube at scale~$j$ is filled in its bulk by old-grid cubes at scale~$j-\ell$, with~$\ell\in\{1,L\}$. The corresponding mean-history weight increases by~$3^{\frac14(1-\gamma)(L+\ell)}\leq3^{\frac12(1-\gamma)L}$; see~\eqref{e.two.grid.mean.weight.shift}. The factor~$3^{-\frac12(1-\gamma)L}$ in~\eqref{e.two.grid.smallness.input} compensates for this increase, leaving an error of order~$\rho$ after transfer. The parameter~$\delta$ in~\eqref{e.two.grid.near} controls the additional error from changing the normalizing matrix. The lower bounds in~\eqref{e.two.grid.separation.input} and~\eqref{e.two.grid.geometry.length.input} make the boundary terms and the contributions below~$j_*$ small enough to be included in~$C\rho$.

\begin{proof}
We first estimate the error and the drift separately, then use~\eqref{e.two.grid.smallness.input} to control the old-grid contributions and~\eqref{e.two.grid.separation.input}--\eqref{e.two.grid.geometry.length.input} to control the boundary and fine-scale contributions. Constants denoted by~$C$ depend only on~$d$ and~$\gamma$.

\smallskip

The bound~$d_{\rm pr}([\m],[\m_+])\leq1$ in~\eqref{e.two.grid.geometry.length.input} gives
\begin{equation*}
	K(\qq,\qq_+)\leq C(d)\qand\mathfrak e(\m_+)\leq e\,\mathfrak e(\m)\,.
\end{equation*}
For the first bound, rescale~$\m$ and~$\m_+$ to have least eigenvalue one; this leaves the rounded geometries unchanged. The first bound in~\eqref{e.two.grid.geometry.length.input} then implies~$e^{-2}\m\leq\m_+\leq e^2\m$. Hence~\mbox{$|\m^{-\nf12}\m_+^{\nf12}|\leq e$} and~\mbox{$|\m_+^{-\nf12}\m^{\nf12}|\leq e$}. Rounding perturbs each square root by a matrix of norm at most~$d3^{-j_*}$, while their least singular values are one. The rounding error bound~$d3^{-j_*}\leq\nf12$ used in~\eqref{e.rounded.grid.bounds} therefore gives the first estimate. The triangle inequality for~$d_{\rm pr}$ gives the second. Fix~$K_0(d)\geq K(\qq,\qq_+)$; the second bound in~\eqref{e.two.grid.geometry.length.input} ensures~$CK_0 3^{-L}\leq\frac12$.

\smallskip

Using the comparison of annealed matrices in~\eqref{e.two.grid.near}, we will prove that
\begin{align}
	\mathcal P_{\qq_+}(n+L;n+L)
	&\leq C3^{\frac12(1-\gamma)L}
	\mathcal P_{\qq}(n+2L;k)+C\delta
	\notag \\ & \qquad
	+C
	\left[1+\Pi\bigl(K_0\mathfrak e(\m)\mathfrak e(\m_+)
	+\mathfrak e(\m_+)^2\bigr)\right]^Q
	3^{-\frac14(1-\gamma)(n-j_*)}\,,
	\label{e.two.grid.profile}
\end{align}
and
\begin{align}
	D_{\qq_+,j_*}(n+L)
	&\leq C
	\bigl(\delta+K_0 3^{-L}
	+3^{\frac14(1-\gamma)L}D_{\qq,j_*}(n+2L)\bigr)
	\notag \\ & \qquad
	+C
	\left[1+\Pi\bigl(\mathfrak e(\m)+\mathfrak e(\m_+)\bigr)^2\right]
	(1+n+L-j_*)3^{-\frac18(1-\gamma)(n-j_*)}\,.
	\label{e.two.grid.drift}
\end{align}

\smallskip

The Whitney partition separates the bulk cubes, the smaller boundary cubes, and the cubes below~$j_*$. After estimating the contribution of the scales below~$j_*$ and reducing the error to these terms in Steps~1--2, we use averaging above~$k$ and the retained history at and below~$k$ in Steps~3--4. The decreasing boundary volumes make the additional scale sums converge. Step~5 combines these estimates; Step~6 treats the drift, and Step~7 uses~\eqref{e.two.grid.smallness.input} and~\eqref{e.two.grid.separation.input}--\eqref{e.two.grid.geometry.length.input} to bound the old-grid, boundary, and fine-scale terms.

\smallskip

\emph{Step 1: We prove~\eqref{e.two.grid.whitney.fine.bound} for~$j\geq j_*+L$, for use in Step~2.} Apply Lemma~\ref{l.two.grid.whitney}(i) to~$y+\cus_j^{\qq_+}$, where~$j_*\leq j\leq n+L$ and~$y\in3^j\qq_+\Zd\cap\cus_{n+L}^{\qq_+}$. Take~$\ell=1$ if~$j\leq k+L$ and~$\ell=L$ otherwise. We call the selected cubes of scale~$j-\ell$ the bulk cubes and the smaller ones the boundary cubes. Thus we use the length-$L$ separation only when the bulk scale~$j-L$ lies above~$k$.

\smallskip

Normalize the matrices by~$\bfAhom_{n+L,\qq_+}$ and set
\begin{equation}
	G_j(y)\coloneqq\bfAhom_{n+L,\qq_+}^{-\nf12}\biggl(\sum_{r\leq j-\ell}\sum_{z\in\mathcal Z_r(y+\cus_j^{\qq_+})}\frac{|\cus_r^{\qq}|}{|\cus_j^{\qq_+}|}\bfA(z+\cus_r^{\qq})\biggr)\bfAhom_{n+L,\qq_+}^{-\nf12}\,.
	\label{e.two.grid.whitney.average}
\end{equation}
All these adapted cubes lie in~$\cu_{2j_*}$ by~\eqref{e.two.grid.cubes.input}. For~$j\geq j_*+L$, we have~$j-\ell\geq j_*$, so we may apply~\eqref{e.two.grid.whitney.source} and normalize by~\eqref{e.two.grid.source.normalization} on the new grid. Weakening the decay gives
\begin{multline}
	\bfAhom_{n+L,\qq_+}^{-\nf12}\biggl(\sum_{r<j_*}\sum_{z\in\mathcal Z_r(y+\cus_j^{\qq_+})}\frac{|\cus_r^{\qq}|}{|\cus_j^{\qq_+}|}\bfA(z+\cus_r^{\qq})\biggr)\bfAhom_{n+L,\qq_+}^{-\nf12}
\\
	\leq C\Pi K_0\mathfrak e(\m)\mathfrak e(\m_+)\RSZ3^{-(1-\gamma)(j-j_*)}\Itwod\,.
	\label{e.two.grid.whitney.fine.bound}
\end{multline}

\smallskip

\emph{Step 2: Reduction of the error estimate.} In this step, we derive~\eqref{e.two.grid.whitney.maximal.gap},~\eqref{e.two.grid.whitney.centered.split}, and~\eqref{e.two.grid.whitney.mean.bound}. These inequalities identify the bulk, boundary, and below-$j_*$ terms that remain to be estimated. The quantity to bound is
\begin{equation}
	\mathcal P_{\qq_+}(n+L;n+L)=\history_{\qq_+}^{\rm fluc}(n+L)+\history_{\qq_+}^{\rm mean}(n+L)\,.
	\label{e.two.grid.profile.decomposition}
\end{equation}

\smallskip

We first consider~$j\geq j_*+L$; the target scales below~$j_*+L$ are treated in Step~5. For convenience, write
\begin{equation*}
	F_j(y)\coloneqq\bfAhom_{n+L,\qq_+}^{-\nf12}\bfA(y+\cus_j^{\qq_+})\bfAhom_{n+L,\qq_+}^{-\nf12}\,.
\end{equation*}
Subadditivity and stationarity give
\begin{equation*}
	0\leq F_j(y)\leq G_j(y)\qand \E[F_j(y)]=P_{j,n+L}^{\qq_+}\geq\Itwod\,.
\end{equation*}
As in the parent--child recurrence, this order does not identify the coarse-grained matrix with its partition average. Lemma~\ref{l.fixed.geometry.positive.gap} controls the discrepancy using the fluctuations of that average and its mean excess. Here~$\E[F_j(y)]\geq\Itwod$, so the positive mean gap is bounded by~$\E[G_j(y)]-\Itwod$. The positive-gap estimate~\eqref{e.fixed.geometry.positive.gap}, raised to the~$Q$-th power, gives
\begin{align}
	\bigl\|F_j(y)-P_{j,n+L}^{\qq_+}\bigr\|_{L^Q(S_Q)}^Q
	&\leq2^{Q-1}\bigl(1+(2d)^{\nf1Q}\bigr)^Q\bigl\|G_j(y)-\E[G_j(y)]\bigr\|_{L^Q(S_Q)}^Q
	\notag \\ & \qquad
	+2^{2Q-1}\bigl(1+d^{1-\nf1Q}\bigr)^Q|\E[G_j(y)]|^{Q-1}\tr\bigl(\E[G_j(y)]-\Itwod\bigr)\,.
	\label{e.two.grid.local.mean.decomposition}
\end{align}
Since~$\Itwod\leq P_{j,n+L}^{\qq_+}\leq\E[G_j(y)]$, the definition of~$\meanpenalty_Q$ gives
\begin{equation*}
	\meanpenalty_Q(P_{j,n+L}^{\qq_+})\leq\meanpenalty_Q(\E[G_j(y)])\qqand |\E[G_j(y)]|^{Q-1}\tr(\E[G_j(y)]-\Itwod)\leq\meanpenalty_Q(\E[G_j(y)])\,.
\end{equation*}

For the fluctuation history we retain the joint maximum over the scales and target cubes in the centered part. The proof of Lemma~\ref{l.fixed.geometry.positive.gap} also gives this form of the estimate: take the weighted maximum in its pointwise gap bound, use the sum only for the terms containing the deterministic means, and apply the same Hölder and Young inequalities. Writing~$\max_y$ for the maximum over~$y\in3^j\qq_+\Zd\cap\cus_{n+L}^{\qq_+}$ throughout this proof, we obtain
\begin{align}
	\lefteqn{
	\E\biggl[\max_{j_*+L\leq j\leq n+L}3^{-Q\rho_{\max}(n+L-j)}\max_y|F_j(y)-\E[F_j(y)]|^Q\biggr]
	} \qquad &
	\notag \\ &
	\leq C\E\biggl[\max_{j_*+L\leq j\leq n+L}3^{-Q\rho_{\max}(n+L-j)}\max_y|G_j(y)-\E[G_j(y)]|_{S_Q}^Q\biggr]
	\notag \\ & \qquad
	+C\sum_{j=j_*+L}^{n+L}3^{-(Q\rho_{\max}-d)(n+L-j)}\max_y\meanpenalty_Q(\E[G_j(y)])\,.
	\label{e.two.grid.whitney.maximal.gap}
\end{align}
In Steps~3 and~4, we keep the centered terms with~$r\leq k$ under this joint maximum and bound it by a sum only for~$r>k$. This avoids a factor for the number of scales in the initial-history term.

\smallskip

The coefficient of the mean penalty in~\eqref{e.two.grid.whitney.maximal.gap} satisfies
\begin{equation*}
	3^{-(Q\rho_{\max}-d)(n+L-j)}=3^{-\left(Q\gamma+\frac14(1-\gamma)\right)(n+L-j)}\leq3^{-\frac14(1-\gamma)(n+L-j)}\,.
\end{equation*}
Thus the same weighted mean bound controls both the mean history and the last term in~\eqref{e.two.grid.whitney.maximal.gap}, including the endpoint~$j=n+L$.

\smallskip

The change of normalization is controlled by~\eqref{e.two.grid.near}:
\begin{equation}
	\frac1{1+\delta}\Itwod\leq\bfAhom_{n+2L,\qq}^{\nf12}\bfAhom_{n+L,\qq_+}^{-1}\bfAhom_{n+2L,\qq}^{\nf12}\leq\frac1{1-\delta}\Itwod\,.
	\label{e.two.grid.inverse.comparison}
\end{equation}
Taking determinants also gives
\begin{equation}
	(1-\delta)^2\det(\bfAhom_{n+2L,\qq})^{\nf1d}\leq\det(\bfAhom_{n+L,\qq_+})^{\nf1d}\leq(1+\delta)^2\det(\bfAhom_{n+2L,\qq})^{\nf1d}\,.
	\label{e.two.grid.determinant}
\end{equation}

\smallskip

Split the centered Whitney sum into the scales~$r=j-\ell$,~$j_*\leq r<j-\ell$, and~$r<j_*$. In the first two parts, stationarity identifies the centered blocks with~$V_{r,n+2L}^{\qq}(z)$ after changing the normalization by~\eqref{e.two.grid.inverse.comparison}. For the last part, apply~\eqref{e.two.grid.whitney.fine.bound} to the sum and take expectation to bound its mean. The triangle inequality gives
\begin{align}
	|G_j(y)-\E[G_j(y)]|_{S_Q}
	&\leq C\biggl|\sum_{z\in\mathcal Z_{j-\ell}(y+\cus_j^{\qq_+})}\frac{|\cus_{j-\ell}^{\qq}|}{|\cus_j^{\qq_+}|}V_{j-\ell,n+2L}^{\qq}(z)\biggr|_{S_Q}
	\notag \\ & \qquad
	+C\sum_{r=j_*}^{j-\ell-1}\biggl|\sum_{z\in\mathcal Z_r(y+\cus_j^{\qq_+})}\frac{|\cus_r^{\qq}|}{|\cus_j^{\qq_+}|}V_{r,n+2L}^{\qq}(z)\biggr|_{S_Q}
	\notag \\ & \qquad
	+C\Pi K_0\mathfrak e(\m)\mathfrak e(\m_+)\bigl(\RSZ+\E[\RSZ]\bigr)3^{-(1-\gamma)(j-j_*)}\,.
	\label{e.two.grid.whitney.centered.split}
\end{align}
The same~$\RSZ$ bounds every target cube. Taking the~$L^Q$ norm of~$\RSZ+\E[\RSZ]$ costs at most~$2\|\RSZ\|_{L^Q}\leq4$. This bound remains valid under the weighted joint maximum over targets and scales, without a factor for their number.

\smallskip

For the mean penalty, take expectation in~\eqref{e.two.grid.whitney.fine.bound} and use~$\E[\RSZ]\leq2$. Replace the fine-scale part of~$\E[G_j(y)]$ by this deterministic upper bound. Insert~$\Itwod$ with the missing volume weight into the mean of the remaining sum. Its old-grid normalization is a convex combination of the matrices~$P_{r,n+2L}^{\qq}\geq\Itwod$ and~$\Itwod$. The comparison~\eqref{e.two.grid.inverse.comparison} changes its trace excess by at most a factor~$(1-\delta)^{-1}$ and an additive~$C\delta$. Applying convexity of~$x\mapsto(1+x)^Q-1$ on~$[0,\infty)$ and using~\eqref{e.two.grid.whitney.volumes}, with~$3^{r-j}\leq3^{-(1-\gamma)(j-r)}$ for the boundary cubes, gives
\begin{align}
	\meanpenalty_Q(\E[G_j(y)])
	&\leq C\meanpenalty_Q(P_{j-\ell,n+2L}^{\qq})+C\sum_{r=j_*}^{j-\ell-1}3^{-(1-\gamma)(j-r)}\meanpenalty_Q(P_{r,n+2L}^{\qq})+C\delta
	\notag \\ & \qquad
	+C\bigl[\Pi K_0\mathfrak e(\m)\mathfrak e(\m_+)\bigr]^Q\bigl(3^{-(1-\gamma)(j-j_*)}+3^{-Q(1-\gamma)(j-j_*)}\bigr)\,.
	\label{e.two.grid.whitney.mean.bound}
\end{align}
The last term retains both the first and the~$Q$-th power of the fine-scale bound. Together with~\eqref{e.two.grid.whitney.maximal.gap}, the inequalities~\eqref{e.two.grid.whitney.centered.split} and~\eqref{e.two.grid.whitney.mean.bound} give the required reduction. Step~3 estimates the first term on each right side, and Step~4 estimates the sums over~$j_*\leq r<j-\ell$. Their contributions to the error, together with the below-$j_*$ and comparison terms already bounded here, are combined in Step~5.

\smallskip

\begin{samepage}
\emph{Step 3: Bulk cubes.} In this step, we prove the following bounds for the first terms on the right sides of~\eqref{e.two.grid.whitney.centered.split} and~\eqref{e.two.grid.whitney.mean.bound}:
\begin{multline}
	\E\Biggl[\max_{j_*+L\leq j\leq n+L}3^{-Q\rho_{\max}(n+L-j)}\max_y\biggl|\sum_{z\in\mathcal Z_{j-\ell}(y+\cus_j^{\qq_+})}\frac{|\cus_{j-\ell}^{\qq}|}{|\cus_j^{\qq_+}|}V_{j-\ell,n+2L}^{\qq}(z)\biggr|_{S_Q}^Q\Biggr]\\
	\leq C3^{\frac14(1-\gamma)L}\mathcal P_{\qq}(n+2L;k)\,,
	\label{e.two.grid.bulk.fluctuations}
\end{multline}
and
\begin{equation}
	\sum_{j=j_*+L}^{n+L}3^{-\frac14(1-\gamma)(n+L-1-j)}\meanpenalty_Q(P_{j-\ell,n+2L}^{\qq})\leq C3^{\frac12(1-\gamma)L}\mathcal P_{\qq}(n+2L;k)\,.
	\label{e.two.grid.bulk.mean}
\end{equation}
\end{samepage}

\smallskip

For the fluctuations, the split at the initialization scale~$k$ matches the two kinds of information retained in the error. Above~$k$, the error records individual fluctuation moments, to which we apply finite-range averaging. At and below~$k$, it retains a weighted maximum over scales and cells; we use that maximum directly. By~\eqref{e.two.grid.whitney.counts},
\begin{equation*}
	\biggl(\sum_{z\in\mathcal Z_{j-\ell}(y+\cus_j^{\qq_+})}\frac{|\cus_{j-\ell}^{\qq}|^2}{|\cus_j^{\qq_+}|^2}\biggr)^{\nf12}\leq C3^{-\frac d2\ell}\,.
\end{equation*}
Apply Lemma~\ref{l.fixed.geometry.matrix.averaging} to these equal-volume adapted cubes, which are aligned integer translates. There are~$3^{d(n+L-j)}$ possible target cubes~$y+\cus_j^{\qq_+}$ at scale~$j$. Bounding their maximum by their sum and taking the~$Q$-th root changes the fluctuation-history weight to
\begin{equation*}
	3^{-\rho_{\max}(n+L-j)}3^{\frac dQ(n+L-j)}=3^{-\left(\gamma+\frac{1-\gamma}{4Q}\right)(n+L-j)}\,.
\end{equation*}
For~$r=j-\ell>k$, averaging and~\eqref{e.two.grid.inverse.comparison} bound the~$Q$-th root of the scale-$j$ contribution in~\eqref{e.two.grid.bulk.fluctuations} by
\begin{equation*}
	C3^{\frac{1-\gamma}{4Q}L}3^{-\gamma(n+L-j)}3^{-\left(\frac d2-\frac{1-\gamma}{4Q}\right)\ell}3^{-\frac{1-\gamma}{4Q}(n+2L-r)}e^{\Delta_{r,n+2L}^{\qq}}\E\bigl[|V_r^{\qq}|_{S_Q}^Q\bigr]^{\nf1Q}\,.
\end{equation*}
Here~$\frac d2-\frac{1-\gamma}{4Q}>0$, and each~$r$ occurs for at most two values of~$j$, namely~$r+1$ and~$r+L$. After taking the~$Q$-th power and summing over~$j$, the remaining old-grid sum satisfies
\begin{equation*}
	\sum_{r=k+1}^{n+2L}3^{-\frac14(1-\gamma)(n+2L-r)}e^{Q\Delta_{r,n+2L}^{\qq}}\E\bigl[|V_r^{\qq}|_{S_Q}^Q\bigr]\leq\mathcal P_{\qq}(n+2L;k)\,.
\end{equation*}
The contribution from these bulk cubes is therefore at most~$C3^{\frac14(1-\gamma)L}\mathcal P_{\qq}(n+2L;k)$.

\smallskip

For bulk scales~$r\leq k$, necessarily~$\ell=1$ and~$j\leq k+1$. Group the adapted cubes by their scale-$k$ parents. Within each such parent,~\eqref{e.scale.selection.fluctuation.history} controls the maximum at scale~$r$ with factor~$3^{\rho_{\max}(k-r)}$. The total relative volume is at most one, so the bulk contribution has factor at most~$C3^{\rho_{\max}(k-j)}$. At most~$C3^{d(n+L-k)}$ scale-$k$ parents meet~$\cus_{n+L}^{\qq_+}$. Taking their maximum, using stationarity, and changing the normalization at scale~$k$ gives the factor
\begin{equation}
	3^{-\left(\rho_{\max}-\frac dQ\right)(n+L-k)}|P_{k,n+2L}^{\qq}|\leq3^{\frac{1-\gamma}{4Q}L}3^{-\frac{1-\gamma}{4Q}(n+2L-k)}\bigl(1+\meanpenalty_Q(P_{k,n+2L}^{\qq})\bigr)^{\nf1Q}\,.
	\label{e.two.grid.old.history.factor}
\end{equation}
We used~$|P|\leq1+\tr(P-\Itwod)$ and~$\rho_{\max}-\nf{d}{Q}=\gamma+\nf{1-\gamma}{4Q}$. Taking the~$Q$-th power and using
\begin{equation*}
	3^{-\frac14(1-\gamma)(n+2L-k)}\bigl(1+\meanpenalty_Q(P_{k,n+2L}^{\qq})\bigr)\history_{\qq}^{\rm fluc}(k)\leq\mathcal P_{\qq}(n+2L;k)
\end{equation*}
proves the same bound for the remaining bulk fluctuations, and hence~\eqref{e.two.grid.bulk.fluctuations}.

\smallskip

To prove~\eqref{e.two.grid.bulk.mean}, put~$r=j-\ell$. The ratio of the new and old mean-history weights satisfies
\begin{equation}
	3^{-\frac14(1-\gamma)(n+L-1-j)}=3^{\frac14(1-\gamma)(L+\ell)}3^{-\frac14(1-\gamma)(n+2L-1-r)}\leq3^{\frac12(1-\gamma)L}3^{-\frac14(1-\gamma)(n+2L-1-r)}\,.
	\label{e.two.grid.mean.weight.shift}
\end{equation}
For~$r\geq k$, the resulting terms belong to~$\history_{\qq}^{\rm mean}(n+2L;k)$. For~$r<k$, the splitting of the mean penalty in Step~1 of the proof of Proposition~\ref{p.fixed.geometry.one.grid.propagation} gives
\begin{equation}
	\meanpenalty_Q(P_{r,n+2L}^{\qq})\leq\bigl(1+\meanpenalty_Q(P_{k,n+2L}^{\qq})\bigr)\meanpenalty_Q(P_{r,k}^{\qq})+\meanpenalty_Q(P_{k,n+2L}^{\qq})\,.
	\label{e.two.grid.mean.split}
\end{equation}
The first term gives the mean part of the initial-history term in~$\mathcal P_{\qq}(n+2L;k)$. Summing the geometric weights of the second gives a constant times the~$r=k$ term of~$\history_{\qq}^{\rm mean}(n+2L;k)$. This proves~\eqref{e.two.grid.bulk.mean}.

\smallskip

\emph{Step 4: Boundary cubes.} In this step, we prove the corresponding bounds for the second terms on the right sides of~\eqref{e.two.grid.whitney.centered.split} and~\eqref{e.two.grid.whitney.mean.bound}:
\begin{multline}
	\E\Biggl[\max_{j_*+L\leq j\leq n+L}3^{-Q\rho_{\max}(n+L-j)}\max_y\Biggl(\sum_{r=j_*}^{j-\ell-1}\biggl|\sum_{z\in\mathcal Z_r(y+\cus_j^{\qq_+})}\frac{|\cus_r^{\qq}|}{|\cus_j^{\qq_+}|}V_{r,n+2L}^{\qq}(z)\biggr|_{S_Q}\Biggr)^Q\Biggr]\\
	\leq C3^{\frac14(1-\gamma)L}\mathcal P_{\qq}(n+2L;k)\,,
	\label{e.two.grid.boundary.fluctuations}
\end{multline}
and
\begin{equation}
	\sum_{j=j_*+L}^{n+L}3^{-\frac14(1-\gamma)(n+L-1-j)}\sum_{r=j_*}^{j-\ell-1}3^{-(1-\gamma)(j-r)}\meanpenalty_Q(P_{r,n+2L}^{\qq})\leq C3^{\frac14(1-\gamma)L}\mathcal P_{\qq}(n+2L;k)\,.
	\label{e.two.grid.boundary.mean}
\end{equation}

\smallskip

As in Step~3, we split the fluctuations at scale~$k$. The additional sum over the smaller scales is controlled by their decreasing relative volumes. At a scale~$j_*\leq r<j-\ell$, the counting estimate gives
\begin{equation*}
	\biggl(\sum_{z\in\mathcal Z_r(y+\cus_j^{\qq_+})}\frac{|\cus_r^{\qq}|^2}{|\cus_j^{\qq_+}|^2}\biggr)^{\nf12}\leq C3^{-\frac{d+1}{2}(j-r)}\,.
\end{equation*}
For~$r>k$, apply Lemma~\ref{l.fixed.geometry.matrix.averaging} as in Step~3. For each pair~$(j,r)$ in~\eqref{e.two.grid.boundary.fluctuations}, the~$L^Q$ norm of the corresponding weighted maximum is at most
\begin{equation*}
	C3^{\frac{1-\gamma}{4Q}L}3^{-\gamma(n+L-j)}3^{-\left(\frac{d+1}{2}-\frac{1-\gamma}{4Q}\right)(j-r)}3^{-\frac{1-\gamma}{4Q}(n+2L-r)}e^{\Delta_{r,n+2L}^{\qq}}\E\bigl[|V_r^{\qq}|_{S_Q}^Q\bigr]^{\nf1Q}\,.
\end{equation*}
The coefficient in~$j-r$ is summable, since~$\frac{d+1}{2}>\frac{1-\gamma}{4Q}$. The triangle inequality over~$r$, followed by the weighted Hölder inequality and summation over~$j$, bounds these boundary fluctuations by~$C3^{\frac14(1-\gamma)L}\mathcal P_{\qq}(n+2L;k)$, using the last sum in the error.

\smallskip

For~$j_*\leq r\leq k$, group the boundary cubes by their scale-$k$ parents and use the maximum in~$\history_{\qq}^{\rm fluc}(k)$. The total relative volume at scale~$r$ is at most~$C3^{r-j}$. If~$j\geq k$, the resulting coefficient is bounded by
\begin{equation*}
	\sum_{r=j_*}^{k}3^{r-j}3^{\rho_{\max}(k-r)}=3^{k-j}\sum_{s=0}^{k-j_*}3^{-(1-\rho_{\max})s}\leq C3^{k-j}\,.
\end{equation*}
If~$j<k$, then~$\ell=1$ and~$r<j-1$, so the same sum, restricted to these scales, is at most~$C3^{\rho_{\max}(k-j)}$. Thus in both cases it is at most~$C3^{\rho_{\max}(k-j)}$, because~$\rho_{\max}<1$. Taking the maximum over the scale-$k$ parents and applying~\eqref{e.two.grid.old.history.factor} gives the same bound as for the bulk fluctuations. Combining the two ranges of~$r$ proves~\eqref{e.two.grid.boundary.fluctuations}.

\smallskip

To prove~\eqref{e.two.grid.boundary.mean}, interchange the sums over~$r$ and~$j$ and use
\begin{equation*}
	\sum_{j=r+1}^{n+L}3^{-\frac14(1-\gamma)(n+L-1-j)}3^{-(1-\gamma)(j-r)}\leq C3^{\frac14(1-\gamma)L}3^{-\frac14(1-\gamma)(n+2L-1-r)}\,.
\end{equation*}
The series is summable because~$1-\gamma>\frac14(1-\gamma)$. As in Step~3, the terms with~$r\geq k$ belong to the later mean history, and~\eqref{e.two.grid.mean.split} bounds those with~$r<k$ by the initial mean history and the~$r=k$ term. This proves~\eqref{e.two.grid.boundary.mean}.

\smallskip

\emph{Step 5: Combination of the error estimates.} In this step, we prove~\eqref{e.two.grid.profile}. Insert the bulk bounds~\eqref{e.two.grid.bulk.fluctuations}--\eqref{e.two.grid.bulk.mean} and the boundary bounds~\eqref{e.two.grid.boundary.fluctuations}--\eqref{e.two.grid.boundary.mean} into~\eqref{e.two.grid.whitney.centered.split} and~\eqref{e.two.grid.whitney.mean.bound}, and then use~\eqref{e.two.grid.whitney.maximal.gap}. The bulk and boundary terms above~$j_*$ contribute at most~$C3^{\frac12(1-\gamma)L}\mathcal P_{\qq}(n+2L;k)$. Summing the comparison error against the geometric weights contributes at most~$C\delta$. It remains to estimate the partition pieces below~$j_*$ and the target scales~$j_*\leq j<j_*+L$.

\smallskip

The fine-scale terms in~\eqref{e.two.grid.whitney.fine.bound} and~\eqref{e.two.grid.whitney.mean.bound} have weights satisfying
\begin{equation*}
	\sum_{j=j_*+L}^{n+L}3^{-\frac14(1-\gamma)(n+L-j)}\bigl(3^{-(1-\gamma)(j-j_*)}+3^{-Q(1-\gamma)(j-j_*)}\bigr)\leq C3^{-\frac14(1-\gamma)(n+L-j_*)}\,.
\end{equation*}
For the remaining target scales~$j_*\leq j<j_*+L$, apply~\eqref{e.source.adapted.bound} to~$y+\cus_j^{\qq_+}$ itself. It gives~$F_j(y)\leq C\Pi\mathfrak e(\m_+)^2\RSZ\Itwod$ simultaneously for all these adapted cubes. Taking expectation controls the mean penalty, and~$\|\RSZ\|_{L^Q}\leq2$ controls the centered contribution under the joint maximum. Their total weight satisfies
\begin{equation*}
	\sum_{j=j_*}^{j_*+L-1}3^{-\frac14(1-\gamma)(n+L-j)}\leq C3^{-\frac14(1-\gamma)(n-j_*)}\,.
\end{equation*}
Since the bracket below is at least one, both remaining contributions are bounded by
\begin{equation*}
	C\left[1+\Pi\bigl(K_0\mathfrak e(\m)\mathfrak e(\m_+)+\mathfrak e(\m_+)^2\bigr)\right]^Q3^{-\frac14(1-\gamma)(n-j_*)}\,.
\end{equation*}
This proves~\eqref{e.two.grid.profile}.

\smallskip

\emph{Step 6: Determinant drift.} The Whitney comparison controls coarse blocks, whereas the drift is defined through consecutive increments. To prove~\eqref{e.two.grid.drift}, we first use Abel summation to express the drift in terms of block excesses:
\begin{align}
	D_{\qq_+,j_*}(n+L) &
	=3^{-\frac18(1-\gamma)(n+L-j_*-1)}\tr\bigl(P_{j_*,n+L}^{\qq_+}-\Itwod\bigr)
	\notag \\ & \qquad
	+\bigl(1-3^{-\frac18(1-\gamma)}\bigr)\sum_{j=j_*+1}^{n+L-1}3^{-\frac18(1-\gamma)(n+L-j-1)}\tr\bigl(P_{j,n+L}^{\qq_+}-\Itwod\bigr)\,.
	\label{e.two.grid.drift.abel}
\end{align}
All these trace excesses are nonnegative. Replacing the new normalizing matrix by the old one using~\eqref{e.two.grid.inverse.comparison} multiplies the nonnegative terms by at most~$(1-\delta)^{-1}$. Subtracting the new annealed block contributes at most~$\nf{2d\delta}{1-\delta}$ in each trace. The Abel coefficients sum to one, so their total comparison error is at most~$C\delta$.

\smallskip

We estimate the traces in~\eqref{e.two.grid.drift.abel} by applying Lemma~\ref{l.two.grid.whitney}(i) with~$\ell=L$ for every~$j\geq j_*+L$. The expectation of the bulk sum is bounded by~$\bfAhom_{j-L,\qq}$, since its total relative volume is at most one. Taking expectation, using~\eqref{e.two.grid.whitney.volumes} and~\eqref{e.two.grid.whitney.source}, and normalizing by~\eqref{e.two.grid.source.normalization}, gives
\begin{equation}
	\bfAhom_{j,\qq_+}\leq\bfAhom_{j-L,\qq}+CK_0\sum_{r=j_*}^{j-L-1}3^{r-j}\bfAhom_{r,\qq}+C\left[1+\Pi\bigl(\mathfrak e(\m)+\mathfrak e(\m_+)\bigr)^2\right]3^{-(1-\gamma)(j-j_*)}\bfAhom_{n+2L,\qq}\,.
	\label{e.two.grid.whitney.drift.comparison}
\end{equation}
Here the last term uses~$\bfAhom_{n+2L,\qq}$ as the normalizing matrix. For~$j_*\leq j<j_*+L$, the adapted cube bound instead bounds the whole matrix~$\bfAhom_{j,\qq_+}$ by~$C\left[1+\Pi\bigl(\mathfrak e(\m)+\mathfrak e(\m_+)\bigr)^2\right]\bfAhom_{n+2L,\qq}$.

\smallskip

For the bulk term in~\eqref{e.two.grid.whitney.drift.comparison}, expand
\begin{equation}
	P_{j-L,n+2L}^{\qq}-\Itwod=\sum_{r=j-L+1}^{n+2L}\bigl(P_{r-1,n+2L}^{\qq}-P_{r,n+2L}^{\qq}\bigr)\geq0\,.
	\label{e.two.grid.drift.increment.expansion}
\end{equation}
A fixed increment with index~$r$ occurs only for~$j\leq r+L-1$. Interchanging the finite sums and summing the geometric weights gives
\begin{align}
	\lefteqn{
	\sum_{j=j_*+L}^{n+L-1}3^{-\frac18(1-\gamma)(n+L-j-1)}\tr\bigl(P_{j-L,n+2L}^{\qq}-\Itwod\bigr)
	} \qquad &
	\notag \\ &
	\leq C3^{\frac14(1-\gamma)L}\sum_{r=j_*+1}^{n+2L}3^{-\frac18(1-\gamma)(n+2L-r)}\tr\bigl(P_{r-1,n+2L}^{\qq}-P_{r,n+2L}^{\qq}\bigr)
	\notag \\ &
	=C3^{\frac14(1-\gamma)L}D_{\qq,j_*}(n+2L)\,.
	\label{e.two.grid.drift.old.grid}
\end{align}

\smallskip

For the boundary term, first write~$P_{r,n+2L}^{\qq}=\Itwod+(P_{r,n+2L}^{\qq}-\Itwod)$. The identity part contributes at most~$CK_0 3^{-L}$, by~$\sum_{r<j-L}3^{r-j}\leq C3^{-L}$. Expand each remaining difference into the positive increments as in~\eqref{e.two.grid.drift.increment.expansion}. For an increment with index~$a$, sum~$3^{r-j}$ over~$r<a$ and then sum the Abel weight over~$j$. Since~$1>\frac18(1-\gamma)$, both geometric sums converge, and we obtain
\begin{equation}
	K_0\sum_{j=j_*+L}^{n+L-1}3^{-\frac18(1-\gamma)(n+L-j-1)}\sum_{r=j_*}^{j-L-1}3^{r-j}\tr\bigl(P_{r,n+2L}^{\qq}-\Itwod\bigr)\leq C3^{\frac14(1-\gamma)L}D_{\qq,j_*}(n+2L)\,.
	\label{e.two.grid.whitney.drift.boundary}
\end{equation}
Thus the smaller boundary cubes contribute a part controlled by the old drift as well as the term~$CK_0 3^{-L}$.

\smallskip

For the lower endpoint in~\eqref{e.two.grid.drift.abel}, the early scales~$j_*,\ldots,j_*+L-1$, and the contribution of the scales below~$j_*$ in~\eqref{e.two.grid.whitney.drift.comparison}, use
\begin{equation}
	3^{-\frac18(1-\gamma)(n-j_*)}+\sum_{j=j_*+L}^{n+L-1}3^{-\frac18(1-\gamma)(n+L-j-1)}3^{-(1-\gamma)(j-j_*)}\leq C(1+n+L-j_*)3^{-\frac18(1-\gamma)(n-j_*)}\,.
	\label{e.two.grid.drift.source.convolution}
\end{equation}
The first term bounds the sum of the early geometric weights; the convolution is summable because~$1-\gamma>\frac18(1-\gamma)$. Multiply by~$C\left[1+\Pi\bigl(\mathfrak e(\m)+\mathfrak e(\m_+)\bigr)^2\right]$. Together with~\eqref{e.two.grid.drift.old.grid},~\eqref{e.two.grid.whitney.drift.boundary}, and the errors~$C\delta+CK_0 3^{-L}$, this proves~\eqref{e.two.grid.drift}.

\smallskip

\emph{Step 7: We deduce~\eqref{e.successful.short.bridge}.} The smallness bound~\eqref{e.two.grid.smallness.input} gives
\begin{equation*}
	3^{\frac12(1-\gamma)L}\mathcal P_{\qq}(n+2L;k)+3^{\frac14(1-\gamma)L}D_{\qq,j_*}(n+2L)\leq\rho\,.
\end{equation*}
Thus the old-grid terms in~\eqref{e.two.grid.profile} and~\eqref{e.two.grid.drift} contribute at most~$C\rho$.

\smallskip

Since~$\mathfrak e(\m_+)\leq e\,\mathfrak e(\m)$ and~$K_0$ depends only on~$d$, we can bound the last terms in~\eqref{e.two.grid.profile} and~\eqref{e.two.grid.drift}, which include both the partition pieces below~$j_*$ and the target scales~$j_*\leq j<j_*+L$. Taking the constant in~\eqref{e.two.grid.separation.input} sufficiently large gives
\begin{equation}
	\bigl((2+\Pi)\mathfrak e(\m)^2\bigr)^Q3^{-\frac14(1-\gamma)(n-j_*)}\leq3^{-L}\,,
	\label{e.bridge.profile.source.bound}
\end{equation}
and
\begin{equation}
	(2+\Pi)\mathfrak e(\m)^2(1+n+L-j_*)3^{-\frac18(1-\gamma)(n-j_*)}\leq C3^{-L}\,.
	\label{e.bridge.source.slopes}
\end{equation}
Indeed, use part of the decay in~$n-j_*$ to absorb the power of~$(2+\Pi)\mathfrak e(\m)^2$ and, in the second estimate, the factor~$1+n+L-j_*$. The remaining decay is at most~$3^{-L}$, since~$n-j_*\geq CL$.

\smallskip

Finally,~$3^{-L}\leq\rho$ by~\eqref{e.two.grid.geometry.length.input}. These last terms and the boundary term therefore contribute at most~$C\rho$, and the comparison terms at most~$C\delta$. Adding~\eqref{e.two.grid.profile} and~\eqref{e.two.grid.drift} proves~\eqref{e.successful.short.bridge}.
\end{proof}

\subsection{Proof of the scale-selection alternatives}
\label{s.scale.selection.proof}

The fixed-grid propagation estimate gives the startup and contraction estimates. For the geometry update, we first advance the small old-grid input from~$n$ to~$n+2L$, compare the old and new means, and then transfer the errors to the new grid at~$n+L$. The transported error need not remain below the update threshold~$\eta$; it must be at most one, so that startup and subsequent contraction can begin again. We now choose the parameters to meet these requirements and check the conditions needed to repeat the step.

\begin{proof}[Proof of Proposition~\ref{p.scale.selection}]
We first treat equal error arguments, then the two tests for unequal error arguments, and finally verify that the output satisfies the input conditions for the next application. Take~$h=2Q$, which is admissible in Proposition~\ref{p.fixed.geometry.one.grid.propagation}. Take~$c_0$ and~$L_0$ from Proposition~\ref{p.successful.short.bridge}, and choose~$c\in(0,c_0]$. Increase~$L_0(d,\gamma)$ if necessary for Proposition~\ref{p.two.grid.transport}, and choose~$\varepsilon_0(d,\gamma)\in(0,(\nf{c_0}{d+1})^2]$ so that
\begin{equation}
	Qd\varepsilon_0\leq\log2
	\qquad\text{and}\qquad
	C(d,\gamma)\varepsilon_0^{\frac18(1-\gamma)}\leq1\,.
	\label{e.local.parameter.choice}
\end{equation}
Fix~$0<\sigma\leq\varepsilon\leq\varepsilon_0$ and choose
\begin{equation}
	L\coloneqq L_0+\biggl\lceil\frac32\log_3(\sigma^{-1})\biggr\rceil\,.
	\label{e.bridge.length.choice}
\end{equation}
Set~$\eta\coloneqq c\varepsilon\sigma$. Take~$B_0$ from Proposition~\ref{p.successful.short.bridge} for this length and comparison tolerance~$\varepsilon^{\nf12}\sigma$; we verify~\eqref{e.bridge.length.input} below. We may decrease~$\varepsilon_0(d,\gamma)$ and increase~$B_0(\varepsilon,\sigma,d,\gamma)$ further when checking the transport hypotheses.

\smallskip

\emph{Step 1: Case~1.} Suppose that~$k=n$. By~\eqref{e.renormalization.input},~$\mathcal P_{\qq}(n;k)+D_{\qq,j_*}(n)\leq1$. Apply~\eqref{e.fixed.geometry.fixed.span.propagation} with~$m=n$ and advance~$h=2Q$. Since~$h$ depends only on~$d$ and~$\gamma$, its factor is absorbed into~$C(d,\gamma)$, giving~\eqref{e.renormalization.startup}.

\smallskip

\emph{Step 2: The cases with~$k<n$.} We prove~\eqref{e.renormalization.service} in Case~2 and~\eqref{e.renormalization.transport.output} in Case~3; Cases~4 and~5 assert no estimate. Assume first that~$k<n$ and
\begin{equation*}
	\mathcal P_{\qq}(n;k)+D_{\qq,j_*}(n)>\eta\,.
\end{equation*}
If~$d^{-1}\widehat\Delta_h^{\qq}(n)>\sigma$, we are in Case~4. Otherwise, we are in Case~2, and
\begin{equation*}
	Q\widehat\Delta_h^{\qq}(n)\leq Qd\sigma\leq Qd\varepsilon_0\leq\log2\,.
\end{equation*}
Consequently,~\eqref{e.fixed.geometry.synchronized.propagation} and~\eqref{e.local.parameter.choice} give
\begin{align*}
	\mathcal P_{\qq}(n+h;k)+D_{\qq,j_*}(n+h)
	&\leq\frac18e^{Q\widehat\Delta_h^{\qq}(n)}
	\bigl(\mathcal P_{\qq}(n;k)+D_{\qq,j_*}(n)\bigr)
	+C(d,\gamma)
	\bigl(e^{Q\widehat\Delta_h^{\qq}(n)}-1\bigr)
	\\
	&\leq\frac14
	\bigl(\mathcal P_{\qq}(n;k)+D_{\qq,j_*}(n)\bigr)
	{}+C(d,\gamma)\widehat\Delta_h^{\qq}(n)\,.
\end{align*}
This is~\eqref{e.renormalization.service}.

\smallskip

Assume
\begin{equation*}
	k<n\,,
	\qquad
	\mathcal P_{\qq}(n;k)+D_{\qq,j_*}(n)\leq\eta\,.
\end{equation*}
If~$d^{-1}\Delta_{n,n+2L}^{\qq}>\varepsilon\sigma$, we are in Case~5. Otherwise, we are in Case~3, and
\begin{equation*}
	\mathcal P_{\qq}(n;k)+D_{\qq,j_*}(n)+\Delta_{n,n+2L}^{\qq}\leq(c+d)\varepsilon\sigma\leq c_0\varepsilon^{\nf12}\sigma\,.
\end{equation*}
Moreover,~$n\geq k+h=k+2Q$ by~\eqref{e.renormalization.input}. Since~$2\varepsilon\leq c_0$ and~$B\geq B_0$, the eccentricity bound~\eqref{e.bridge.eccentricity.input} follows from~\eqref{e.renormalization.eccentricity}. Lemma~\ref{l.projective.step} gives~$d_{\rm pr}([\m],[\m_+])\leq\varepsilon\leq1$, and~\eqref{e.renormalization.containment} places both cubes inside~$\cu_{2j_*}$.

\smallskip

Since~$\sigma\leq\varepsilon$, the chosen length satisfies
\begin{equation*}
	L-L_0\geq\frac32\log_3(\sigma^{-1})\geq\log_3\bigl((\varepsilon^{\nf12}\sigma)^{-1}\bigr)\,.
\end{equation*}
We may therefore apply Proposition~\ref{p.successful.short.bridge} with~$\sigma$ replaced by~$\varepsilon^{\nf12}\sigma$, without changing~$L$. We obtain~\eqref{e.successful.short.near} with this smaller tolerance. We apply Proposition~\ref{p.two.grid.transport} with~$\delta=\varepsilon^{\nf12}\sigma$ and~$\rho=\sigma^{\frac18(1-\gamma)}$.

\smallskip

It remains to verify that the small input in Case~3 pays for the transfer of the history. The fixed-span estimate~\eqref{e.fixed.geometry.fixed.span.propagation} and the determinant bound give
\begin{equation*}
	\mathcal P_{\qq}(n+2L;k)+D_{\qq,j_*}(n+2L)\leq C(d,\gamma)(1+L)\varepsilon\sigma\,.
\end{equation*}
By~\eqref{e.bridge.length.choice},
\begin{equation*}
	(1+L)\varepsilon\sigma3^{\frac12(1-\gamma)L}\leq C(d,\gamma)\varepsilon\bigl(1+\log(\sigma^{-1})\bigr)\sigma^{\frac18(1-\gamma)+\frac18(1+7\gamma)}\,.
\end{equation*}
The logarithmic choice of~$L$ turns the transport cost into a power of~$\sigma^{-1}$, leaving the positive power~$\sigma^{\frac18(1+7\gamma)}$ beyond the required~$\rho=\sigma^{\frac18(1-\gamma)}$. Since~$\bigl(1+\log(\sigma^{-1})\bigr)\sigma^{\frac18(1+7\gamma)}$ is bounded for~$0<\sigma\leq1$, decreasing~$\varepsilon_0(d,\gamma)$ ensures the required old-grid smallness.

\smallskip

The choice of~$L_0$ gives~$L\geq C+\log_3(\rho^{-1})$, with~$C$ from Proposition~\ref{p.two.grid.transport}. The eccentricity bound~\eqref{e.renormalization.eccentricity} gives
\begin{equation*}
	n-j_*\geq B\log_3(2+\Pi)+\frac L\varepsilon\log\mathfrak e(\m)\,.
\end{equation*}
Taking~$B_0\geq C(L+1)$ and decreasing~$\varepsilon_0$ so that~$\nf{L_0}{\varepsilon_0}\geq\nf{2C}{\log3}$ verifies~\eqref{e.two.grid.separation.input}. We have already checked the distance bound in~\eqref{e.two.grid.geometry.length.input} and that both cubes lie in~$\cu_{2j_*}$ as required by~\eqref{e.two.grid.cubes.input}. The preceding application of~\eqref{e.successful.short.near} gives the matrix comparison with~$\delta=\varepsilon^{\nf12}\sigma$. Thus Proposition~\ref{p.two.grid.transport} applies. Its conclusion and~$\delta\leq\rho$ prove~\eqref{e.renormalization.transport.output}. Since~$\sigma\leq\varepsilon_0$, the second inequality in~\eqref{e.local.parameter.choice} yields
\begin{equation*}
	\mathcal P_{\qq_+}(n+L;n+L)+D_{\qq_+,j_*}(n+L)
	\leq C(d,\gamma)\sigma^{\frac18(1-\gamma)}\leq1\,.
\end{equation*}
This proves the first alternative in~\eqref{e.renormalization.output}.

\smallskip

\emph{Step 3: Conditions on the output.} We prove~\eqref{e.renormalization.eccentricity.output} and~\eqref{e.renormalization.output} in every case. Consider first the cases in which the geometry is unchanged. In each of them,~$\m'=\m$ and~$k'=k$, so~\eqref{e.renormalization.eccentricity.output} is exactly~\eqref{e.renormalization.eccentricity}. In Cases~1,~2, and~4,~$n'=n+h$ by~\eqref{e.renormalization.output.data}. If~$k=n$, then~$n'=k'+h$; if~$k<n$, then the input condition~\eqref{e.renormalization.input} gives~$n'=n+h\geq k'+h$. In Case~5,~$k<n$ and~$n'=n+2L\geq n\geq k'+h$ by~\eqref{e.renormalization.input}. Thus~$k'<n'$ and~$n'\geq k'+h$ in every unchanged-geometry case, which is the second alternative in~\eqref{e.renormalization.output}.

\smallskip

In Case~3,~$k'=n'=n+L$. Lemma~\ref{l.projective.step}, the inequality~$k\leq n$, and~\eqref{e.renormalization.eccentricity} give
\begin{equation*}
	\frac12\log\bigl(|\m_+|\,|\m_+^{-1}|\bigr)\leq\frac{\varepsilon}{L}\left(k-j_*-\left\lceil B\log_3(2+\Pi)\right\rceil\right)+\varepsilon\leq\frac{\varepsilon}{L}\left(n+L-j_*-\left\lceil B\log_3(2+\Pi)\right\rceil\right)\,.
\end{equation*}
This is~\eqref{e.renormalization.eccentricity.output} with~$k'=n+L$. The bound~\eqref{e.renormalization.transport.output} proves the first alternative in~\eqref{e.renormalization.output}.

\smallskip

The split between~$k=n$ and~$k<n$, the split at~$\mathcal P_{\qq}(n;k)+D_{\qq,j_*}(n)=\eta$, and the two strict-versus-nonstrict determinant thresholds are disjoint and exhaustive. This completes the proof.
\end{proof}

\section{Closing the duality gap}
\label{s.polynomial.entry.proof}

The proof of Theorem~\ref{t.polynomial.entry} has two parts. We first find cubes on which the coarse matrices fluctuate little, their means change little with scale, and the shape of the cubes is adapted to those means. We then use the equation to show that the annealed coarse-grained matrix nearly agrees with its dual. Proposition~\ref{p.global.selection} carries out the first part by iterating Proposition~\ref{p.scale.selection}; Proposition~\ref{p.response.transfer} gives the second in a single application. The distinction matters: the errors reduced during the iteration are not the duality gap that we ultimately need to close.

\paragraph{Fluctuations and means.}
On a fixed grid~$\qq$, the fluctuation~$V_{j,m}^{\qq}$ measures the difference between a random coarse block and its expectation, normalized by~$\bfAhom_{m,\qq}$. The mean excess~$P_{j,m}^{\qq}-\Itwod$ compares the means at scales~$j$ and~$m$. The histories and drift retain these fluctuations and changes of the mean at earlier scales, with weights that decrease as those scales become more distant. Thus smallness of~$\mathcal P_{\qq}(m;m)+D_{\qq,j_*}(m)$ controls both the random fluctuations and the recent changes of the mean. These estimates concern past and present scales: they do not yet compare~$\bfAhom_{m,\qq}$ with its homogenized limit, or control its future changes.

The duality gap instead compares the annealed coarse-grained matrix with its dual. To measure it, let~$\kappa_m\geq1$ be the least number such that
\begin{equation*}
	\bfAhom_{m,\qq}\leq\kappa_m\mathbf R\bfAhom_{m,\qq}^{-1}\mathbf R\,.
\end{equation*}
The quantity~$\kappa_m-1$ measures their relative separation and controls the coarse contrast through~\eqref{e.response.euclidean.contrast} on the Euclidean grid. Concentration around the mean does not compare these two matrices, and they may remain far apart even if each changes little with scale. The equation supplies the additional relation needed to show that they are close to one another.

\paragraph{Initialization.}
Proposition~\ref{p.initial.fixed.grid.scale} first obtains small fluctuation and mean errors on the Euclidean grid. Averaging makes the mean of many small-cube blocks fluctuate less than an individual block. Subadditivity places the large-cube block below this average, and the expectation of their difference is the decrease of the annealed block. Thus either averaging improves the error or a change of the mean contributes to the determinant loss. The fixed-grid estimate reduces the error and drift, apart from a cost measured by this loss. The initial error is polynomially bounded in~$\Pi$, while the total determinant loss on this grid is at most~$C\log(2+\Pi)$. Taking the logarithm makes both effects additive: a fixed-factor contraction gives a fixed decrease, and any growth is bounded in terms of determinant loss. Summing these changes yields a scale~$n_0$ within~$C\log(2+\Pi)$ further scales at which~$\mathcal P_{\Id}(n_0;n_0)+D_{\Id,j_*}(n_0)\leq\eta_{\rm init}$, for any fixed tolerance~$\eta_{\rm init}>0$.

This smallness is obtained at the selected scale~$n_0$. The earlier errors remain in its history with their decaying weights; they need not themselves have been small. Nor have we yet bounded the duality gap at~$n_0$. We have obtained the initial error bound needed to begin adjusting the grid.

\paragraph{Choosing the geometry.}
To compare the primal and dual coarse-grained matrices, we must estimate gradients and fluxes on the same cubes. The difficulty is directional: a gradient of a given size can cost very different amounts of energy in different directions. The canonical metric describes how to stretch the cubes so that the gradient and flux estimates are balanced. The comparison of optimizer energies in Proposition~\ref{p.response.transfer} makes the benefit explicit: it contains the factor
\begin{equation*}
	|\qq\m^{-\nf12}|\,|\m^{\nf12}\qq^{-1}|\,.
\end{equation*}
This measures the mismatch between the cube shape and the energy metric~$\m$. It is at most three on the adapted grid~$\qq=\mathcal Q(\m)$. On a Euclidean grid it is~$(|\m|\,|\m^{-1}|)^{\nf12}$, which can grow with the reference ellipticity ratio~$\Pi$. Thus an error that is small on Euclidean cubes can still be multiplied by a large factor when we use it to estimate the duality gap.

Keeping Euclidean cubes would therefore require a tolerance that decreases with~$\Pi$. But reaching a smaller tolerance makes the cost of each determinant loss larger in the initialization estimate. Requesting a tolerance that is a negative power of~$\Pi$ would lose the bound of~$C\log(2+\Pi)$ on the number of scales. By adapting the cube shapes instead, we can work with a fixed small tolerance, however large~$\Pi$ is.

\paragraph{Scale selection.}
Proposition~\ref{p.global.selection} repeatedly applies Proposition~\ref{p.scale.selection} to reduce the errors on a fixed grid and change its shape when those errors are small enough and the mean changes little. The estimates on the old grid transfer to the new one: the errors may grow, but the small input keeps the new initial error bounded. The mean coarse matrix may still change as the scale increases, so the relative energy costs in different directions may change as well. Cubes adapted to an earlier mean need not balance the estimates at a later scale. We therefore adjust their shape during the iteration, using the mean matrices at the new scales.

To count the steps, the proof combines the logarithmic size of the error with the distance from the grid metric to the canonical metric. A fixed-grid step reduces the error up to determinant loss; a change of shape reduces the geometric mismatch, while allowing the error to increase. Summing the determinant losses and accounting for comparisons between grids bounds the number of steps by~$C\log(2+\Pi)$. We stop on an interval~$[s,t]$, with~$t=s+H$, on which the histories are small, the mean changes little, and the grid is canonical for a block close to~$\bfAhom_{s,\qq}$. These conditions allow us to use the equation to close the duality gap. We have not required~$\kappa_m-1$ to contract at each selection step.

\paragraph{Closing the gap.}
Proposition~\ref{p.response.transfer} now relates the duality gap to the gradients and fluxes of the primal and dual optimizers. Since the mean changes little between~$s$ and~$t$, the optimizer on the large cube is close in energy to the optimizers on its smaller cubes. The histories control the averages of their centered gradients and fluxes over those cubes and their descendants. Using the equation, we turn this control of averages into a bound on the centered optimizer energies, which gives
\begin{equation*}
	\kappa_t-1\leq\zeta\kappa_s\,.
\end{equation*}
Here~$\zeta$ is small when~$H$ is large and the selected errors and determinant loss are small. Crucially,~$\kappa_s$ appears only to the first power, and~$\zeta$ can be chosen small even when~$\kappa_s-1$ is large. Adapting the cubes and choosing the primal and dual loads together make this possible without a further loss depending on~$\Pi$.

Set~$r=\nf{\det\bfAhom_{s,\qq}}{\det\bfAhom_{t,\qq}}$. The small determinant loss on~$[s,t]$ gives~$\bfAhom_{t,\qq}\leq\bfAhom_{s,\qq}\leq r\bfAhom_{t,\qq}$, with~$r$ close to one. Comparing the dual blocks as well gives~$\kappa_s\leq r^2\kappa_t$. Hence, writing~$\omega=\zeta r^2$, we obtain
\begin{equation*}
	\kappa_t-1\leq\omega\kappa_t
	\qquad\Longrightarrow\qquad
	\kappa_t-1\leq\frac{\omega}{1-\omega}\qquad(\omega<1)\,.
\end{equation*}
This is the step that makes the duality gap small, regardless of how large it was initially: the same unknown~$\kappa_t$ occurs on both sides, and we rearrange the inequality to bound it. There is no further iteration of the gap across successive windows. We first choose~$H$ and the error tolerances for the desired accuracy, then find an interval satisfying those requirements, and apply this argument once.

Finally, we compare the annealed coarse-grained matrices on the selected adapted cube with those on a larger Euclidean cube. The geometric eccentricity of the selected grid is polynomially bounded in~$\Pi$, so this comparison costs another~$C\log(2+\Pi)$ scales. On the Euclidean grid the annealed matrices decrease with scale and their duals increase, so the resulting relative primal--dual bound persists at all larger scales. This completes the proof of Theorem~\ref{t.polynomial.entry}. The subsequent small-contrast iteration in Subsection~\ref{ss.algebraic.convergence} has a different purpose: starting from this small gap, it proves a rate of decay of the coarse-graining defect and yields Theorem~\ref{t.algebraic.convergence}.

\smallskip

Throughout this section, we use the choices of parameters for Proposition~\ref{p.scale.selection} made in Subsection~\ref{s.scale.selection.proof}; in particular,~$h=2Q$.

\subsection{Initialization}
\label{s.initialization}

Initialization is distinct from the startup step in Proposition~\ref{p.scale.selection}: here smallness must be proved, whereas startup assumes a bounded error on a fresh grid. The adapted cube bound gives only a bound polynomial in~$\Pi$. On the Euclidean grid, the total logarithmic determinant loss is at most~$C\log(2+\Pi)$, so the fixed-grid recurrence reduces that initial bound to a prescribed tolerance within a logarithmic number of scales. This reduction is made once; subsequent geometry updates transport the acquired control instead of estimating the errors anew from the adapted cube bound.

\begin{proposition}[Euclidean initialization]
\label{p.initial.fixed.grid.scale}
Fix~$h(d,\gamma)\in\N$ and~$\varepsilon_0(d,\gamma)\in(0,1)$ from Proposition~\ref{p.scale.selection}. Let~$0<\sigma\leq\varepsilon\leq\varepsilon_0$, and fix~$L(\varepsilon,\sigma,d,\gamma)\in\N$ and~$B_0(\varepsilon,\sigma,d,\gamma)\geq1$ furnished there. Let~$j_*$ satisfy~\eqref{e.source.lower.scale}. For every~$\eta_{\rm init}\in(0,1]$, there exists~$C(d,\gamma,\eta_{\rm init})<\infty$ such that the following holds for every~$B\geq B_0$. Set
\begin{equation*}
	R\coloneqq j_*+\left\lceil B\log_3(2+\Pi)\right\rceil\,.
\end{equation*}

\smallskip

There is a deterministic scale~$n_0$ such that
\begin{equation}
	R\leq n_0\leq R+\left\lceil C(d,\gamma,\eta_{\rm init})\log_3(2+\Pi)\right\rceil
	\qquad\text{and}\qquad
	\mathcal P_{\Id}(n_0;n_0)+D_{\Id,j_*}(n_0)\leq\eta_{\rm init}\,.
	\label{e.renormalization.entry}
\end{equation}

\smallskip

Moreover,
\begin{equation}
	\frac12\mathbf R\bfE^{-1}\mathbf R
	\leq\bfAhom_{n_0,\Id}\leq2\bfE
	\qand
	d_{\rm pr}\!\left([\Id],\left[\cmet(\bfAhom_{n_0,\Id})\right]\right)
	\leq C(d)\log(2+4\Pi)\,.
	\label{e.initial.geometry.bounds}
\end{equation}
\end{proposition}

\begin{proof}
Since~$\mathcal Q(\Id)=\Id$, the fixed-geometry estimates apply on the Euclidean grid. The proof has four steps. We first bound the Euclidean error directly from~\eqref{e.coarse.ellipticity}. Step~2 derives a decrease for its logarithm whenever the error and the drift are not yet small. Step~3 sums this inequality and the determinant losses to select~$n_0$. Step~4 verifies the geometric conditions. The containing region for the subsequent iteration is chosen in the global argument. Throughout the proof,~$C$ may change from line to line and depends only on~$d$ and~$\gamma$. After choosing~$\eta_0$, we indicate every occurrence for which~$C$ also depends on~$\eta_0$.

\smallskip

\emph{Step 1: A bound for the initial error.} In this step, we prove
\begin{equation}
	\mathcal P_{\Id}(m;j_*)+D_{\Id,j_*}(m)\leq C(2+\Pi)^C
	\qquad(m\geq j_*)\,.
	\label{e.initial.crude.profile}
\end{equation}
We bound the normalized blocks first and then insert these bounds into the histories, drift, and error.

\smallskip

For each~$j\geq j_*$, repeat the construction~\eqref{e.source.multiplier} with~$j$ in place of~$j_*$ and use the standard-cube conclusion of~\eqref{e.source.adapted.bound} on~$\cu_j$. The resulting random variable has~$L^Q$ norm at most two, so~$\bfAhom_{j,\Id}\leq2\bfE$. The order~$\bfAhom_*(U)\leq\bfAhom(U)$ gives~$\bfAhom_{j,\Id}\geq\mathbf R\bfAhom_{j,\Id}^{-1}\mathbf R\geq\frac12\mathbf R\bfE^{-1}\mathbf R$. Together with the reference comparison~$\bfE\leq6\Pi\mathbf R\bfE^{-1}\mathbf R$, these estimates give, uniformly in~$j$,
\begin{equation}
	\frac12\mathbf R\bfE^{-1}\mathbf R\leq\bfAhom_{j,\Id}\leq2\bfE
	\qand
	\|V_j^{\Id}\|_{L^Q(S_Q)}\leq C\Pi\,.
	\label{e.initial.source.bounds}
\end{equation}
For~$j_*\leq j\leq m$, the aligned-subdivision estimate gives~$\bfAhom_{m,\Id}\leq\bfAhom_{j,\Id}$. Combining~\eqref{e.initial.source.bounds} with~\eqref{e.matrix.reference.comparison} at the two scales gives
\begin{equation}
\left\{
\begin{aligned}  
& \bfAhom_{m,\Id}\leq\bfAhom_{j,\Id}\leq24\Pi\bfAhom_{m,\Id}, \\ 
& \Itwod\leq P_{j,m}^{\Id}\leq24\Pi\Itwod, \\ 
& 0 \leq\Delta_{j,m}^{\Id}\leq2d\log(24\Pi)\leq C\log(2+\Pi)
,.
\end{aligned} 
\right. 
\label{e.initial.normalization.bounds}
\end{equation}
Changing the normalization from scale~$j$ to scale~$m$, using stationarity for the translated scale-$j$ cell, we get
\begin{equation*}
	\E\bigl[|V_{j,m}^{\Id}(z)|^Q\bigr]
	\leq e^{Q\Delta_{j,m}^{\Id}}\E\bigl[|V_j^{\Id}|_{S_Q}^Q\bigr]
	\qquad\bigl(z\in3^j\Zd\cap\cu_m\bigr)\,.
\end{equation*}
There are~$3^{d(m-j)}$ such cells. In the definition of~$\history_{\Id}^{\rm fluc}(m)$, we bound the supremum over the scale-$j$ cells by the corresponding sum. Using~\eqref{e.initial.source.bounds}--\eqref{e.initial.normalization.bounds}, we find
\begin{equation*}
	\history_{\Id}^{\rm fluc}(m)\leq\sum_{j=j_*}^{m}3^{-(Q\rho_{\max}-d)(m-j)}e^{Q\Delta_{j,m}^{\Id}}\E\bigl[|V_j^{\Id}|_{S_Q}^Q\bigr]\leq C(2+\Pi)^C\,.
\end{equation*}
The series is bounded uniformly in~$m$, since
\begin{equation*}
	Q\rho_{\max}-d=Q\gamma+\frac14(1-\gamma)>0\,.
\end{equation*}
The middle inequality in~\eqref{e.initial.normalization.bounds} also gives
\begin{equation*}
	\meanpenalty_Q(P_{j,m}^{\Id})\leq\bigl(1+2d(24\Pi-1)\bigr)^Q-1
	\qand
	\history_{\Id}^{\rm mean}(m;j_*)\leq C(2+\Pi)^C\,.
\end{equation*}

\smallskip

The increments in the definition of~$D_{\Id,j_*}(m)$ are positive semidefinite. Discarding their weights and telescoping gives
\begin{equation*}
	D_{\Id,j_*}(m)\leq\sum_{j=j_*+1}^{m}\tr\bigl(P_{j-1,m}^{\Id}-P_{j,m}^{\Id}\bigr)=\tr(P_{j_*,m}^{\Id}-\Itwod)\leq C\Pi\,.
\end{equation*}
The three terms in the definition of the error~$\mathcal P_{\Id}(m;j_*)$ now satisfy
\begin{align*}
	\mathcal P_{\Id}(m;j_*)
	&=3^{-\frac14(1-\gamma)(m-j_*)}
	\bigl(1+\meanpenalty_Q(P_{j_*,m}^{\Id})\bigr)\history_{\Id}(j_*)
	+\history_{\Id}^{\rm mean}(m;j_*)\\
	&\qquad +\sum_{j=j_*+1}^{m}3^{-\frac14(1-\gamma)(m-j)}e^{Q\Delta_{j,m}^{\Id}}
	\E\bigl[|V_j^{\Id}|_{S_Q}^Q\bigr]\\
	&\leq C(2+\Pi)^C\Biggl(3^{-\frac14(1-\gamma)(m-j_*)}
	+\sum_{j=j_*}^{m-1}3^{-\frac14(1-\gamma)(m-1-j)}
	+\sum_{j=j_*+1}^{m}3^{-\frac14(1-\gamma)(m-j)}\Biggr)\\
	&\leq C(2+\Pi)^C\,.
\end{align*}
Combining this with the preceding bound for~$D_{\Id,j_*}(m)$ proves~\eqref{e.initial.crude.profile}.

\smallskip

\emph{Step 2: Logarithmic decrease.} Let~$C$ be the constant in~\eqref{e.fixed.geometry.carried.majorization}, enlarged so that~$C\geq1$, and choose~$\eta_0\in(0,1]$ so that
\begin{equation*}
	C\eta_0\leq\eta_{\rm init}\,.
\end{equation*}
For~$\ell\in\Z_{\geq0}$, set
\begin{equation}
	T_\ell\coloneqq R+h+\ell h
	\qand
	W_\ell\coloneqq\log\!\left(1+\eta_0^{-1}\bigl(\mathcal P_{\Id}(T_\ell;j_*)+D_{\Id,j_*}(T_\ell)\bigr)\right)\,.
	\label{e.initial.synchronized.scales}
\end{equation}
The scales~$T_\ell$ are spaced for the synchronized propagation estimate. The logarithm converts contraction above the threshold~$\eta_0$ into a fixed decrease, with determinant variation as an additive error. A large determinant loss may increase~$W_\ell$, but the sum of these increases is controlled; no small-loss assumption is imposed on individual steps. Since~$T_\ell\geq j_*+h$, the synchronized propagation estimate~\eqref{e.fixed.geometry.synchronized.propagation} applies with the second argument of the error equal to~$j_*$.

\smallskip

In this step, we prove that, whenever
\begin{equation*}
	\mathcal P_{\Id}(T_\ell;j_*)+D_{\Id,j_*}(T_\ell)>\eta_0\,,
\end{equation*}
we have
\begin{equation}
	W_{\ell+1}\leq W_\ell-\log\frac{16}{9}+C\widehat\Delta_h^{\Id}(T_\ell)\,.
	\label{e.initial.log.decrement}
\end{equation}

\smallskip

We derive this from the propagation estimate by an elementary scalar calculation. Suppose that~$x\geq1$,~$y\geq0$, and
\begin{equation*}
	x'\leq\frac18e^{Qy}x+C\eta_0^{-1}(e^{Qy}-1)\,.
\end{equation*}
Since~$e^{Qy}-1\leq Qye^{Qy}$ and~$\nf{1+\nf{x}{8}}{1+x}\leq\nf{9}{16}$, we have
\begin{equation*}
	\frac{1+x'}{1+x}\leq e^{Qy}\left(\frac9{16}+\frac{CQ}{2\eta_0}y\right)\leq\frac9{16}\exp \left(\left(Q+\frac{8CQ}{9\eta_0}\right)y\right)\leq\frac9{16}e^{Cy}\,.
\end{equation*}
The second inequality follows from~$1+t\leq e^t$, and the last constant depends only on~$d$,~$\gamma$, and~$\eta_0$. In~\eqref{e.fixed.geometry.synchronized.propagation}, take~$x=\eta_0^{-1}\bigl(\mathcal P_{\Id}(T_\ell;j_*)+D_{\Id,j_*}(T_\ell)\bigr)$,~$x'=\eta_0^{-1}\bigl(\mathcal P_{\Id}(T_{\ell+1};j_*)+D_{\Id,j_*}(T_{\ell+1})\bigr)$, and~$y=\widehat\Delta_h^{\Id}(T_\ell)$. Taking logarithms proves~\eqref{e.initial.log.decrement}.

\smallskip

\emph{Step 3: Selection of the initial scale.} We use~\eqref{e.initial.log.decrement} to find~$n_0$ satisfying~\eqref{e.renormalization.entry}. If the sum of the error and the drift stayed above~$\eta_0$, the fixed decrease would eventually exceed the initial logarithm and all determinant errors. The bound~\eqref{e.initial.crude.profile} gives
\begin{equation}
	W_0\leq C(d,\gamma,\eta_0)\log_3(2+\Pi)\,.
	\label{e.initial.starting.log}
\end{equation}
Since~$T_0\geq j_*+h$, estimates~\eqref{e.fixed.geometry.synchronized.multiplicity} and~\eqref{e.initial.normalization.bounds} give, for every integer~$K\geq1$,
\begin{equation*}
	\sum_{\ell=0}^{K-1}\widehat\Delta_h^{\Id}(T_\ell)\leq h\Delta_{T_0+1-h,T_0+Kh}^{\Id}\leq C\log_3(2+\Pi)\,.
\end{equation*}
Here~$T_\ell=T_0+\ell h$,~$T_0+1-h=R+1\geq j_*$, and~$h=2Q$ depends only on~$d$ and~$\gamma$.

\smallskip

Together with~\eqref{e.initial.starting.log}, the preceding estimate gives
\begin{equation*}
	W_0+C\sum_{\ell=0}^{K-1}\widehat\Delta_h^{\Id}(T_\ell)
	\leq C(d,\gamma,\eta_{\rm init})\log_3(2+\Pi)\,.
\end{equation*}
If necessary, increase the constant in the preceding estimate and set
\begin{equation*}
	K\coloneqq\left\lceil C(d,\gamma,\eta_{\rm init})\log_3(2+\Pi)\right\rceil\,.
\end{equation*}
Then~$K\geq1$,~$K\leq C(d,\gamma,\eta_{\rm init})\log_3(2+\Pi)$, and
\begin{equation*}
	K\log\frac{16}{9}>W_0+C\sum_{\ell=0}^{K-1}\widehat\Delta_h^{\Id}(T_\ell)\,.
\end{equation*}
If~$\mathcal P_{\Id}(T_\ell;j_*)+D_{\Id,j_*}(T_\ell)>\eta_0$ for every~$0\leq\ell<K$, then summing~\eqref{e.initial.log.decrement} would give
\begin{equation*}
	0\leq W_K\leq W_0-K\log\frac{16}{9}+C\sum_{\ell=0}^{K-1}\widehat\Delta_h^{\Id}(T_\ell)<0\,.
\end{equation*}
This is impossible. Hence there is~$\ell<K$ such that
\begin{equation*}
	\mathcal P_{\Id}(T_\ell;j_*)+D_{\Id,j_*}(T_\ell)\leq\eta_0\,.
\end{equation*}
Set~$n_0\coloneqq T_\ell$. Since~$\log_3(2+\Pi)\geq1$, the initial displacement~$h$ and the length~$Kh$ are bounded by~$C(d,\gamma,\eta_{\rm init})\log_3(2+\Pi)$. Thus~$n_0$ has the range asserted in~\eqref{e.renormalization.entry}. Applying~\eqref{e.fixed.geometry.carried.majorization} at~$n_0$ gives
\begin{equation*}
	\mathcal P_{\Id}(n_0;n_0)+D_{\Id,j_*}(n_0)=\history_{\Id}(n_0)+D_{\Id,j_*}(n_0)\leq C\bigl(\mathcal P_{\Id}(n_0;j_*)+D_{\Id,j_*}(n_0)\bigr)\leq C\eta_0\leq\eta_{\rm init}\,.
\end{equation*}
The last inequality follows from the choice of~$\eta_0$. This proves the smallness assertion in~\eqref{e.renormalization.entry}. Passing to the fresh error~$\mathcal P_{\Id}(n_0;n_0)=\history_{\Id}(n_0)$ retains all errors from~$j_*$ to~$n_0$; it does not discard the earlier scales to obtain smallness.

\smallskip

\emph{Step 4: Geometry and the first application.} It remains to prove~\eqref{e.initial.geometry.bounds} and verify the input conditions for Proposition~\ref{p.scale.selection}. Put~$A_0=\bfAhom_{n_0,\Id}$. The first estimate in~\eqref{e.initial.source.bounds}, evaluated at~$n_0$, gives the two-sided block comparison in~\eqref{e.initial.geometry.bounds}. Together with~\eqref{e.matrix.reference.comparison}, it yields~$(12\Pi)^{-1}\bfE\leq A_0\leq2\bfE$. Monotonicity and homogeneity of the geometric mean, applied also to the inverse bounds, therefore give
\begin{equation*}
	(24\Pi)^{-\nf12}\cmet(\bfE)\leq \cmet(A_0)\leq(24\Pi)^{\nf12}\cmet(\bfE)\,.
\end{equation*}
The lower-right principal blocks in~\eqref{e.matrix.canonical.order} give~$\s_{*,0}\leq \cmet(\bfE)\leq\s_0$, so~$|\cmet(\bfE)|\,|\cmet(\bfE)^{-1}|\leq\Pi$. The triangle inequality for~$d_{\rm pr}$ now gives~$d_{\rm pr}([\Id],[\cmet(A_0)])\leq\frac12\log\Pi+\frac12\log(24\Pi)$, and hence
\begin{equation*}
	d_{\rm pr}\!\left([\Id],\left[\cmet(\bfAhom_{n_0,\Id})\right]\right)
	\leq C(d)\log(2+4\Pi)\,.
\end{equation*}
This proves~\eqref{e.initial.geometry.bounds}.

\smallskip

For the data
\begin{equation*}
	(\m,\qq,k,n)=(\Id,\Id,n_0,n_0)\,,
\end{equation*}
the lower bound for~$n_0$ and~\eqref{e.renormalization.entry} give
\begin{equation*}
	0=\frac12\log\bigl(|\Id|\,|\Id^{-1}|\bigr)
	\leq\frac{\varepsilon}{L}\left(n_0-j_*-\left\lceil B\log_3(2+\Pi)\right\rceil\right)
	\qand
	\mathcal P_{\Id}(n_0;n_0)+D_{\Id,j_*}(n_0)\leq1\,.
\end{equation*}
This supplies the grid eccentricity and error conditions for the first application of Proposition~\ref{p.scale.selection}. The global argument verifies~\eqref{e.renormalization.containment} separately.
\end{proof}

\subsection{Global selection}
\label{s.global.selection}

We iterate Proposition~\ref{p.scale.selection} until the metric defining the geometry is the canonical metric of a selected annealed block and the determinant changes little over a prescribed interval. Small errors alone do not ensure the first condition, and reaching a canonical metric does not ensure the second: the annealed block may still change appreciably at the next scales. The stopping rule below tests both conditions. To prove that it succeeds, we combine the logarithmic size of the errors with the distance to the canonical metric. Neither quantity must decrease separately; their combined change is controlled by determinant losses, whose sum is estimated across changes of geometry.

\begin{proposition}[Global selection]
\label{p.global.selection}
There exists~$\varepsilon(d,\gamma)\in(0,\varepsilon_0]$ with the following property. Fix an integer~$H\geq\max\{4,h\}$ and~$\sigma\in(0,\varepsilon]$, and choose~$L$,~$\eta$, and~$B_0$ as in Proposition~\ref{p.scale.selection}. There exists~$C(H,\sigma,d,\gamma)<\infty$ such that, for every~$B\geq B_0$ and every integer~$j_*$ satisfying
\begin{equation}
	j_*\geq\bigl\lceil C(B+1)\log_3(2+\Pi)+C_{\rm src}(d,\gamma)\log_3(2K_{\Psi_{\S}})\bigr\rceil\,,
	\label{e.global.selection.lower.scale}
\end{equation}
there are a deterministic positive block~$F$, a matrix~$\m=\cmet(F)$, the grid~$\qq=\mathcal Q(\m)$, and scales~$s<t$ satisfying
\begin{equation}
	t=s+H\,,\qquad s\geq j_*+\left\lceil B\log_3(2+\Pi)\right\rceil\,,\qquad t\leq j_*+\left\lceil(B+C)\log_3(2+\Pi)\right\rceil\,.
	\label{e.global.selection.scales}
\end{equation}
The selected adapted cube~$\cus_t^{\qq}$ is contained in~$\cu_{2j_*}$, and
\begin{equation}
	(1-\varepsilon^{\nf12}\sigma)F\leq\bfAhom_{s,\qq}\leq(1+\varepsilon^{\nf12}\sigma)F\,,\qquad d^{-1}\Delta_{s,t}^{\qq}<\sigma\,.
	\label{e.global.selection.calibration}
\end{equation}
Moreover,
\begin{equation}
	\max\bigl\{\mathcal P_{\qq}(s;s),\mathcal P_{\qq}(t;s),\mathcal P_{\qq}(t;t)\bigr\}+D_{\qq,j_*}(s)+D_{\qq,j_*}(t)\leq C(H,d,\gamma)\sigma^{\frac18(1-\gamma)}\,,
	\label{e.global.selection.profile}
\end{equation}
and
\begin{equation}
	\bigl(|\m|\,|\m^{-1}|\bigr)^{\nf12}\leq(2+\Pi)^C\,.
	\label{e.global.selection.eccentricity}
\end{equation}
\end{proposition}

The metric is exactly canonical for~$F$, while the new-grid mean at~$s$ is close to~$F$. The determinant bound then keeps the means close throughout~$[s,t]$, and the three error terms retain the information needed at both endpoints. These estimates allow us to compare the primal and dual optimizer energies and thereby close the duality gap. That comparison is still to come; no small-contrast conclusion is asserted yet.

\begin{proof}
We first construct the sequence within a fixed containing cube. Step~2 defines the scalar~\eqref{e.global.selection.scalar} and fixes its constants. Step~3 estimates its change in each continuing alternative. Step~4 sums those estimates and the determinant losses to force a successful test, and Step~5 verifies the conclusions at the selected scales. The constants in the preliminary scalar estimates depend only on~$d$ and~$\gamma$.

\smallskip

\emph{Step 1: Construction and persistence.} We construct the continuing data and verify that every application of Proposition~\ref{p.scale.selection} is permitted until the test~\eqref{e.global.selection.test} succeeds or the fixed number of attempts is exhausted. We choose~$\varepsilon$ in Step~2. Fix~$H$,~$\sigma$,~$L$,~$\eta$, and~$B\geq B_0$. Allow at most~$J\coloneqq\lceil C_1(H,\sigma,d,\gamma)\log_3(2+\Pi)\rceil$ continuing steps, where~$C_1$ is fixed in Step~4. Apply Proposition~\ref{p.initial.fixed.grid.scale} with~$\eta_{\rm init}=1$ and denote its scale by~$n_0$. Each step described below advances the scale by at most~$2L+H+h$ and changes the metric by projective distance at most~$\varepsilon$. Starting from~$\Id$, every retained or tested metric therefore satisfies
\begin{equation*}
	\mathfrak e(\m_i)\leq e^{\varepsilon(J+1)}\qand |\mathcal Q(\m_i)|\leq2e^{\varepsilon(J+1)}\,.
\end{equation*}
The second inequality follows from~\eqref{e.rounded.grid.bounds}. The scale of any queried adapted cube, including the last length-$H$ test, is at most~$n_0+(2L+H+h)(J+2)$. To place all these cubes in~$\cu_{2j_*}$, it is therefore enough that
\begin{equation*}
	n_0+(2L+H+h)(J+2)+\biggl\lceil\log_3(2\sqrt d)+\frac{\varepsilon(J+1)}{\log3}\biggr\rceil\leq2j_*\,.
\end{equation*}
This follows from~\eqref{e.renormalization.entry} and~\eqref{e.global.selection.lower.scale} if~$C$ in the latter is sufficiently large in terms of~$H$,~$\sigma$,~$d$,~$\gamma$, and the chosen~$C_1$. These choices precede the iteration. Use the single random variable~$\RSZ$ defined in~\eqref{e.source.multiplier} to apply~\eqref{e.source.adapted.bound} throughout this region. In particular,~\eqref{e.renormalization.containment} holds at every attempted application; its verification uses only the bound~$J$ just imposed, not the later stopping argument.

\smallskip

Since~$\mathcal Q(\Id)=\Id$, apply Proposition~\ref{p.scale.selection} to~$(\Id,\Id,n_0,n_0)$ and begin the continuing sequence with its output
\begin{equation*}
	(\m_1,\qq_1,k_1,n_1)=(\Id,\Id,n_0,n_0+h)\,.
\end{equation*}
All subsequent continuing data will satisfy~$k_i<n_i$. In Cases~1,~2,~4, and~5, use the output of Proposition~\ref{p.scale.selection} as~$(\m_{i+1},\qq_{i+1},k_{i+1},n_{i+1})$.

\smallskip

On a change of geometry, form~$\m_+$ and~$\qq_+$ as in that proposition. If~$\m_+\ne \cmet(\bfAhom_{n_i+2L,\qq_i})$, apply the length-$h$ step for equal error arguments immediately after the change and retain
\begin{equation*}
	(\m_{i+1},\qq_{i+1},k_{i+1},n_{i+1})=(\m_+,\qq_+,n_i+L,n_i+L+h)\,.
\end{equation*}
If~$\m_+=\cmet(\bfAhom_{n_i+2L,\qq_i})$, test whether
\begin{equation}
	d^{-1}\Delta_{n_i+L,n_i+L+H}^{\qq_+}<\sigma\,.
	\label{e.global.selection.test}
\end{equation}
When the test succeeds, stop and retain~$F=\bfAhom_{n_i+2L,\qq_i}$,~$\m=\m_+$,~$\qq=\qq_+$,~$s=n_i+L$, and~$t=s+H$. Otherwise retain
\begin{equation*}
	(\m_{i+1},\qq_{i+1},k_{i+1},n_{i+1})=(\m_+,\qq_+,n_i+L,n_i+L+H)\,.
\end{equation*}
Thus one counted step contains either one application of Proposition~\ref{p.scale.selection}, or a change of geometry together with its following length-$h$ step or length-$H$ test. Pairing a partial geometry change with startup lets its geometric progress compensate for the possible startup increase. Upon reaching the candidate, the length-$H$ test serves a different purpose: it checks that the means on the new grid vary little enough to compare the optimizer energies at the two scales and close the duality gap. Failure records a definite determinant loss on the new grid, after which the iteration continues there.

\smallskip

The initial grid is Euclidean, so its logarithmic eccentricity is zero, and~\eqref{e.renormalization.entry} gives the initial input bound. Equations~\eqref{e.renormalization.eccentricity.output} and~\eqref{e.renormalization.output} preserve both conditions at every application of Proposition~\ref{p.scale.selection}. At a failed length-$H$ test the metric~$\m_+$ and the scale~$k_{i+1}$ have not changed, so~\eqref{e.renormalization.eccentricity} still holds; the remaining input condition is~$n_{i+1}\geq k_{i+1}+h$, which follows from~$H\geq h$. The length-$h$ step after a geometry change is applicable because the new error and the drift have sum at most one. Since all the required cubes lie inside~$\cu_{2j_*}$, as verified above, we can take every step up to~$J$ unless the stopping test succeeds earlier.

\smallskip

\emph{Step 2: The scalar and its constants.} A change of geometry can increase the error and the drift even as it improves the metric. We therefore combine their logarithmic size with the distance to the canonical metric. Taking a logarithm turns the exponential growth allowed by the fixed-grid estimates into a cost proportional to determinant loss. Summing the resulting estimates will bound both the obstructed steps and the further steps needed to reduce the error again. We will choose the constants so that the four estimates~\eqref{e.global.selection.h.step}--\eqref{e.global.selection.failed.test} hold; each has a fixed negative term and an error expressed by determinant losses. Let~$a\geq1$ be a constant to be chosen below. For every continuing index~$i\geq1$, define
\begin{equation}
	\Phi_i\coloneqq\eta\log\left(1+\frac{\mathcal P_{\qq_i}(n_i;k_i)+D_{\qq_i,j_*}(n_i)}{\eta}\right)+a\,d_{\rm pr}\left([\m_i],\left[\cmet(\bfAhom_{k_i,\qq_i})\right]\right)\,.
	\label{e.global.selection.scalar}
\end{equation}
The first term decreases, up to a determinant error, when the geometry stays fixed and the error plus drift exceeds~$\eta$. In the second term the comparison metric is evaluated at the initialization scale~$k_i$, not the current scale~$n_i$. This keeps that term constant during fixed-grid advances. When the geometry changes, we compare this retained target with the new candidate;~\eqref{e.global.selection.metric.loss} bounds their separation by the determinant loss accumulated since~$k_i$. A partial geometry change reduces the distance to its candidate by~$\varepsilon$, compensating for the possible increase of the first term. If the candidate is reached, we either stop or use the determinant loss from the failed test.

\smallskip

We scale the logarithm by~$\eta$ so that the coefficient of its determinant error is independent of~$\eta$. This allows us to choose~$\varepsilon$ independently of~$\sigma$. We give all parameter choices independently of~$\Pi$,~$B$, and~$j_*$. First record the two scalar estimates needed for the logarithmic term. For~$x\geq0$ and~$y\geq0$, differentiation gives
\begin{equation*}
	\frac{d}{dy}\left[\eta\log\left(1+\eta^{-1}\left(e^{Qy}\frac{x}{8}+C(e^{Qy}-1)\right)\right)\right]\leq Q\max\{1,C\}\,.
\end{equation*}
Indeed, the derivative is~$Q\eta e^{Qy}\nf{\nf{x}{8}+C}{\eta+e^{Qy}\nf{x}{8}+C(e^{Qy}-1)}$. When~$C\geq\eta$, this derivative is maximized at~$y=0$; otherwise it is bounded by~$Q\eta$. When~$x>\eta$,
\begin{equation*}
	\eta\log\frac{1+\frac{x}{8\eta}}{1+\frac{x}{\eta}}\leq-\eta\log\frac{16}{9}\,.
\end{equation*}
Consequently, if~$\mathcal P_{\qq}(m;n)+D_{\qq,j_*}(m)>\eta$, estimate~\eqref{e.fixed.geometry.synchronized.propagation} decreases the scaled logarithm of this sum by~$\eta\log(\nf{16}{9})$, up to a constant multiple of~$\widehat\Delta_h$.

\smallskip

For a fixed span~$\ell\geq1$, if~$0\leq x\leq1$ and~$x'\leq C\ell(x+e^{Qy}-1)$, then
\begin{equation}
	\eta\log\left(1+\frac{x'}{\eta}\right)\leq\log(1+C\ell)+Qy\,.
	\label{e.global.selection.scalar.span}
\end{equation}
To see this, use~$x+e^{Qy}-1\leq e^{Qy}$ and~$\eta\log(1+\nf{C\ell}{\eta})\leq\log(1+C\ell)$, valid for~$0<\eta\leq1$. This estimate applies to~\eqref{e.fixed.geometry.fixed.span.propagation}, with~$y$ equal to the determinant loss on the span. The coefficient of~$y$ is independent of~$\ell$ and~$\eta$.

\smallskip

We next bound the variation of the canonical metric, using the joint monotonicity and homogeneity of the matrix geometric mean proved after~\eqref{e.scale.selection.canonical.metric}. For positive blocks~$G\leq F$,
\begin{equation}
	d_{\rm pr}([\cmet(F)],[\cmet(G)])\leq\frac12\log\frac{\det F}{\det G}\,.
	\label{e.global.selection.metric.loss}
\end{equation}
Indeed, the~$2d$ eigenvalues of~$G^{-\nf12}FG^{-\nf12}$ are at least one and have product~$\nf{\det F}{\det G}$. Each is therefore at most this product, giving~$G\leq F\leq(\nf{\det F}{\det G})G$. Applying these bounds and their inverses to the two factors in the geometric mean bounds the geometric mean between~$(\nf{\det F}{\det G})^{-\nf12}$ and~$(\nf{\det F}{\det G})^{\nf12}$ times the mean for~$G$. Taking the lower-right principal block and then its inverse gives the same bounds for the generalized eigenvalues of~$\cmet(F)$ relative to~$\cmet(G)$. Their ratio is at most~$\nf{\det F}{\det G}$, which proves~\eqref{e.global.selection.metric.loss} by the definition of~$d_{\rm pr}$.

\smallskip

If instead~$(1-\delta)F\leq G\leq(1+\delta)F$, the two factors in the geometric mean are bounded in opposite directions, so the generalized eigenvalues of~$\cmet(G)$ relative to~$\cmet(F)$ lie between~$(\nf{1-\delta}{1+\delta})^{\nf12}$ and~$(\nf{1+\delta}{1-\delta})^{\nf12}$. Consequently,
\begin{equation}
	d_{\rm pr}([\cmet(F)],[\cmet(G)])\leq\frac12\log\frac{1+\delta}{1-\delta}\qquad(0\leq\delta<1)\,.
	\label{e.global.selection.metric.comparison}
\end{equation}

\smallskip

At each change of geometry, the verification in Subsection~\ref{s.scale.selection.proof} allows us to apply Proposition~\ref{p.successful.short.bridge} with comparison tolerance~$\delta\coloneqq\varepsilon^{\nf12}\sigma$. It gives
\begin{equation}
	(1-\delta)\bfAhom_{n_i+2L,\qq_i}\leq\bfAhom_{n_i+L,\qq_+}\leq(1+\delta)\bfAhom_{n_i+2L,\qq_i}\,.
	\label{e.global.selection.geometry.comparison}
\end{equation}
The logarithmic determinant can therefore increase by at most~$2d\log(1+\delta)$. Step~4 shows that each one-scale determinant loss is counted at most~$h+2$ times. We therefore retain a decrease of~$2aC(h+2)\log(1+\delta)$ at every geometry change, in addition to a fixed decrease, to cancel the determinant jumps between fixed-geometry intervals when we sum.

\smallskip

Fix~$C=C(d,\gamma)\geq d$ larger than every constant in the fixed-grid estimates above. Decrease~$\varepsilon\in(0,\varepsilon_0]$, depending only on~$d$ and~$\gamma$, so that, for every~$0<\sigma\leq\varepsilon$ and~$\delta=\varepsilon^{\nf12}\sigma$,
\begin{equation}
	\frac12\log\frac{1+\delta}{1-\delta}+2C(h+2)\log(1+\delta)\leq\min\left\{\frac\varepsilon2,\frac{C\sigma}{4}\right\}\,.
	\label{e.global.selection.comparison.choice}
\end{equation}
This is possible because~$\nf{\delta}{\sigma}=\varepsilon^{\nf12}$ and~$\nf{\delta}{\varepsilon}\leq\varepsilon^{\nf12}$. Now fix the constants~$L$,~$\eta$, and~$B_0$ supplied by Proposition~\ref{p.scale.selection}, and set~$c\coloneqq\frac12\eta\log(\nf{16}{9})$. Choose~$a\geq1$ with~$\nf{aC}{d}\geq4Q\max\{1,C\}$ and large enough that
\begin{equation}
	\frac{a\varepsilon}{2}\geq\log(1+Ch)+c\,,\qquad aC\varepsilon\sigma\geq4\bigl(\log(1+2CL)+\log(1+CH)+c\bigr)\,.
	\label{e.global.selection.weight.choice}
\end{equation}
The constants~$c$ and~$a$ depend only on~$H$,~$\sigma$,~$d$, and~$\gamma$.

\smallskip

\emph{Step 3: Decrease in each alternative.} In this step, we prove~\eqref{e.global.selection.h.step}--\eqref{e.global.selection.failed.test}. The constants are now fixed. On a length-$h$ step, the scale~$k_i$ and the metric remain fixed and~$\mathcal P_{\qq_i}(n_i;k_i)+D_{\qq_i,j_*}(n_i)>\eta$. The synchronized propagation estimate and the first scalar calculation give
\begin{equation}
	\Phi_{i+1}-\Phi_i\leq-c+\frac{aC}d\widehat\Delta_h^{\qq_i}(n_i)\,.
	\label{e.global.selection.h.step}
\end{equation}
This holds also when that step is a determinant obstruction. On a length-$2L$ obstruction the incoming sum is at most~$\eta$, so~\eqref{e.global.selection.scalar.span} applies. Since~$\Delta_{n_i,n_i+2L}^{\qq_i}>d\varepsilon\sigma$, the last condition in~\eqref{e.global.selection.weight.choice} absorbs~$\log(1+2CL)+c$ and gives
\begin{equation}
	\Phi_{i+1}-\Phi_i\leq-c+\frac{aC}d\Delta_{n_i,n_i+2L}^{\qq_i}\,.
	\label{e.global.selection.long.step}
\end{equation}

\smallskip

Consider a change of geometry which does not reach its canonical candidate. We prove
\begin{equation}
	\Phi_{i+1}-\Phi_i\leq-c-2aC(h+2)\log(1+\delta)+\frac{aC}d\left(\Delta_{k_i,n_i+2L}^{\qq_i}+\Delta_{k_{i+1},n_{i+1}}^{\qq_{i+1}}\right)\,.
	\label{e.global.selection.partial.change}
\end{equation}
By~\eqref{e.renormalization.geometry.update}, the update reduces its projective distance to that candidate by~$\varepsilon$. Apply~\eqref{e.global.selection.metric.loss} to the old-grid blocks at~$k_i$ and~$n_i+2L$, and~\eqref{e.global.selection.metric.comparison} to~\eqref{e.global.selection.geometry.comparison}. The triangle inequality gives
\begin{equation*}
	d_{\rm pr}\left([\m_{i+1}],\left[\cmet(\bfAhom_{k_{i+1},\qq_{i+1}})\right]\right)-d_{\rm pr}\left([\m_i],\left[\cmet(\bfAhom_{k_i,\qq_i})\right]\right)\leq-\varepsilon+\frac12\Delta_{k_i,n_i+2L}^{\qq_i}+\frac12\log\frac{1+\delta}{1-\delta}\,.
\end{equation*}
The new error and the drift have sum at most one. Applying~\eqref{e.global.selection.scalar.span} to the following length-$h$ step, and discarding the nonnegative incoming logarithm, bounds~$\Phi_{i+1}-\Phi_i$ by
\begin{equation*}
	\log(1+Ch)-a\varepsilon+\frac a2\log\frac{1+\delta}{1-\delta}+\frac a2\Delta_{k_i,n_i+2L}^{\qq_i}+Q\Delta_{k_{i+1},n_{i+1}}^{\qq_{i+1}}\,.
\end{equation*}
The parameter choices~\eqref{e.global.selection.comparison.choice}--\eqref{e.global.selection.weight.choice} now prove~\eqref{e.global.selection.partial.change}.

\smallskip

Suppose that the canonical candidate is reached but~\eqref{e.global.selection.test} fails. We prove
\begin{equation}
	\Phi_{i+1}-\Phi_i\leq-c-2aC(h+2)\log(1+\delta)+\frac{aC}d\Delta_{k_{i+1},n_{i+1}}^{\qq_{i+1}}\,.
	\label{e.global.selection.failed.test}
\end{equation}
The new projective distance is at most~$\frac12\log(\nf{1+\delta}{1-\delta})$. Applying~\eqref{e.global.selection.scalar.span} over the length-$H$ interval and discarding both nonnegative terms of~$\Phi_i$ gives
\begin{equation*}
	\Phi_{i+1}-\Phi_i\leq\log(1+CH)+\frac a2\log\frac{1+\delta}{1-\delta}+Q\Delta_{k_{i+1},n_{i+1}}^{\qq_{i+1}}\,.
\end{equation*}
The failed test gives~$\Delta_{k_{i+1},n_{i+1}}^{\qq_{i+1}}\geq d\sigma$. We use this lower bound to absorb the remaining terms. The comparison term and~$2aC(h+2)\log(1+\delta)$ together use at most~$\nf{aC\sigma}{4}$. The direct determinant term uses at most one quarter of~$(\nf{aC}{d})\Delta_{k_{i+1},n_{i+1}}^{\qq_{i+1}}$, and~$\log(1+CH)+c$ uses at most another quarter. These bounds prove~\eqref{e.global.selection.failed.test} using only the parameter choices already fixed.

\smallskip

\emph{Step 4: Determinant sum and the number of steps.} Fix any~$M$ continuing steps, indexed by~$1\leq i\leq M$, and let~$N$ be the number of changes of geometry among them. In this step, we prove
\begin{equation}
	cM\leq\Phi_1+aC(h+2)\log(24\Pi)\,.
	\label{e.global.selection.count}
\end{equation}
This will bound the number of continuing steps once we show that~$\Phi_1$ is logarithmic in~$2+\Pi$. It is not enough to count only the steps where a determinant test fails: small losses also occur during contraction, and both kinds can increase the logarithmic error. We sum every determinant contribution in the preceding inequalities. The only additional costs are jumps between grids and repeated appearances of the same loss, which we now bound.

\smallskip

The determinants telescope while the geometry is fixed. We must account for the jumps when it changes and for repeated appearances of each one-scale loss. Divide this finite sequence into maximal intervals with the same retained geometry, keeping two intervals distinct even if rounding produces the same grid. The old geometry ends at~$n_i+2L$ when it changes, and the new one begins at~$k_{i+1}=n_i+L$. The upper bound in~\eqref{e.global.selection.geometry.comparison} gives
\begin{equation*}
	\log\det\bfAhom_{k_{i+1},\qq_{i+1}}-\log\det\bfAhom_{n_i+2L,\qq_i}\leq2d\log(1+\delta)\,.
\end{equation*}
Within each fixed geometry the determinant decreases. Summing its losses from the beginning to the end of each such interval cancels all intermediate determinants except these comparison differences. The order~$\bfAhom_*(U)\leq\bfAhom(U)$ gives~$\det\bfAhom(U)\geq1$. At the initial Euclidean scale,~\eqref{e.initial.source.bounds} gives~$\bfAhom_{n_0,\Id}\leq2\bfE$. The order~$\mathbf R\bfE^{-1}\mathbf R\leq\bfE$ and the factor-six reference comparison give~$\det\bfE\leq(6\Pi)^d$, and hence
\begin{equation*}
	\det\bfAhom_{n_0,\Id}\leq2^{2d}\det\bfE\leq(24\Pi)^d\,.
\end{equation*}
Thus the sum of the fixed-geometry determinant losses, divided by~$d$, is at most
\begin{equation}
	\log(24\Pi)+2N\log(1+\delta)\,.
	\label{e.global.selection.determinant.telescope}
\end{equation}
The lower determinant bound follows by taking determinants in the inequality~$\bfAhom_*(U)\leq\bfAhom(U)$.

\smallskip

We count how often each one-scale loss occurs in~\eqref{e.global.selection.h.step}--\eqref{e.global.selection.failed.test}. Expanding a synchronized loss gives
\begin{equation*}
	\widehat\Delta_h^{\qq_i}(n_i)=\sum_{b=n_i+1}^{n_i+h}\sum_{r=b-h}^{b-1}\Delta_{r,r+1}^{\qq_i}\,.
\end{equation*}
The ranges of the outer index~$b$ in these expansions are disjoint within a fixed geometry. A given increment can occur only for~$r+1\leq b\leq r+h$, and hence at most~$h$ times. Also~$n_i\geq k_i+h$, so none of these terms reaches before the beginning of its fixed-geometry interval.

\smallskip

The initial startup contributes no separate advancing-interval term to this sum: its logarithmic increase is already included in~$\Phi_1$. Its increments which occur in the first synchronized loss are included in the multiplicity-$h$ bound, beginning at~$k_1=n_0$. The length-$2L$ obstruction intervals and the length-$h$ or length-$H$ intervals immediately after geometry changes are ordered and disjoint within their respective geometries; together they contribute at most one further copy. The full loss~$\Delta_{k_i,n_i+2L}^{\qq_i}$ in~\eqref{e.global.selection.partial.change} contributes at most one more copy, since it is used only when that geometry is left. Thus the total determinant contribution in the sum of~\eqref{e.global.selection.h.step}--\eqref{e.global.selection.failed.test} is at most
\begin{equation*}
	aC(h+2)\bigl(\log(24\Pi)+2N\log(1+\delta)\bigr)\,.
\end{equation*}

\smallskip

The decrease retained at each of the~$N$ geometry changes cancels the entire contribution~$2aC(h+2)N\log(1+\delta)$ from the determinant jumps. Since~$\Phi_{M+1}\geq0$, this proves~\eqref{e.global.selection.count}. Thus the determinant bound is not assumed to telescope unchanged through a geometry update: the comparison error is included and compensated by the decrease established in Step~3.

\smallskip

The initial error and the drift have sum at most one. Equations~\eqref{e.global.selection.scalar.span} and~\eqref{e.initial.geometry.bounds} give
\begin{equation*}
	\Phi_1\leq\log(1+Ch)+Q\Delta_{n_0,n_0+h}^{\Id}+aC(d)\log(2+4\Pi)\,.
\end{equation*}
The Euclidean comparison~\eqref{e.initial.normalization.bounds} bounds its determinant term by~$C(d,\gamma)\log(2+\Pi)$. It follows that~$\Phi_1\leq C(H,\sigma,d,\gamma)\log_3(2+\Pi)$. Substituting into~\eqref{e.global.selection.count} proves
\begin{equation*}
	M\leq C(H,\sigma,d,\gamma)\log_3(2+\Pi)\,.
\end{equation*}
The constant in this bound is independent of~$J$,~$j_*$, and~$B$: all applications of the adapted cube bound used only~$\|\RSZ\|_{L^Q}\leq2$. Choose~$C_1(H,\sigma,d,\gamma)$ larger than this constant and fix~$C$ in~\eqref{e.global.selection.lower.scale} as required in Step~1. If all~$J$ attempted steps continued, taking~$M=J$ would contradict the preceding estimate. Hence the test succeeds before~$J$ steps. This proves termination without increasing the region or changing~$\RSZ$ during the construction.

\smallskip

Each geometry change satisfies~$d_{\rm pr}([\m_i],[\m_{i+1}])\leq\varepsilon$ by~\eqref{e.projective.step}. Starting from~$\Id$, the triangle inequality therefore gives~\eqref{e.global.selection.eccentricity}. Each continuing step increases the current scale by at most~$2L+H+h$, and the last tested endpoint lies at most~$2L+H+h$ scales farther. Combining these observations with~\eqref{e.renormalization.entry} gives~\eqref{e.global.selection.scales}.

\smallskip

\emph{Step 5: Estimates at the selected scales.} We have proved the scale and grid eccentricity bounds. It remains to verify~\eqref{e.global.selection.calibration} and~\eqref{e.global.selection.profile}. By the stopping rule,~$F=\bfAhom_{n_i+2L,\qq_i}$ for the final continuing index~$i$,~$\m=\cmet(F)$, and~$t=s+H$. The comparison~\eqref{e.global.selection.geometry.comparison} and the successful test give~\eqref{e.global.selection.calibration}. The change-of-geometry output gives
\begin{equation*}
	\mathcal P_{\qq}(s;s)+D_{\qq,j_*}(s)\leq C(d,\gamma)\sigma^{\frac18(1-\gamma)}\,.
\end{equation*}
Its left side is also at most one by~\eqref{e.renormalization.output}, so~\eqref{e.fixed.geometry.fixed.span.propagation} applies from~$s$ to~$t$. Therefore
\begin{equation*}
	\mathcal P_{\qq}(t;s)+D_{\qq,j_*}(t)\leq C(H,d,\gamma)\left(\sigma^{\frac18(1-\gamma)}+e^{Qd\sigma}-1\right)\,.
\end{equation*}
Since~$0<\sigma\leq\varepsilon\leq1$, the second term in parentheses is at most~$C(d,\gamma)\sigma$, and hence at most~$C(d,\gamma)\sigma^{\frac18(1-\gamma)}$. Finally,~\eqref{e.fixed.geometry.profile.majorization} gives~$\mathcal P_{\qq}(t;t)=\history_{\qq}(t)\leq C(d,\gamma)\mathcal P_{\qq}(t;s)$. Combining the three estimates proves~\eqref{e.global.selection.profile}.
\end{proof}

\subsection{Duality gap and Euclidean transfer}
\label{s.response.transfer}
The selected geometry has small fluctuation and mean errors, but these measure randomness and scale variation, not the gap between the annealed coarse-grained matrix and its dual. To control that gap, we compare the primal and dual optimizers on the large cube with the corresponding optimizers on its smaller cubes, using a smooth cutoff to apply the equation without boundary terms. The small histories bound the cell averages of their gradients and fluxes after subtracting the annealed means, while the small change of the annealed matrix controls the energy cost of comparing optimizers at different scales. These bounds control the resulting errors and yield a small duality gap on the selected adapted cube. Two opposite Whitney comparisons then transfer this bound to a Euclidean cube.

\smallskip

Recall that the dual of a positive block~$F$ is~$F_*=\mathbf RF^{-1}\mathbf R$. We measure the gap by the smallest factor~$\mathfrak d(F)$ such that~$F\leq\mathfrak d(F)F_*$, namely
\begin{equation}
	\mathfrak d(F)\coloneqq\bigl|F_*^{-\nf12}FF_*^{-\nf12}\bigr|\,.
	\label{e.response.canonical.imbalance}
\end{equation}
For~$F=\bfAhom(\cu_m)$, making~$\mathfrak d(F)-1$ small is sufficient to make~$\Theta_m-1$ small, as the following calculation shows. If~$F_*\leq F$ and~$\kappa=\mathfrak d(F)$, then~$F\leq\kappa F_*$. Let~$S,S_*,K$ be the matrices in its block representation. Evaluate this inequality at~$\binom{x}{-K^tx}$ to obtain
\begin{equation*}
	S+(K+K^t)S_*^{-1}(K+K^t)\leq\kappa S_*\,.
\end{equation*}
Taking~$h=\nf{K-K^t}{2}$ and using~$S_*\leq S$ gives
\begin{equation*}
	1\leq
	\bigl|S_*^{-\nf12}\bigl(S+(K-h)^tS_*^{-1}(K-h)\bigr)S_*^{-\nf12}\bigr|
	\leq\kappa\,.
\end{equation*}
In particular, minimizing over~$h$ for~$F=\bfAhom(\cu_m)$ gives~$\Theta_m\leq\mathfrak d(F)$. We will use the weaker bound
\begin{equation}
	\Theta_m-1\leq3\bigl(\mathfrak d(\bfAhom(\cu_m))-1\bigr)\,.
	\label{e.response.euclidean.contrast}
\end{equation}

\smallskip

Fix~$\varepsilon(d,\gamma)$ as in Proposition~\ref{p.global.selection}.

\begin{proposition}[Duality gap on Euclidean cubes]
\label{p.response.transfer}
For every~$\delta\in(0,1]$, there exist an integer~$H=H(d,\gamma,\delta)\geq\max\{4,h\}$ and~$\sigma_0(d,\gamma,\delta)\in(0,\varepsilon]$ with the following property. For every~$\sigma\in(0,\sigma_0]$, there are constants~$B_{\rm resp}(d,\gamma,\delta,\sigma)\geq1$ and~$C(d,\gamma,\delta,\sigma)<\infty$ such that any output of Proposition~\ref{p.global.selection} with these parameters and~$B\geq\max\{B_0,B_{\rm resp}\}$ admits a deterministic Euclidean scale~$m_{\rm ent}>t$ satisfying
\begin{equation}
	m_{\rm ent}-t\leq\bigl\lceil C(d,\gamma,\delta,\sigma)\log_3(2+\Pi)\bigr\rceil\qquad\text{and}\qquad\Theta_{m_{\rm ent}}-1\leq\delta\,.
	\label{e.response.transfer.conclusion}
\end{equation}
All choices are independent of~$\bfE$, the coefficient law, and~$\Pi$.
\end{proposition}

\begin{proof}
Steps~1--2 fix the tolerances and collect the selected inputs. In Steps~3--4 we express the duality gap in terms of the centered energies of the primal and dual optimizers and bound their energies and mean gradients and fluxes. Step~5 bounds the averages of the centered gradients and fluxes on smaller cubes, and Step~6 uses these bounds and the equation to estimate the centered energies. The loads are chosen from the terminal block, while the grid remains fixed by the selected matrix~$F$. Once the gap is small on the adapted cube, Step~7 uses two opposite Whitney comparisons to transfer the bound to a Euclidean cube.

\smallskip

\emph{Step 1: Tolerances.} Fix~$\delta\in(0,1]$. Choose~$\delta_{\rm ad}\in(0,1]$ and~$\eta_+,\eta_-,\eta_{\rm iso}\in(0,1)$, depending only on~$d$ and~$\delta$, so that
\begin{equation}
	\eta_+\leq\eta_{\rm iso}\,,\qquad(1+\delta_{\rm ad})^{-d}-\eta_-\geq1-\eta_{\rm iso}\,,\qquad3\biggl(\frac{(1+\eta_{\rm iso})^3}{1-\eta_{\rm iso}}(1+\delta_{\rm ad})-1\biggr)\leq\delta\,.
	\label{e.response.transfer.tolerances}
\end{equation}
First choose~$\eta_{\rm iso}$ small, then~$\delta_{\rm ad}$ small, and finally~$\eta_+$ and~$\eta_-$. We will choose the scale separation~$H\geq\max\{4,h\}$ and~$\sigma_0$ at the end of Step~6. Until then, assume~$0<\sigma\leq\min\{\varepsilon,\frac12\}$.

\smallskip

\emph{Step 2: The selected inputs.} Put~$U_u=\cus_u^\qq$,~$E_u=\bfAhom_{u,\qq}$, and~$\kappa_u=\mathfrak d(E_u)$. Write~$\xi=\varepsilon^{\nf12}\sigma$ for the calibration error. The selection estimates give~$(1-\xi)F\leq E_s\leq(1+\xi)F$, while subadditivity, stationarity, and~\eqref{e.global.selection.calibration} give
\begin{equation*}
	E_t\leq E_s\leq rE_t\,,\qquad
	r=\frac{\det E_s}{\det E_t}<e^{d\sigma}\,,\qquad t=s+H\,.
\end{equation*}
The middle comparison follows by bounding each eigenvalue of~$E_t^{-\nf12}E_sE_t^{-\nf12}$ by their product. The order~$E_{*,t}\leq E_t$ and~$E_{*,s}\leq E_s$ holds by~\eqref{e.matrix.annealed.sharp.order}. In particular,
\begin{equation}
	1\leq\kappa_t\leq\kappa_s\leq r^2\kappa_t\,.
	\label{e.response.imbalance.comparison}
\end{equation}
Indeed, apply the orders above also to the inverses and compare each block with its dual.

\smallskip

Let~$\eta\in(0,\frac12)$ be a small number to be chosen after~$H$. For fixed~$H$, decreasing~$\sigma_0$ makes every error term and drift term in~\eqref{e.global.selection.profile} at most~$\eta$. In particular, the fluctuation and mean histories at~$s$ and~$t$ are each at most~$\eta$. Contributions below~$j_*$ are controlled by the same~$\RSZ$ as in the selection argument. The adapted cube bound and~\eqref{e.two.grid.source.normalization} give
\begin{equation}
	C(d,\gamma)\Pi\mathfrak e(\m)^2 3^{-\rho_{\max}(t-j_*)}\leq\eta^{\frac{1}{Q}}
	\qand
	C(d,\gamma)\Pi\mathfrak e(\m)^2 3^{-\frac32(s-j_*)}\leq1
	\label{e.response.source.smallness}
\end{equation}
once~$B_{\rm resp}$ is large enough. To justify this choice,~\eqref{e.global.selection.scales} and~\eqref{e.global.selection.eccentricity} bound the left sides by constants times~$(2+\Pi)^{1+2C(H,\sigma,d,\gamma)-\rho_{\max}B}$ and~$(2+\Pi)^{1+2C(H,\sigma,d,\gamma)-\nf{3B}{2}}$. The exponent contributed by the grid eccentricity is independent of~$B$. We choose~$B_{\rm resp}$ after~$H$ and~$\sigma$, using~$2+\Pi\geq3$. Taking expectations or~$L^Q$ norms costs only~$\|\RSZ\|_{L^Q}\leq2$.

\smallskip

\emph{Step 3: Choice of loads.} In Steps~3--6 we prove
\begin{equation}
	\mathfrak d(\bfAhom_{t,\qq})\leq1+\delta_{\rm ad}\,.
	\label{e.response.adapted.conclusion}
\end{equation}
All constants denoted by~$C$ in Steps~3--6 depend only on~$d$ and~$\gamma$;~$C_H$ may also depend on~$H$.

\smallskip

Let~$S,S_*,K$ be the matrices in the block representation of~$E_t$, and put
\begin{equation*}
	h_t=\tfrac12(K-K^t)\,,\qquad b_t=S+\tfrac14(K+K^t)S_*^{-1}(K+K^t)\,,\qquad
	\cmet_t=b_t\#S_*\,.
\end{equation*}
For each unit vector~$e$, set~$p=\cmet_t^{-\nf12}e$ and~$q=\cmet_t^{\nf12}e$. The reciprocal normalizations give~$p\cdot q=1$ and cancel inside the trace when we sum over an orthonormal basis. The metric~$\cmet_t$ used to choose these loads need not equal the canonical metric~$\cmet(F)$ used to choose the grid.

\smallskip

We first express the mismatch in terms of optimizer energies. For any coefficient field~$b$ and deterministic loads~$(p,q')$, let~$v=v(\cdot,U_t,p,q';b)$ maximize
\begin{equation*}
	J(U_t,p,q';b)=\sup_{v\in\mathcal A(U_t;b)}
	\fint_{U_t}\bigl(-\tfrac12\nabla v\cdot b\nabla v-p\cdot b\nabla v+q'\cdot\nabla v\bigr)\,,
\end{equation*}
where~$\mathcal A(U_t;b)$ is the space of finite-energy~$b$-harmonic functions, as in~\cite[(2.9)]{AK.HC}. Define the centered energy by
\begin{equation*}
	\widetilde J(U_t,p,q';b)
	=\E[J(U_t,p,q';b)]
	-\frac12\E[(\nabla v)_{U_t}]\cdot\E[(b\nabla v)_{U_t}]\,.
\end{equation*}
Denote the centered energies for~$(\a;p,q-h_tp)$ and~$(\a^t;p,q+h_tp)$ by~$\widetilde J^-(e)$ and~$\widetilde J^+(e)$, respectively. The two optimizers, and hence their two pairs of means, are computed separately.

\smallskip

The block formulas~\cite[(2.15), (2.32)]{AK.HC} give, for the primal load~$(p,q')$,
\begin{equation*}
	\E[J]=\tfrac12p\cdot Sp+\tfrac12(q'+Kp)\cdot S_*^{-1}(q'+Kp)-p\cdot q'\,,
\end{equation*}
and the two means are~$-p+S_*^{-1}(q'+Kp)$ and~$q'-K^tS_*^{-1}(q'+Kp)-Sp$. For the adjoint, replace~$K$ by~$-K$. Subtracting the products of these means and adding the resulting centered energies yields, with~$r_t=\nf{K+K^t}{2}$,
\begin{equation*}
	\widetilde J^-(e)+\widetilde J^+(e)
	=p\cdot(SS_*^{-1}-\Id)q+2r_tp\cdot S_*^{-1}r_tS_*^{-1}q\,.
\end{equation*}
Summing over a Euclidean orthonormal basis cancels the factors~$\cmet_t^{\pm\nf12}$ inside the trace and gives
\begin{equation*}
	\sum_{i=1}^d\bigl(\widetilde J^-(e_i)+\widetilde J^+(e_i)\bigr)
	=\tr(S_*^{-1}S-\Id)+2\tr(S_*^{-1}r_tS_*^{-1}r_t)\,.
\end{equation*}
Both trace terms are nonnegative. Since~$b_t=S+r_tS_*^{-1}r_t$ and~$S\geq S_*$, this bounds~$\theta-1$, where~$\theta=|S_*^{-\nf12}b_tS_*^{-\nf12}|$. Applying~\eqref{e.matrix.block.comparison} with~$h=h_t$ therefore gives
\begin{equation}
	\kappa_t-1\leq12d\sup_{|e|=1}\bigl(|\widetilde J^-(e)|+|\widetilde J^+(e)|\bigr)\,.
	\label{e.response.by.centered.energies}
\end{equation}
The same calculation used for~\eqref{e.response.euclidean.contrast} gives~$1\leq\theta\leq\kappa_t$.

\smallskip

\emph{Step 4: Energy and mean bounds on the selected grid.} The canonical mean~$M(F)$ is self-dual: the inverse and congruence identities for the geometric mean give~$\mathbf RM(F)^{-1}\mathbf R=M(F)$. The matrices in its block representation are therefore~$(\m,\m,g)$ with~$g^t=-g$. Define
\begin{equation*}
	G=\begin{pmatrix}\Id&0\\g&\Id\end{pmatrix}\,,\qquad
	M_0=\begin{pmatrix}\m&0\\0&\m^{-1}\end{pmatrix}\,,
	\qquad G^tM(F)G=M_0\,.
\end{equation*}
Recenter the primal coefficient by~$g$: set~$a_-=\a-g$ and~$a_+=\a^t+g$, and use the loads
\begin{equation*}
	q^-=q+(g-h_t)p\qand q^+=q+(h_t-g)p\,.
\end{equation*}
Constant-skew covariance preserves the centered energies, since it changes a flux mean~$Q$ by a skew multiple of its gradient mean~$P$ and~$P\cdot gP=0$. The block for~$a_-$ is~$\widehat E_u=G^tE_uG$; the block for~$a_+$ is~$\widehat E_u^+=D\widehat E_uD$, where~$D=\operatorname{diag}(\Id,-\Id)$. Write~$\widehat E_u^-=\widehat E_u$.

\smallskip

For a matrix~$E$ with~$E_*\leq E$, the Riccati equation~$M(E)E^{-1}M(E)=E_*$ shows, after conjugation by~$M(E)^{-\nf12}$, that
\begin{equation*}
	M(E)\leq E\leq\mathfrak d(E)^{\nf12}M(E)\,.
\end{equation*}
Indeed, in these coordinates the two blocks are~$X$ and~$X^{-1}$, and~$|X|^2=\mathfrak d(E)$. Calibration at~$s$, the comparison~$E_t\leq E_s\leq rE_t$, and monotonicity of the geometric mean now give
\begin{equation}
	C^{-1}M_0\leq\widehat E_t^\pm\leq C\kappa_t^{\nf12}M_0
	\qand \widehat E_s^\pm\leq C\kappa_s^{\nf12}M_0\,.
	\label{e.response.calibrated.blocks}
\end{equation}
For example, the comparison of~$G^tM(E_t)G$ with~$M_0$ costs at most~$\sqrt{\nf{r(1+\xi)}{1-\xi}}$, bounded in terms of~$d$ because~$\xi\leq\frac12$ and~$r<e^{\nf{d}{2}}$. The adjoint bounds follow since~$D$ commutes with~$M_0$.

\smallskip

Put~$x^\pm=\binom{-p}{q^\pm}$ and~$(L^\pm)^2=x^\pm\cdot\widehat E_t^\pm x^\pm$. The opposite signs in the loads give
\begin{equation*}
	(L^-)^2+(L^+)^2
	=2p\cdot b_tp+2q\cdot S_*^{-1}q
	=4e\cdot \cmet_t^{\nf12}S_*^{-1}\cmet_t^{\nf12}e
	\leq4\sqrt\theta\,.
\end{equation*}
Here~$\cmet_tS_*^{-1}\cmet_t=b_t$ and~$q=\cmet_tp$. Let~$v_t^\pm$ be the corresponding optimizers and denote~$J_t^\pm \coloneqq J(U_t,p,q^\pm;a_\pm)$. Since~$p\cdot q^\pm=1$, we have
\begin{equation}
	\begin{aligned}
	&
	0\leq\E[J_t^\pm]=\tfrac12(L^\pm)^2-1\leq C\kappa_s^{\nf12}\,,\\
&	0\leq\tau^\pm \coloneqq\E[J(U_s,p,q^\pm;a_\pm)-J_t^\pm]
	\leq C(r-1)\kappa_s^{\nf12}\,.
	\end{aligned}
	\label{e.response.energy.and.defect}
\end{equation}
The defect bound follows by applying~$0\leq E_s-E_t\leq(r-1)E_t$ to the same loads.

\smallskip

Let~$X_t^\pm=\binom{\nabla v_t^\pm}{a_\pm\nabla v_t^\pm}$ and define the separate annealed means~$Y^\pm=\E[(X_t^\pm)_{U_t}]=\binom{P^\pm}{Q^\pm}$. The mean identity~\cite[(2.32)]{AK.HC} reads~$Y^\pm=(\Itwod+\mathbf R\widehat E_t^\pm)x^\pm$. Since~$\mathbf RM_0\mathbf R=M_0^{-1}$,~\eqref{e.response.calibrated.blocks} gives
\begin{equation}
	|M_0^{\nf12}x^\pm|^2\leq C\kappa_s^{\nf12}
	\qand |M_0^{\nf12}Y^\pm|^2\leq C\kappa_s\,.
	\label{e.response.load.and.mean}
\end{equation}
For the second bound, use~$\widehat E_t^\pm M_0^{-1}\widehat E_t^\pm\leq C\kappa_t^{\nf12}\widehat E_t^\pm$ and the preceding bound on~$(L^\pm)^2$.

\smallskip

\emph{Step 5: The weak-norm estimate.} Put~$\alpha=\nf{1-\gamma}{4}$ and~$\rho=\nf{1+\gamma}{2}$. Then~$\rho_{\max}<\rho<1$ and~$\frac12-\frac\rho2=\alpha$. For a doubled field~$Z$ on~$U_t$, use the concrete scale-average seminorm
\begin{equation*}
	[Z]_{\Besov{-\frac{1}{2}}{2}{1}(U_t)}
	=\sum_{k=-\infty}^t3^{\frac{k}{2}}
	\biggl(\avsum_{z\in3^k\qq\Zd\cap U_t}|(Z)_{z+U_k}|^2\biggr)^{\nf12}\,,
\end{equation*}
and set
\begin{equation*}
	W^\pm=3^{-t}\E\Bigl[
	[M_0^{\nf12}(X_t^\pm-Y^\pm)]_{\Besov{-\frac{1}{2}}{2}{1}(U_t)}^2\Bigr]\,.
\end{equation*}
This is the seminorm in~\cite[(2.130)]{AK.HC}, on the selected grid, not the compact-test dual norm used in the theorem statements. We prove
\begin{equation}
	(W^\pm)^{\nf12}
	\leq C\bigl(3^{-\alpha H}+C_H\eta^{\frac{1}{2Q}}\bigr)\kappa_s^{\nf12}\,.
	\label{e.response.weak.estimate}
\end{equation}
The first term comes from scales preceding the selected interval. The second accounts for the recent cell fluctuations, the mean differences, and the correction from random to annealed centering. The histories retained throughout the iteration thus control the cell averages needed to compare the optimizer energies: small fluctuations at the terminal scale alone would not bound this seminorm.

\smallskip

Write~$A_k(z)=\bfA(z+U_k;\a)$. For a symmetric matrix, the subscript~$+$ denotes the spectral positive part. The all-scale maximum
\begin{equation*}
	\mathcal M=\sup_{k\leq t}3^{-\rho(t-k)}
	\max_{z\in3^k\qq\Zd\cap U_t}
	\bigl|\bigl(E_t^{-\nf12}A_k(z)E_t^{-\nf12}-\Itwod\bigr)_+\bigr|
\end{equation*}
satisfies~$\E[\mathcal M^Q]\leq C\eta$. For~$k\geq j_*$, decompose the normalized block into~$V_{k,t}^\qq(z)$ and~$P_{k,t}^\qq$. The fluctuation part is bounded by the maximum defining~$\history_\qq^{\rm fluc}(t)$, because~$\rho>\rho_{\max}$. The mean part telescopes into the nonnegative increments in~$D_{\qq,j_*}(t)$; its weight is bounded by the drift weight because~$\rho>\nf{1-\gamma}{8}$. For~$k<j_*$, the adapted cube bound controls the maximum by~$C\Pi\mathfrak e(\m)^2\RSZ3^{-\rho_{\max}(t-j_*)}$. Now apply~\eqref{e.response.source.smallness} and~$\|\RSZ\|_{L^Q}\leq2$. Fixed congruence leaves these normalized operator norms unchanged, so the same maximum controls both recentered coefficient fields.

\smallskip

We use the cell-average argument of~\cite[Lemma~2.16]{AK.HC} with exponent~$\frac12$, decay exponent~$\rho$, cutoff~$\delta=1$, and window~$H$. We first explain why this argument applies on the selected grid. Suppress the sign, and let
\begin{equation*}
	K_0=|M_0^{-\nf12}\widehat E_t^\pm M_0^{-\nf12}|^{\nf12}\leq C\kappa_s^{\frac{1}{4}}\,.
\end{equation*}
Write~$\widehat A(x)$ for the pointwise doubled block of~$a_\pm$ and~$\widehat A(V)$ for its coarse block on a cell~$V$. The variational bound on the mean of a gradient--solenoidal field~$X$ gives
\begin{equation*}
	|M_0^{\nf12}(X)_V|^2
	\leq K_0^2
	|(\widehat E_t^\pm)^{-\nf12}\widehat A(V)(\widehat E_t^\pm)^{-\nf12}|
	\fint_V X\cdot\widehat A X\,.
\end{equation*}
Indeed, the mean energy is bounded below by~$(X)_V\cdot\widehat A_*(V)(X)_V$; inversion and~$\mathbf RM_0\mathbf R=M_0^{-1}$ give the displayed inequality. This argument holds on every Lipschitz cell. On each recent cell, split~$(X_t)_V-(X_t)_{U_t}$ through the mean of that cell's optimizer. One difference compares the means of the cell and terminal optimizers, each taken on its own domain; the optimizer-mean identity bounds it by the corresponding coarse-matrix difference. The other is the average over the small cell of the difference between the two optimizer fields. Apply the displayed inequality to this field difference; the quadratic energy comparison bounds its total energy by four times the coarse-energy drop. Exact partition averaging now gives the matrix-difference sum and the square-root energy-drop sum in~\cite[Lemma~2.16]{AK.HC}. Summing the same cell bound over the older scales gives its remaining energy term. No metric--grid compatibility or stationarity is used in these pathwise calculations.

\smallskip

The recent cell differences, normalized by~$E_t$, split as~$V_{k,t}^\qq(z)-V_t^\qq+(P_{k,t}^\qq-\Itwod)$. For~$s\leq k\leq t$, their~$L^2$ cell-average norm is at most~$C_H\eta^{\nf{1}{Q}}$. Indeed, the fluctuation history bounds the first two terms, while the mean history gives~$\meanpenalty_Q(P_{k,t}^\qq)\leq3^{\alpha(t-1-k)}\eta$ for~$k<t$. For~$P\geq\Itwod$, the defining formula for~$\meanpenalty_Q$ gives~$|P-\Itwod|\leq\tr(P-\Itwod)\leq\meanpenalty_Q(P)^{\nf{1}{Q}}$. This bounds the matrix-difference sum. For the square-root term, the averaged cell difference is nonnegative by subadditivity, and its expected trace is~$\tr(P_{k,t}^\qq-\Itwod)$. Consequently, the~$L^2$ norm of its square root is at most~$C_H\eta^{\nf{1}{2Q}}$. Together the two finite sums contribute at most~$C_HK_0L^\pm\eta^{\nf{1}{2Q}}$.

\smallskip

The remaining term contains the optimizer energy~$\mathcal E_t^\pm=(2J_t^\pm)^{\nf12}$. Since~$p\cdot q^\pm=1$, the block energy formula gives~$(\mathcal E_t^\pm)^2\leq(1+\mathcal M)(L^\pm)^2$. On~$\{\mathcal M>1\}$, it follows that~$\mathcal M^{\nf12}\mathcal E_t^\pm\leq\sqrt2\,\mathcal M L^\pm$. Since~$Q\geq2$,
\begin{equation*}
	\E\bigl[\indc_{\{\mathcal M>1\}}\mathcal M^2\bigr]\leq\E[\mathcal M^Q]\leq C\eta\,.
\end{equation*}
On the complementary event, the energy is at most~$\sqrt2L^\pm$, and the discarded scales contribute~$3^{-\alpha H}$. These two terms cost at most~$CK_0L^\pm(\eta^{\nf12}+3^{-\alpha H})$.

\smallskip

The argument of~\cite{AK.HC} centers by the random cell average. The optimizer-average identity gives
\begin{equation*}
	\bigl\|M_0^{\nf12}\bigl((X_t^\pm)_{U_t}-Y^\pm\bigr)\bigr\|_{L^2}
	\leq K_0L^\pm\eta^{\frac{1}{Q}}\,.
\end{equation*}
A constant vector~$c$ has~$3^{-\nf{t}{2}}[c]_{\Besov{-\nf{1}{2}}{2}{1}(U_t)}=\nf{|c|}{1-3^{-\nf12}}$, so this correction costs the same order. Finally,~$K_0L^\pm\leq C\kappa_s^{\nf12}$ by~\eqref{e.response.calibrated.blocks} and the load bounds. This proves~\eqref{e.response.weak.estimate}.

\smallskip

\emph{Step 6: Comparing optimizer energies.} For each sign, write the two diagonal blocks of~$\E[\bfA(z+U_k;a_\pm)]$ as~$b_{k,z}^\pm$ and~$(S_{*,k,z}^\pm)^{-1}$. In particular,~$S_{*,k,z}^\pm$ is the inverse of the expected lower-right block, not the expectation of its inverse. Define
\begin{equation*}
	\mathcal L_s^\pm=\sum_{k=-\infty}^s3^{\frac32(k-s)}
	\avsum_{z\in3^k\qq\Zd\cap U_s}
	\bigl(|(b_{k,z}^\pm)^{\nf12}P^\pm|
	+|(S_{*,k,z}^\pm)^{-\nf12}Q^\pm|\bigr)^2\,.
\end{equation*}
Then~$\mathcal L_s^\pm\leq C\kappa_s^{\nf{3}{2}}$. For~$k\geq j_*$, stationarity removes~$z$, and the weighted sum of~$E_k$ is bounded by~$C(1+D_{\qq,j_*}(s))E_s$: telescope the mean increments and use~$\frac32>\nf{1-\gamma}{8}$. For~$k<j_*$, the adapted cube bound and the second inequality in~\eqref{e.response.source.smallness} bound the remaining sum by~$CE_s$. Apply these matrix bounds to the two coordinate projections of~$Y^\pm$, and then use~\eqref{e.response.calibrated.blocks} and~\eqref{e.response.load.and.mean}. The series converges because~$\gamma<\frac32$.

\smallskip

Choose a smooth cutoff~$\varphi$ in~$U_t$ with~$0\leq\varphi\leq2$,~$(\varphi)_{U_t}=1$, and derivatives at the scale~$3^t$ in~$\qq$-coordinates. The variational calculation in~\cite[(3.45)--(3.54)]{AK.HC}, before separating the products by Young's inequality, gives
\begin{equation}
	\big|\widetilde J^\pm(e)\big|
	\leq C\Bigl(\tau^\pm+(\tau^\pm\E[J_t^\pm])^{\nf12}
	+(\tau^\pm\mathcal L_s^\pm)^{\nf12}
	+3^{-H}\bigl(\E[J_t^\pm]+(\E[J_t^\pm]\mathcal L_s^\pm)^{\nf12}\bigr)
	+W^\pm\Bigr)\,.
	\label{e.response.cutoff.estimate}
\end{equation}
We compare the terminal optimizer with its scale-$s$ cell optimizers in the integrals weighted by~$\varphi$. The quadratic energy identity measures the mean energy of their difference by~$2\tau^\pm$. Insert~$P^\pm,Q^\pm$ in the cutoff decomposition~\cite[(3.46)]{AK.HC}; the term containing the difference between these means and the actual annealed means is zero. Integration by parts, the cutoff product estimate of~\cite[Lemma~A.3]{AK.HC} at order~$\nf12$, and the direct full-dual pairing in the proof of~\cite[Lemma~A.1, (A.4)]{AK.HC} give~$\E[|\fint_{U_t}\varphi(\nabla v_t^\pm-P^\pm)\cdot(a_\pm\nabla v_t^\pm-Q^\pm)|]\leq CW^\pm$. The comparison of the rounded and unrounded metrics costs only a dimensional constant: if~$c=|\m^{-1}|^{\nf12}$, the rounding error gives~$|\m^{-\nf12}\qq-c\Id|\leq \nf{c}{2}$, hence~$|\m^{-\nf12}\qq|\,|\qq^{-1}\m^{\nf12}|\leq3$. More explicitly, pullback by~$x=\qq y$ preserves normalized cell averages and their Besov sums. The two gradient factors in the Euclidean pairing acquire the norms~$|\qq\m^{-\nf12}|$ and~$|\m^{\nf12}\qq^{-1}|$, whose product is at most three; thus the constant is independent of the grid eccentricity.

\smallskip

Replacing the terminal optimizer by its cell optimizers bounds the cutoff energy error by~$C(\tau^\pm+(\tau^\pm\E[J_t^\pm])^{\nf12}+3^{-H}\E[J_t^\pm])$. For the two cutoff-mean terms, subtract the cell average of~$\varphi$. The constant cell means cancel after expectation, because the scale-$s$ translations are integral and~$(\varphi)_{U_t}=1$. The optimizer differences cost~$C(\tau^\pm\mathcal L_s^\pm)^{\nf12}$. Pairing the within-cell oscillations with the gradient and flux, then summing the descendants with weights~$3^{\nf{3(k-s)}{2}}$, gives~$C3^{-H}(\E[J_t^\pm]\mathcal L_s^\pm)^{\nf12}$. These are precisely the terms in~\eqref{e.response.cutoff.estimate}. Keeping this last product is essential: it is of order~$3^{-H}\kappa_s$, whereas replacing it by~$3^{-H}\mathcal L_s^\pm$ would lose a factor~$\kappa_s^{\nf12}$.

\smallskip

Substituting~\eqref{e.response.energy.and.defect},~\eqref{e.response.weak.estimate}, and~$\mathcal L_s^\pm\leq C\kappa_s^{\nf{3}{2}}$ in~\eqref{e.response.cutoff.estimate} gives
\begin{equation*}
	|\widetilde J^\pm(e)|
	\leq C\Bigl((r-1)+\sqrt{r-1}+3^{-H}
	+\bigl(3^{-\alpha H}+C_H\eta^{\frac{1}{2Q}}\bigr)^2\Bigr)\kappa_s\,.
\end{equation*}
The balanced geometry and reciprocal loads keep this bound uniform in~$e$ and linear in~$\kappa_s$, with a coefficient independent of the incoming mismatch. Hence the tolerances need not depend on its size. Combining the bound with~\eqref{e.response.by.centered.energies} and~\eqref{e.response.imbalance.comparison}, we obtain~$\kappa_t-1\leq\omega\kappa_t$, where
\begin{equation*}
	\omega\leq Cr^2\Bigl((r-1)+\sqrt{r-1}+3^{-H}
	+\bigl(3^{-\alpha H}+C_H\eta^{\frac{1}{2Q}}\bigr)^2\Bigr)\,.
\end{equation*}
Choose~$H=H(d,\gamma,\delta)$ large, then~$\eta=\eta(d,\gamma,\delta,H)$ small, and finally~$\sigma_0$ small enough that~$r<e^{d\sigma_0}$ and all the selected error and drift terms are at most~$\eta$. These choices make~$\omega\leq\nf{\delta_{\rm ad}}{1+\delta_{\rm ad}}$. The buffer~$B_{\rm resp}$ is then chosen as in Step~2. Rearranging~$\kappa_t-1\leq\omega\kappa_t$ proves~\eqref{e.response.adapted.conclusion}.

\smallskip

\emph{Step 7: Euclidean transfer.} We seek a scale~$m_{\rm ent}>t$ within the range in~\eqref{e.response.transfer.conclusion} such that
\begin{equation}
	(1-\eta_{\rm iso})\bfAhom_{t,\qq}\leq\bfAhom_{m_{\rm ent},\Id}\leq(1+\eta_{\rm iso})\bfAhom_{t,\qq}\,.
	\label{e.response.transfer.block.comparison}
\end{equation}
This last change of geometry need not preserve the error: the duality gap is already small on the adapted cube, and only an annealed matrix comparison remains. We can therefore make the direct transfer to Euclidean cubes, paying a logarithmic scale separation once. Small mismatch also leaves little determinant variation on the adapted grid, which provides the lower bound needed in the reverse comparison.

\smallskip

First, the adapted estimate persists on the fixed grid. For~$u\geq t$, partition~$U_u$ into scale-$t$ cells. Their translations are integral, so subadditivity and stationarity give~$E_u\leq E_t$. The order~$E_{*,u}\leq E_u$ gives~$\det E_u\geq1$. Taking determinants in~$E_t\leq\kappa_t E_{*,t}$ yields~$\det E_t\leq\kappa_t^d\leq(1+\delta_{\rm ad})^d$. Every eigenvalue of~$E_u^{-\nf12}E_tE_u^{-\nf12}$ is at least one, so its largest eigenvalue is at most their product. Therefore
\begin{equation*}
	E_u\leq E_t\leq(1+\delta_{\rm ad})^dE_u\qquad(u\geq t)\,.
\end{equation*}

\smallskip

We next compare the grids, following the two Whitney decompositions in~\cite[Lemma~2.15]{AK.HC}. The rounding bounds give~$|\qq|\leq2\mathfrak e(\m)$ and~$|\qq^{-1}|\leq2$. The same face-strip argument as in Lemma~\ref{l.two.grid.whitney} shows that the relative volume of the scale-$i$ boundary cells is at most~$C(d)\mathfrak e(\m)3^{i-n}$ when an adapted scale-$n$ cube is filled with Euclidean cubes, and at most~$C(d)\mathfrak e(\m)3^{i-m}$ for the reverse filling of a Euclidean scale-$m$ cube. Moreover,~$3^n\qq\Zd\subseteq\Zd$ for every~$n\geq j_*$. Thus the proof of the two comparisons of~\cite{AK.HC} applies at every later aligned scale, with constants independent of~$j_*$.

\smallskip

The comparison~\eqref{e.two.grid.source.normalization} gives~$\bfE\leq C(d,\gamma)\Pi\mathfrak e(\m)E_t$. Put~$\overline K=\max\{2,K_{\Psi_{\S}}\}\leq2K_{\Psi_{\S}}$. This remains an admissible growth constant in~\eqref{e.source.tail}. Enlarge~$C_{\rm src}$ in~\eqref{e.source.lower.scale} so that~$\overline K^2 3^{-j_*}\leq1$. The factors~$(1+\overline K^2 3^{-n})^\gamma$ coming from the tail of~$\S$ in the comparisons of~\cite{AK.HC} are then at most two for~$n\geq j_*$. For integers~$\ell,r\geq1$, set~$m_{\rm ent}=t+\ell$. The gap~$\ell$ will give the upper bound for the Euclidean block; the further gap~$r$, together with persistence on the adapted grid, will give the lower bound. Apply~\cite[(2.128)]{AK.HC} with~$k=t-1<n=t<m=m_{\rm ent}$, and~\cite[(2.127)]{AK.HC} with~$k=m_{\rm ent}<n=m_{\rm ent}+r<m=m_{\rm ent}+r+1$. After normalization, these give
\begin{align*}
	\bfAhom_{m_{\rm ent},\Id}
	&\leq E_t+C(d,\gamma)\Pi\mathfrak e(\m)^2 3^{-\ell}E_t\,,\\
	E_{m_{\rm ent}+r}
	&\leq\bfAhom_{m_{\rm ent},\Id}
	+C(d,\gamma)\Pi\mathfrak e(\m)^2 3^{-r}E_t\,.
\end{align*}
These are unconditional annealed comparisons, using the tail of~$\S$ and stationarity. They do not require enlarging the region on which~$\RSZ$ was constructed. Choose~$\ell$ so that the first relative error is at most~$\eta_+$, and then~$r$ so that the second is at most~$\eta_-$. The grid eccentricity bound~\eqref{e.global.selection.eccentricity} gives
\begin{equation}
	\ell+r\leq C(d,\gamma,\delta,\sigma)\log_3(2+\Pi)\,.
	\label{e.response.transfer.gaps}
\end{equation}

\smallskip

Combining the reverse comparison with persistence at scale~$m_{\rm ent}+r$, and using the forward comparison, gives
\begin{equation*}
	\bigl((1+\delta_{\rm ad})^{-d}-\eta_-\bigr)E_t
	\leq\bfAhom_{m_{\rm ent},\Id}\leq(1+\eta_+)E_t\,.
\end{equation*}
The choices in~\eqref{e.response.transfer.tolerances} prove~\eqref{e.response.transfer.block.comparison}. To transfer the imbalance, put~$F_{\rm ent}=\bfAhom_{m_{\rm ent},\Id}$. Inverting~$F_{\rm ent}\leq(1+\eta_{\rm iso})E_t$ gives~$E_{*,t}\leq(1+\eta_{\rm iso})F_{*,\rm ent}$, hence
\begin{equation*}
	F_{\rm ent}\leq(1+\eta_{\rm iso})E_t
	\leq(1+\eta_{\rm iso})\kappa_t E_{*,t}
	\leq(1+\eta_{\rm iso})^2\kappa_t F_{*,\rm ent}\,.
\end{equation*}
In particular,~\eqref{e.response.adapted.conclusion} implies
\begin{equation*}
	\mathfrak d(F_{\rm ent})\leq
	\frac{(1+\eta_{\rm iso})^3}{1-\eta_{\rm iso}}(1+\delta_{\rm ad})\,.
\end{equation*}
Now~\eqref{e.response.euclidean.contrast} and~\eqref{e.response.transfer.tolerances} imply~$\Theta_{m_{\rm ent}}-1\leq\delta$. Since~$m_{\rm ent}-t=\ell\leq\ell+r$, estimate~\eqref{e.response.transfer.gaps} proves the asserted scale bound.
\end{proof}

\subsection{Proof of polynomial entry}

\begin{proof}[Proof of Theorem~\ref{t.polynomial.entry}]
Fix~$\sigma\in(0,1]$. Apply Proposition~\ref{p.response.transfer} with~$\delta=\sigma$, fixing~$H$ and then a determinant tolerance~$\tau\in(0,\sigma_0]$. Choose~$L$,~$\eta$, and~$B_0$ as in Proposition~\ref{p.scale.selection} at determinant tolerance~$\tau$, and fix~$B\geq\max\{B_0,B_{\rm resp}\}$. Finally, choose~$j_*$ to be the smallest integer satisfying~\eqref{e.global.selection.lower.scale}. All constants have been fixed in terms of~$\sigma$,~$d$, and~$\gamma$, so
\begin{equation*}
	j_*\leq C(\sigma,d,\gamma)\log_3(2+\Pi K_{\Psi_{\S}})\,.
\end{equation*}
Proposition~\ref{p.global.selection} applies with this~$j_*$ and determinant tolerance~$\tau$, giving
\begin{equation*}
	t\leq j_*+\bigl\lceil(B+C(H,\tau,d,\gamma))\log_3(2+\Pi)\bigr\rceil\,.
\end{equation*}
Proposition~\ref{p.response.transfer} supplies~$m_{\rm ent}$ with~$\Theta_{m_{\rm ent}}\leq1+\sigma$ and at most~$C(\sigma,d,\gamma)\log_3(2+\Pi)$ further scales. Adding these bounds proves~\eqref{e.polynomial.entry} after enlarging~$C(\sigma,d,\gamma)$. 
\end{proof}

\section{Quantitative homogenization}
\label{s.convergence.homogenization}

In this section we prove algebraic convergence of the coarse-grained matrices and derive the homogenization and regularity results stated in the introduction. The previous sections make the gap between the annealed coarse-grained matrix and its dual small at one scale. Smallness allows us to improve this gap repeatedly: averaging reduces fluctuations, and the comparison of optimizer energies turns this control into a further decrease of the gap. The nonlinear errors are quadratic in the remaining mismatch and can therefore be absorbed once that mismatch is small. The resulting geometric decay in the scale index gives algebraic decay in length and identifies a limiting coefficient, proving Theorem~\ref{t.algebraic.convergence}.

\smallskip

An estimate for the annealed matrices does not by itself control solutions for an individual realization of the coefficient field. Concentration gives a random radius beyond which the required block estimates hold at every larger scale. Deterministic energy comparisons then turn these estimates into bounds on solutions and their fluxes. We carry this out first for uniformly elliptic fields in Theorem~\ref{t.uniform.homogenization}, and then under coarse ellipticity in Theorem~\ref{t.random.homogenization}, where the radius must also exceed the original random scale~$\S$. In the latter argument, the errors are summable over scales: this both controls the Dirichlet homogenization error and allows finite-volume correctors to converge as their domains grow.

\subsection{Algebraic convergence}
\label{ss.algebraic.convergence}

The small-contrast iteration couples two quantities, defined below: the mismatch~$a_j$ between the matrices~$S_j$ and~$S_{*,j}$ in the block representation of the annealed coarse-grained matrix, and the variance~$v_j$ of the normalized coarse-grained matrix. Averaging bounds~$v_k$ by an earlier squared mismatch~$a_l^2$ and an error that decays with the separation~$k-l$, as in~\eqref{e.algebraic.variance}. The optimizer comparison then bounds~$a_n$ by these variances and weighted mean decreases~$a_k-a_n$, as in~\eqref{e.algebraic.response.recursion}. Moving the negative copies of~$a_n$ to the left gives a strict contraction; the squared mismatch can be absorbed because polynomial entry has made it small. Before iterating, we choose a normalization and a starting scale at which the required moment bounds are independent of the original contrast.

\begin{proof}[Proof of Theorem~\ref{t.algebraic.convergence}]
Fix a tolerance~$\sigma=\sigma(d,\gamma)>0$, to be chosen below. Apply Theorem~\ref{t.polynomial.entry} with this tolerance and let~$m_{\rm ent}$ be the resulting scale. We use the small-contrast argument of~\cite[Section~4.3]{AK.HC}, giving the moment initialization and iteration explicitly.

\smallskip

Consider the law of~$\a(3^{m_{\rm ent}}\cdot)$. It is~$\Zd$-stationary, has dependence range at most~$3^{-m_{\rm ent}}\leq1$, and satisfies~\eqref{e.coarse.ellipticity} with the same~$\bfE$ and~$\gamma$ and random scale~$3^{-m_{\rm ent}}\S$. Indeed, dilation sends~$y+\cu_k$ to~$3^{m_{\rm ent}}y+\cu_{m_{\rm ent}+k}$, and~\eqref{e.coarse.ellipticity} holds at every integer scale. Moreover,
\begin{equation*}
	\P[3^{-m_{\rm ent}}\S>t]\leq\Psi_{\S}(3^{m_{\rm ent}}t)^{-1}\leq\Psi_{\S}(t)^{-1}\qquad(t>0)\,.
\end{equation*}
Thus the original tail function and its growth constant remain admissible. Until the final change of scale, all blocks below refer to this dilated law; its contrast at scale zero is at most~$1+\sigma$.

\smallskip

\emph{Step 1: A normalization with bounded moments.} Put~$F=\bfAhom(\cu_0)$, and let~$S$ be the Schur complement of its lower-right block. Choose~$\qq=\mathcal Q(S)$ with a fixed dimensional rounding scale~$j_*$ large enough for the separation used in Lemma~\ref{l.fixed.geometry.matrix.averaging}. We write~$F_j=\bfAhom(\cus_j^\qq)$. Constants~$C$ in this proof depend only on~$d,\gamma$. We use~$H\geq2$ for a deterministic bound satisfying
\begin{equation*}
	H\leq C(2+\Pi K_{\Psi_{\S}})^C\,.
\end{equation*}
The same bound can be enlarged a fixed number of times below.

\smallskip

The moments of~$\S$ and~\eqref{e.coarse.ellipticity}, applied inside the unit parent of a standard cube, give
\begin{equation}
	\bigl\|F^{-\nf12}\bfA(z+\cu_r)F^{-\nf12}\bigr\|_{L^Q(S_Q)}
	\leq H3^{-\gamma r}\qquad(r\leq0,\ z\in3^r\Zd)\,.
	\label{e.algebraic.unit.moments}
\end{equation}
Here~$Q$ is as in~\eqref{e.scale.selection.Q.choice}. To check the normalization, the adapted cube bound at the first nonnegative scale above~$\S(z)$ gives~$F\leq C K_{\Psi_{\S}}^C\bfE$. Inverting this inequality and using~\eqref{e.matrix.reference.comparison} and~$F_*\leq F$ gives~$\bfE\leq C\Pi K_{\Psi_{\S}}^C F$. These inequalities prove~\eqref{e.algebraic.unit.moments}; their Schur complements also bound~$|\qq|$ by~$H$.

\smallskip

Fill~$y+\cus_j^\qq$, for~$j\geq0$, by maximal standard cubes with scales at most zero. The full unit cubes have mean~$F$. Their centered sum, divided by~$|\cus_j^\qq|$, has~$L^Q(S_Q)$ norm at most~$CH3^{-\nf{dj}{2}}$ after normalization by~$F$: apply the coloring argument in the proof of Lemma~\ref{l.fixed.geometry.matrix.averaging}, using~$\det\qq\geq2^{-d}$. The smaller cubes occupy relative volume at most~$C3^{r-j}$ at scale~$r<0$, by the boundary-strip argument of Lemma~\ref{l.source.whitney}. Their contribution has norm at most
\begin{equation*}
	CH\sum_{r<0}3^{r-j}3^{-\gamma r}\leq CH3^{-j}\,.
\end{equation*}
Subadditivity therefore gives, uniformly in~$y\in\Rd$,
\begin{equation}
	\left\|\left(F^{-\nf12}\bfA(y+\cus_j^\qq)F^{-\nf12}-\Itwod\right)_+\right\|_{L^Q(S_Q)}
	\leq H3^{-j}\qquad(j\geq0)\,.
	\label{e.algebraic.positive.moments}
\end{equation}
The subscript~$+$ denotes the spectral positive part. Indeed, the positive part of a symmetric matrix bounded above by another has operator norm no larger than the norm of that upper bound; passage to~$S_Q$ costs only a dimensional factor. For~$j<0$, the same filling capped at~$j$ and~\eqref{e.algebraic.unit.moments} give the bound~$H3^{-\gamma j}$ for the full normalized block. Taking expectations in the capped-unit filling also gives
\begin{equation}
	(1+H3^{-j})^{-1}F_*\leq F_j\leq(1+H3^{-j})F\qquad(j\geq0)\,.
	\label{e.algebraic.mean.initialization}
\end{equation}
The lower bound follows from the upper bound by inversion and~$F_{*,j}\leq F_j$; no commutation of expectation and inversion is used.

\smallskip

Choose an integer~$b\geq j_*+5$ with~$b\leq C\log_3(\nf{H}{\sigma})$ sufficiently large. Since~$F\leq(1+6\sigma)F_*$ by~\eqref{e.matrix.block.comparison}, we have, for every~$j\geq b$,
\begin{equation}
	(1+C\sigma)^{-1}F\leq F_j\leq(1+C\sigma)F
	\qand
	\bigl\|F^{-\nf12}\bfA(\cus_j^\qq)F^{-\nf12}\bigr\|_{L^Q(S_Q)}\leq C\,.
	\label{e.algebraic.normalized.moments}
\end{equation}
Polynomial entry makes~$F$ close to~$F_*$, and the choice of~$b$ makes the boundary and fluctuation errors small in the~$F$-normalization. The iteration therefore starts with a bounded moment, independent of the original contrast.

\smallskip

\emph{Step 2: The weak-norm estimate.} Set
\begin{equation*}
	\rho=\frac{1+\gamma}{2}\,,\qquad
	\zeta=\rho-\frac dQ>\gamma\,,\qquad
	\beta=\frac{1-\rho}{2}>0\,.
\end{equation*}
For~$n\geq0$, define the positive-excess maximum
\begin{equation*}
	\mathcal M_n=
	\sup_{r\leq n}3^{-\rho(n-r)}
	\max_{z\in3^r\qq\Zd\cap\cus_n^\qq}
	\left|\left(F^{-\nf12}\bfA(z+\cus_r^\qq)F^{-\nf12}-\Itwod\right)_+\right|\,.
\end{equation*}
There are~$3^{d(n-r)}$ cells at scale~$r$. The maximum is bounded in~$L^Q$ by this number to the power~$\nf{1}{Q}$ times the uniform cell bound. Thus~\eqref{e.algebraic.positive.moments} and the bound for negative scales give
\begin{equation}
	\|\mathcal M_n\|_{L^Q}
	\leq H3^{-\zeta n}
	\left(\sum_{r=0}^n3^{-(1-\zeta)r}
	+\sum_{r<0}3^{(\zeta-\gamma)r}\right)
	\leq CH3^{-\zeta n}\,.
	\label{e.algebraic.maximum.decay}
\end{equation}
Both series converge since~$\gamma<\zeta<1$. Increasing~$b$ makes~$H3^{-\zeta b}\leq1$.

\smallskip

Write~$S_j,S_{*,j},K_j$ for the matrices in the block representation of~$F_j$, and put
\begin{equation*}
	a_j=\frac1d\tr(S_{*,j}^{-1}S_j)-1
	\qand
	v_j=\E\left[\left|F^{-\nf12}\bigl(\bfA(\cus_j^\qq)-F_j\bigr)F^{-\nf12}\right|^2\right]\,.
\end{equation*}
For~$j\geq b$, the sequence~$a_j$ is nonnegative and decreasing, and~\eqref{e.algebraic.normalized.moments} gives~$a_j\leq C\sigma$. The fluctuation estimate~\cite[Lemma~4.4]{AK.HC}, with terminal scale~$k$ and subdivision scale~$l$, followed by Lemma~\ref{l.fixed.geometry.matrix.averaging} with~$N=2$ and~\eqref{e.algebraic.normalized.moments}, yields
\begin{equation}
	v_k\leq C a_l^2+C3^{-d(k-l)}\qquad(b\leq l\leq k)\,.
	\label{e.algebraic.variance}
\end{equation}
All normalization changes between~$F$ and the annealed blocks cost only~$C$.

\smallskip

We record precisely the weak-norm input from the proof of~\cite[Lemma~4.5, (4.19)--(4.22)]{AK.HC}. At terminal scale~$n$, subtract~$\nf{K_n-K_n^t}{2}$ from the coefficient and use the primal and adjoint loads~$p=S_{*,n}^{-\nf12}e$,~$q=S_{*,n}^{\nf12}e$, with~$|e|=1$. The block formulas, small contrast and~\eqref{e.algebraic.normalized.moments} give the energy-normalized load bounds and the metric-pairing bound
\begin{equation*}
	\biggl|F^{\nf12}\binom{-p}{\pm q-\frac12(K_n-K_n^t)p}\biggr|\leq C\qand |\qq S_{*,n}^{-\nf12}|\,|S_{*,n}^{\nf12}\qq^{-1}|\leq C\,.
\end{equation*}
The two signs correspond to the recentered primal and adjoint normalizers. Apply~\cite[Lemma~2.16]{AK.HC} with the corresponding congruence transform of~$F$,~$s=\nf{1}{2}$, cutoff level~$1$, and~$h=n-b$. Its maximum is exactly~$\mathcal M_n$. The squared sum of matrix differences has weights~$3^{-\nf{n-k}{2}}$. For the square-root mean differences, Cauchy--Schwarz gives the weaker weights~$3^{-\beta(n-k)}$, which we use for both sums. The mean-difference estimate in Lemma~4.4 bounds these terms by~$C(a_k-a_n)$.

\smallskip

The last term of Lemma~2.16 contributes at most
\begin{equation*}
	C3^{-(1-\rho)(n-b)}
	+C\E\bigl[\mathcal M_n(1+\mathcal M_n)\indc_{\{\mathcal M_n>1\}}\bigr]
	\leq C3^{-(1-\rho)(n-b)}+CH^2 3^{-2\zeta n}\,.
\end{equation*}
Indeed, the optimizer energy is at most~$C(1+\mathcal M_n)$, and~$Q\geq2$. The mean terms in the primal and dual energy calculations contribute~$C(v_n+a_n^2)$, as in the last calculation of that proof. Comparing the scale-$n$ optimizer with its scale-$(n-4)$ cell optimizers in the integrals weighted by the smooth cutoff contributes~$C(a_{n-4}-a_n)$, already bounded by the weighted difference sum. Consequently, after absorbing~$Ca_n^2$ by choosing~$\sigma$ small, we obtain
\begin{equation}
	a_n\leq C\sum_{k=b}^n3^{-\beta(n-k)}\bigl(v_k+a_k-a_n\bigr)
	+C3^{-\beta(n-b)}\qquad(n\geq b+4)\,.
	\label{e.algebraic.response.recursion}
\end{equation}
The recursion separates block fluctuations, the accumulated decrease of the mean mismatch, and the old-scale tail. Only annealed smallness is used here; the original reference matrix~$\bfE$ need not have small contrast.

\smallskip

\emph{Step 3: Iteration and return to standard cubes.} Set~$u_j=a_{b+j}$ and, in~\eqref{e.algebraic.variance}, take~$k=b+i$,~$l=b+\lceil\nf{3i}{4}\rceil$. Since~$k-l\geq \nf{i}{4}-1$, summing the variance term in~\eqref{e.algebraic.response.recursion} gives constants~$A\geq1$ and~$\eta>0$, depending only on~$d,\gamma$, such that
\begin{equation}
	u_j\leq A\sum_{i=0}^j3^{-\beta(j-i)}
	\bigl(u_{\lceil\frac{3i}{4}\rceil}^2+u_i-u_j\bigr)+A3^{-\eta j}
	\qquad(j\geq1)\,.
	\label{e.algebraic.shifted.iteration}
\end{equation}
The first three values of~$j$ are included by enlarging~$A$. There is no contrast-dependent prefactor in this inequality.

\smallskip

The negative copies of~$u_j$ provide a strict contraction in the linear terms. The delayed square can then be absorbed using smallness and induction. Put~$q_0=3^{-\beta}$ and~$T_j=\sum_{h=1}^j q_0^h$. Moving the negative copies of~$u_j$ to the left in~\eqref{e.algebraic.shifted.iteration} gives
\begin{equation*}
	(1+AT_j)u_j
	\leq A\sum_{h=1}^j q_0^h u_{j-h}
	+A\sum_{i=0}^j q_0^{j-i}u_{\lceil\frac{3i}{4}\rceil}^2+A3^{-\eta j}\,.
\end{equation*}
Choose~$r\in(\max\{q_0,3^{-\eta}\},1)$ sufficiently close to one that
\begin{equation*}
	\sup_{j\geq1}
	\frac{A\sum_{h=1}^j q_0^h r^{-h}}{1+AT_j}<1\,.
\end{equation*}
Such a choice exists since, at~$r=1$, the supremum is at most~$\nf{AT_\infty}{1+AT_\infty}<1$. To prove~$u_j\leq Mr^j$ by induction, use~$u_i\leq C\sigma$ and
\begin{equation*}
	u_{\lceil\frac{3i}{4}\rceil}^2
	\leq\min\{(C\sigma)^2,M^2r^{2\lceil\frac{3i}{4}\rceil}\}
	\leq(C\sigma)^{\frac{2}{3}}M^{\frac{4}{3}}r^i\,.
\end{equation*}
For~$i\leq3$, where the delayed index can equal~$i$, use~$(C\sigma)^2$ directly; their total contribution is at most~$C\sigma^2r^j$. Choose~$M$ to absorb the fixed forcing term, then~$\sigma(d,\gamma)>0$ small enough to absorb the delayed-square sum in the strict contraction margin. This proves~$a_n\leq C3^{-c(n-b)}$ for~$n\geq b$, with~$c=c(d,\gamma)>0$.

\smallskip

Finally, the adapted-to-standard comparison~\cite[Lemma~2.15]{AK.HC} and~$\bfE\leq HF$ give
\begin{equation*}
	\Theta_m-1\leq C a_n+CH3^{-(m-n)}\qquad(m>n\geq b)\,.
\end{equation*}
For clarity, the small-contrast estimate used here is~$|S_{*,n}^{-\nf12}(K_n+K_n^t)S_{*,n}^{-\nf12}|\leq d a_n$, obtained by optimizing the inequality~$F_{*,n}\leq F_n$ as in the proof of~\cite[Lemma~2.6]{AK.HC}. It bounds the skew-corrected contrast by~$1+Ca_n$; the comparison error is measured in uniformly equivalent normalized blocks. For~$m\geq2b+2$, choose~$n=\lfloor \nf{m}{2}\rfloor+b$. Then~$n-b\geq \nf{m}{2}-1$ and~$m-n\geq\nf{m-2b}{2}$, so both errors are bounded by~$CH3^{-c'(m-2b)}$ for some~$c'=c'(d,\gamma)>0$. Enlarging the dilated starting scale by~$C\log_3 H$ absorbs this prefactor. Returning to the original law gives~\eqref{e.algebraic.contrast.decay} at a scale~$m_0$ with~$m_0-m_{\rm ent}\leq C\log_3(2+\Pi K_{\Psi_{\S}})$. Together with~\eqref{e.polynomial.entry}, this proves~\eqref{e.algebraic.entry}.

\smallskip

By subadditivity and the order~$\bfAhom_*(U)\leq\bfAhom(U)$, for every~$m\geq m_0$,
\begin{equation*}
	\mathbf R\bfAhom(\cu_{m_0})^{-1}\mathbf R\leq\bfAhom(\cu_m)\leq\bfAhom(\cu_{m_0})\,.
\end{equation*}
Thus~$\bfAhom(\cu_m)$ decreases to a positive definite matrix~$\bfAhom$. Apply~\eqref{e.matrix.block.comparison} and minimize over the skew matrix in~\eqref{e.Theta.m} to obtain
\begin{equation*}
	\bfAhom(\cu_m)\leq\bigl(1+6(\Theta_m-1)\bigr)\mathbf R\bfAhom(\cu_m)^{-1}\mathbf R\,.
\end{equation*}
Letting~$m\to\infty$ and using the reverse order~$\bfAhom\geq\mathbf R\bfAhom^{-1}\mathbf R$ from~\eqref{e.matrix.annealed.sharp.order} gives~$\bfAhom=\mathbf R\bfAhom^{-1}\mathbf R$. Comparing blocks using~\eqref{e.annealed.schur} yields~$\shom_* = \shom>0$ and~$\khom=-\khom^t$, so~$\bfAhom$ is the block associated with~$\ahom=\shom+\khom$. Finally, order reversal under inversion gives~$\mathbf R\bfAhom(\cu_m)^{-1}\mathbf R\leq\bfAhom$, and the preceding comparison yields
\begin{equation*}
	0\leq\bfAhom(\cu_m)-\bfAhom\leq6(\Theta_m-1)\bfAhom\,.
\end{equation*}
Substituting~\eqref{e.algebraic.contrast.decay} proves~\eqref{e.algebraic.block.decay}.
\end{proof}

\subsection{Uniformly elliptic fields}
\label{ss.uniform.homogenization}

\begin{proof}[Proof of Theorem~\ref{t.uniform.homogenization}]
It suffices to consider~$\delta\in(0,1]$. Constants denoted by~$C$ and~$c$ in the proof depend only on~$d$, unless another dependence is indicated, and may change from line to line. Uniform ellipticity allows us to take~$\S=0$. Concentration then turns Theorem~\ref{t.algebraic.convergence} into an estimate valid for an individual coefficient field on all the triadic subcubes needed below, beyond a random radius~$\X_0$ whose size is bounded polynomially in the ellipticity ratio, up to a random factor with stretched-exponential tails.

\smallskip

To pass from coarse matrices to solutions, divide the domain into cubes large enough for homogenization but small compared with the domain. On an interior cube~$Q$, replace~$u$ by the~$\a$-harmonic function~$v_Q$ with the same boundary data. The forcing contributes a small error because~$Q$ is small, while the coarse-matrix bound controls the difference between the average fluxes~$(\a\nabla v_Q)_Q$ and~$(\ahom\nabla v_Q)_Q$. A constant-coefficient adjoint problem converts these average-flux bounds into an estimate for~$u-\uhom$. Its regularity controls the variation within each cube and the uncovered boundary layer. Choosing the mesh size as a power of the radius balances these errors. Finally, the summable block errors give the large-scale regularity estimate of~\cite{AK.HC}, after enlarging the same random radius.

\smallskip

\emph{Step 1: The quenched block estimate.} Applying the second inequality in~\eqref{e.uniform.ellipticity} to~$\a e$ gives
\begin{equation*}
	\Lambda^{-1}|\a e|^2\leq e\cdot\a e\leq|e|\,|\a e|\,,
\end{equation*}
so~$|\a|\leq\Lambda$. Condition~\eqref{e.uniform.ellipticity} also gives~$\s^{-1}\leq\lambda^{-1}\Id$ and~$\s+\k^t\s^{-1}\k\leq\Lambda\Id$. By the variational formula for~$\bfA(U)$, using the zero competitor, the assumptions of Theorem~\ref{t.algebraic.convergence} hold with
\begin{equation*}
	\gamma=0\,,\qquad\S=0\,,\qquad\bfE=\begin{pmatrix}2\Lambda\Id&0\\0&2\lambda^{-1}\Id\end{pmatrix}\,,\qquad\Pi=4\frac{\Lambda}{\lambda}\,.
\end{equation*}
We may take~$\Psi_{\S}(t)=\exp(t^2)$ and~$K_{\Psi_{\S}}=3$. Applying the same variational bound separately to the two diagonal blocks and passing to the limit gives
\begin{equation}
	\lambda\Id\leq\shom\,,\qquad\shom+\khom^t\shom^{-1}\khom\leq\Lambda\Id\,,\qquad|\ahom|\leq\Lambda\,.
	\label{e.uniform.homogenized.bounds}
\end{equation}
In particular,~$B_r\subseteq E_r\subseteq B_{r(\nf{\Lambda}{\lambda})^{\nf12}}$.

\smallskip

Unit range gives~$(\beta,\nu)=(0,\nf{d}{2})$ and~$\Psi(t)=\max\{1,\frac12\exp(c(d)t^2)\}$ in the concentration assumption of~\cite{AK.HC}, denoted~(P3) there; see~\cite[discussion following~(P3)]{AK.HC}. The associated~$K_\Psi$ depends only on~$d$. We use the stopping construction in the proof of~\cite[Corollary~4.3]{AK.HC}, with~$\rho=\nf{1}{4}$ and the scale~$m_0$ of Theorem~\ref{t.algebraic.convergence}. Its algebraic input is~\eqref{e.algebraic.block.decay}. In that construction, let the tolerance decrease geometrically, slowly enough that the cutoff has an affine upper bound in the scale with slope at most~$\frac12$. Meeting the remaining scale conditions in the proof of~\cite[Corollary~4.3]{AK.HC} costs at most~$C\log_3(2+\Pi)$ additional scales: the concentration growth constant and the growth constant of the tail of~$\S$ depend only on~$d$, and the reference contrast is bounded by the reference ellipticity ratio~$\Pi$. Thus there exist~$\theta(d)>0$ and a random variable~$\X_0\geq1$ such that
\begin{equation}
	\P\bigl[\X_0\geq C3^{m_0}(2+\Pi)^C t\bigr]\leq\exp(-t^d)\qquad(t\geq1)\,,
	\label{e.uniform.quenched.tail}
\end{equation}
and, almost surely, for every~$m\in\N$ with~$3^m\geq\X_0$,
\begin{equation}
	\sum_{k=-\infty}^{m}3^{-\frac14(m-k)}\max_{z\in3^k\Zd\cap\cu_m}\bigl|\bigl(\bfAhom^{-\nf12}(\bfA(z+\cu_k)-\bfAhom)\bfAhom^{-\nf12}\bigr)_+\bigr|\leq C\left(\frac{3^m}{\X_0}\right)^{-\theta}\,.
	\label{e.uniform.quenched.row}
\end{equation}
In defining~$\X_0$, we multiply the supremum of the radii at which the estimate fails by~$3$, so that~\eqref{e.uniform.quenched.row} holds with the non-strict inequality~$3^m\geq\X_0$. The power~$d$ in~\eqref{e.uniform.quenched.tail} comes from evaluating the finite-range Gaussian concentration gauge at~$t^{\nf d2}$ and absorbing its prefactor into the dimensional scale constant. Since~$\S=0$, the events involving~$\S$ in the construction hold identically and contribute no probability term. The estimate~\eqref{e.uniform.quenched.row} retains every integer child scale.

\smallskip

Choose
\begin{equation}
	\kappa\coloneqq\frac{2\theta}{5}\qquad\text{and}\qquad\X\coloneqq C\delta^{-\nf5{2\theta}}\left(\frac{\Lambda}{\lambda}\right)^{\nf{15}{2\theta}}(2+\Pi)^{\nf2\theta}\X_0\,,
	\label{e.uniform.final.scale}
\end{equation}
where~$C$ is chosen large enough to obtain~\eqref{e.uniform.mesh.blocks} and, in Step 4, to satisfy the smallness condition in~\cite[Proposition~5.10]{AK.HC}. On the probability-one event in~\eqref{e.uniform.quenched.row}, fix any~$r\geq\X$ and set
\begin{equation}
	\tau\coloneqq c\delta^2\left(\frac{\lambda}{\Lambda}\right)^6\left(\frac{r}{\X}\right)^{-2\kappa}\,,
	\label{e.uniform.mesh.size}
\end{equation}
where~$c$ is small enough that~$\tau\leq(4\sqrt d)^{-1}$ and the coefficient in~\eqref{e.uniform.normalized.dirichlet} is at most~$\delta(\nf{r}{\X})^{-\kappa}$. Choose integers~$m,k$ such that~$E_r\subseteq\cu_m$,~$2r(\nf{\Lambda}{\lambda})^{\nf12}\leq3^m\leq6r(\nf{\Lambda}{\lambda})^{\nf12}$, and~$\nf{\tau r}{3}<3^k\leq\tau r$. Each triadic cube~$z+\cu_k\subseteq E_r$ then satisfies, by~\eqref{e.uniform.quenched.row},
\begin{equation}
	\bfA(z+\cu_k)\leq\left(1+C\left(\frac{\Lambda}{\lambda}\right)^{\nf18}\tau^{-\nf14}\left(\frac{r}{\X_0}\right)^{-\theta}\right)\bfAhom\leq(1+\tau)\bfAhom\,.
	\label{e.uniform.mesh.blocks}
\end{equation}
The last inequality follows from~\eqref{e.uniform.final.scale} and~\eqref{e.uniform.mesh.size}: since~$\frac52\kappa=\theta$, the powers of~$r$ cancel when comparing the block excess with~$\tau$.

\smallskip

\emph{Step 2: The average flux on an interior cube.} Rescale~$E_r$ to~$E_1$. In this step and the next, the symbols~$\a,u,\uhom,g,f$ denote the rescaled coefficient and functions; in particular, the rescaled forcing is~$r^2f(r\cdot)$. The cubes~$Q=r^{-1}(z+\cu_k)\subseteq E_1$ have side length~$\ell\coloneqq\nf{3^k}{r}\in(\nf{\tau}{3},\tau]$, and their coarse blocks satisfy~\eqref{e.uniform.mesh.blocks} by normalized change of variables. On each such cube let~$v_Q\in u+H^1_0(Q)$ be~$\a$-harmonic. Testing the equation for~$u-v_Q$ and using the cube Poincar\'e inequality gives
\begin{equation}
	\|\nabla(u-v_Q)\|_{L^2(Q)}\leq C\ell\lambda^{-1}\|f\|_{L^2(Q)}\,.
	\label{e.uniform.local.forcing}
\end{equation}
Since~$\bfAhom$ is the block associated with~$\ahom$, for every~$p\in\Rd$,
\begin{equation*}
	(-p,\ahom^tp)\cdot\bfAhom(-p,\ahom^tp)=2p\cdot\shom p\,.
\end{equation*}
The averaged-flux inequality~\cite[(2.33)]{AK.HC}, applied to~$v_Q$ with second vector~$\ahom^tp$, and~\eqref{e.uniform.mesh.blocks} therefore give
\begin{equation*}
	\bigl|p\cdot\bigl((\a-\ahom)\nabla v_Q\bigr)_Q\bigr|\leq\bigl(2\tau p\cdot\shom p\bigr)^{\nf12}\bigl(\fint_Q\nabla v_Q\cdot\s\nabla v_Q\bigr)^{\nf12}\,.
\end{equation*}
The triangle inequality gives
\begin{equation*}
	\|\s^{\nf12}\nabla v_Q\|_{L^2(Q)}\leq\Lambda^{\nf12}\bigl(\|\nabla u\|_{L^2(Q)}+\|\nabla(u-v_Q)\|_{L^2(Q)}\bigr)\,.
\end{equation*}
Taking the supremum over~$|p|=1$ in the flux estimate and using~\eqref{e.uniform.local.forcing},~$\shom\leq\Lambda\Id$, and~$|\a-\ahom|\leq2\Lambda$ gives
\begin{equation}
	\bigl|\bigl((\a-\ahom)\nabla u\bigr)_Q\bigr|\leq C\Lambda\sqrt\tau\bigl(\fint_Q|\nabla u|^2\bigr)^{\nf12}+C\frac{\Lambda}{\lambda}\ell\bigl(\fint_Q|f|^2\bigr)^{\nf12}\,.
	\label{e.uniform.cell.flux}
\end{equation}

\smallskip

\emph{Step 3: The forced Dirichlet estimate.} Let~$\psi\in H^1_0(E_1)$ solve
\begin{equation*}
	-\nabla\cdot\ahom^t\nabla\psi=u-\uhom\qquad\text{in }E_1\,.
\end{equation*}
The change of variables~$x=\overline\lambda^{-\nf12}\shom^{\nf12}y$ takes~$B_1$ onto~$E_1$ and reduces this equation to the Laplacian with coefficient~$\overline\lambda$. The~$H^2$ estimate on the ball~\cite[Theorem~0.3]{JeKe} and~$|\overline\lambda^{\nf12}\shom^{-\nf12}|\leq1$ give
\begin{equation}
	\|\psi\|_{H^2(E_1)}\leq C\lambda^{-1}\|u-\uhom\|_{L^2(E_1)}\,.
	\label{e.uniform.adjoint.regularity}
\end{equation}
For~$v\in H^1(B_1)$, the one-dimensional~$H^1$ estimate on radial segments gives
\begin{equation*}
	\int_{B_1\setminus B_{1-\delta}}|v|^2\leq C\delta\|v\|_{H^1(B_1)}^2\qquad(0<\delta<1)\,.
\end{equation*}
Apply this estimate to the components of the transformed gradient with~$\delta=\sqrt d\,\ell$. Returning to~$E_1$ and using~\eqref{e.uniform.adjoint.regularity}, we obtain
\begin{equation}
	\|\nabla\psi\|_{L^2(E_1\setminus E_{1-\sqrt d\,\ell})}\leq C\sqrt\ell\,\lambda^{-1}\|u-\uhom\|_{L^2(E_1)}\,.
	\label{e.uniform.adjoint.layer}
\end{equation}
The interior mesh cubes cover~$E_1$ except for a set at distance at most~$\sqrt d\,\ell$ from its boundary, which is contained in the shell in~\eqref{e.uniform.adjoint.layer}. We separate the cell-average flux error, the oscillation of the adjoint gradient inside each cell, and the uncovered boundary layer. Testing the two equations with~$\psi$, and the adjoint equation with~$u-\uhom$, yields
\begin{align}
	\|u-\uhom\|_{L^2(E_1)}^2
	&=-\sum_Q|Q|(\nabla\psi)_Q\cdot\bigl((\a-\ahom)\nabla u\bigr)_Q
	\notag\\ &\qquad-\sum_Q\int_Q\bigl(\nabla\psi-(\nabla\psi)_Q\bigr)\cdot(\a-\ahom)\nabla u
	\notag\\ &\qquad-\int_{E_1\setminus\bigcup_Q Q}\nabla\psi\cdot(\a-\ahom)\nabla u\,.
	\label{e.uniform.duality.split}
\end{align}
Use~\eqref{e.uniform.cell.flux} for the first sum, the cube Poincar\'e inequality for the second, and~\eqref{e.uniform.adjoint.layer} for the last integral. Cauchy--Schwarz,~$|\a-\ahom|\leq2\Lambda$,~\eqref{e.uniform.adjoint.regularity}, and~$\ell\leq\tau\leq1$ give, after cancelling~$\|u-\uhom\|_{L^2(E_1)}$ when it is nonzero,
\begin{equation}
	\|u-\uhom\|_{L^2(E_1)}\leq C\frac{\Lambda}{\lambda}\sqrt\tau\|\nabla u\|_{L^2(E_1)}+C\frac{\Lambda}{\lambda^2}\tau\|f\|_{L^2(E_1)}\,.
	\label{e.uniform.forced.energy.error}
\end{equation}
Testing the equation for~$u$ with~$u-g$, using~$\operatorname{diam}(E_1)\leq2(\nf{\Lambda}{\lambda})^{\nf12}$ in Poincar\'e's inequality, gives
\begin{equation*}
	\|\nabla u\|_{L^2(E_1)}\leq C\frac{\Lambda}{\lambda}\|\nabla g\|_{L^2(E_1)}+C\lambda^{-1}\left(\frac{\Lambda}{\lambda}\right)^{\nf12}\|f\|_{L^2(E_1)}\,.
\end{equation*}
Since~$\lambda^{-1}\leq\nf{\Lambda}{\lambda}$, substituting this bound into~\eqref{e.uniform.forced.energy.error} gives
\begin{equation}
	\|u-\uhom\|_{L^2(E_1)}\leq C\left(\frac{\Lambda}{\lambda}\right)^3\sqrt\tau\bigl(\|\nabla g\|_{L^2(E_1)}+\|f\|_{L^2(E_1)}\bigr)\,.
	\label{e.uniform.normalized.dirichlet}
\end{equation}
By~\eqref{e.uniform.mesh.size} and the choice of~$c$, the coefficient on the right is at most~$\delta(\nf{r}{\X})^{-\kappa}$. Returning to~$E_r$ gives~\eqref{e.uniform.dirichlet}, since the gradient datum scales by~$r$ and the forcing by~$r^2$. All tests above are valid for~$f\in L^2$ and~$g\in H^1$; no extra regularity of either datum was used.

\smallskip

\emph{Step 4: Large-scale energy and the common scale.} The radius already chosen makes the Dirichlet error small. Its constant will also ensure the summable-error threshold for large-scale regularity. Return to the original variables. For the rounded geometry~$\qq=\mathcal Q(\shom)$ of Section~\ref{s.scale.selection}, let~$\mathcal E_{\nf{1}{4},2}(\cus_m^\qq;\a,\ahom)$ denote the multiscale error of~\cite[Definition~5.1]{AK.HC}, formed on the corresponding adapted subcubes. The relation between the block excess and that error, followed by the boundary-layer subdivision of adapted cubes into Euclidean cubes, gives
\begin{equation}
	\mathcal E_{\frac{1}{4},2}(\cus_m^\qq;\a,\ahom)^2\leq C\Pi^{\nf34}\left(\frac{3^m}{\X_0}\right)^{-\theta}\qquad(3^m\geq\X_0)\,.
	\label{e.uniform.adapted.error}
\end{equation}
Here the Euclidean weights are~$3^{-\frac12(m-k)}\leq3^{-\frac14(m-k)}$, so~\eqref{e.uniform.quenched.row} applies. In the adapted-cube comparison, the boundary-layer weights~$3^{j-k}$ are summed against~$3^{-\frac12(m-k)}$; the resulting factor is~$C\Pi^{\nf34}$. This is the comparison in the proof of~\cite[Theorem~B]{AK.HC}, with~$s=\rho=\nf{1}{4}$. Taking square roots before summing yields
\begin{equation}
	\sum_{m=n}^\infty\mathcal E_{\frac{1}{4},2}(\cus_m^\qq;\a,\ahom)\leq C\Pi^{\nf38}\left(\frac{3^n}{\X_0}\right)^{-\nf\theta2}\qquad(3^n\geq\X_0)\,.
	\label{e.uniform.adapted.sum}
\end{equation}
Subtract~$\khom$ and make the affine change of variables generated by~$\qq$. After multiplication by~$\overline\lambda^{-1}$, the homogenized coefficient has ellipticity bounds depending only on~$d$, because~$\qq$ differs from~$\overline\lambda^{-\nf12}\shom^{\nf12}$ by a matrix of norm at most~$\nf{1}{2}$. Indeed,~\eqref{e.rounded.grid.bounds} gives
\begin{equation*}
	\bigl|\qq^{-1}\overline\lambda^{-\nf12}\shom^{\nf12}\bigr|\leq2\qand\bigl|\overline\lambda^{\nf12}\shom^{-\nf12}\qq\bigr|\leq\frac32\,.
\end{equation*}
Thus~\cite[Proposition~5.10]{AK.HC} applies whenever the right side of~\eqref{e.uniform.adapted.sum} is below a threshold depending only on~$d$. The scale~\eqref{e.uniform.final.scale}, with~$C$ sufficiently large, ensures this for every adapted scale comparable to a radius at least~$\X$. The two bounds in the preceding display show that adapted cubes and the ellipsoids~$E_r$ contain one another at radii differing by factors depending only on~$d$, with comparable volumes. If~$R\geq C(d)r$, choose nested adapted cubes, the smaller containing~$E_r$ and the larger contained in~$E_R$, with volumes comparable to those of~$E_r$ and~$E_R$, respectively. Proposition~5.10 applied to these cubes gives~\eqref{e.uniform.energy}. If~$R<C(d)r$, the estimate follows from~$E_r\subseteq E_R$ and~$\nf{|E_R|}{|E_r|}\leq C(d)$. The resulting constant~$C_0$ depends only on~$d$.

\smallskip

Finally,~\eqref{e.uniform.final.scale} and~$\Pi=4\nf{\Lambda}{\lambda}$ imply
\begin{equation*}
	\X\leq C\delta^{-\nf5{2\theta}}(2+\Pi)^C\X_0\,.
\end{equation*}
Theorem~\ref{t.algebraic.convergence}, with~$K_{\Psi_{\S}}$ fixed at~$3$, gives~$3^{m_0}\leq3(2+3\Pi)^{C(d)}$. Combining these two bounds with~\eqref{e.uniform.quenched.tail} proves~\eqref{e.uniform.scale.tail}. The deterministic scale-enlargement factors depend only on~$d$,~$\delta$, and~$\nf{\Lambda}{\lambda}$, and both conclusions hold above the same~$\X$.
\end{proof}

\subsection{The random scale and Dirichlet homogenization}
\label{ss.random.dirichlet}

\begin{proof}[Proof of the scale estimate and part~\emph{(a)} of Theorem~\ref{t.random.homogenization}]
We return to the coarse ellipticity assumption~\ref{a.coarse.ellipticity}, with its growth exponent~$\gamma\in[0,1)$ and random scale~$\S$. Recall that~$\P[\S>t]\leq\Psi_{\S}(t)^{-1}$. Theorem~\ref{t.algebraic.convergence} controls the annealed matrices beyond~$m_0$. We now need bounds that hold for an individual coefficient field, simultaneously over the subcubes and scales used in the deterministic homogenization estimates. The quenched construction of~\cite{AK.HC} supplies a random radius~$\X_0$ for the concentration estimate. The block bounds also require the scale to exceed~$\S$, so the final radius~$\X$ is a deterministic multiple of~$\max\{1,\S,\X_0\}$, as in~\eqref{e.random.final.scale}. This is why its tail contains both a concentration term and a contribution from~$\Psi_{\S}$.

\smallskip

We convert the block estimates into a bound on~$\mathcal E_{s_0,2}(\cus_m^\qq;\a,\ahom)$ that is summable in~$m$. This error controls the Dirichlet comparison and will also be used to construct correctors in the next subsection. Choosing one fixed order~$s_0$ makes the radius independent of the output Sobolev order~$s$; the error at every permitted~$s$ is controlled by the error at~$s_0$. We choose the fractional order~$s_0$ and the block-sum weight~$\rho$ so that the decay weights dominate the allowed small-scale growth:
\begin{equation}
	s_0\coloneqq\frac{1+\gamma}{4}\,,\qquad\rho\coloneqq\frac{1+3\gamma}{4}\,,\qquad\rho-\gamma=2s_0-\rho=\frac{1-\gamma}{4}\,.
	\label{e.random.orders}
\end{equation}
Constants denoted by~$C$ depend only on~$d$ and~$\gamma$, unless another dependence is indicated, and may change from line to line.

\smallskip

\emph{Step 1: A summable quenched error.} Apply the stopping construction in the proof of~\cite[Corollary~4.3]{AK.HC} with the weight~$\rho$ in~\eqref{e.random.orders} and concentration exponent~$\nf{d}{2}-\gamma$, retaining~$\S$ and~$\Psi_{\S}$. Meeting the remaining scale conditions in the proof of~\cite[Corollary~4.3]{AK.HC} costs at most~$C\log_3(2+\Pi K_{\Psi_{\S}})$ additional scales, since~$\rho-\gamma=2s_0-\rho=\nf{1-\gamma}{4}>0$. With the endpoint enlarged by one scale, it gives a measurable~$\X_0\geq1$ such that
\begin{equation}
	\P\bigl[\X_0\geq C3^{m_0}(2+\Pi K_{\Psi_{\S}})^C t\bigr]\leq\exp\bigl(-c(d)t^{d-2\gamma}\bigr)\qquad(t\geq1)\,,
	\label{e.random.quenched.tail}
\end{equation}
and, for some~$\kappa(d,\gamma)>0$, almost surely at every integer scale with~$3^m\geq\max\{1,\S,\X_0\}$,
\begin{equation}
	\sum_{k=-\infty}^m3^{-\rho(m-k)}\max_{z\in3^k\Zd\cap\cu_m}\bigl|\bigl(\bfAhom^{-\nf12}(\bfA(z+\cu_k)-\bfAhom)\bfAhom^{-\nf12}\bigr)_+\bigr|\leq C\biggl(\frac{3^m}{\max\{1,\S,\X_0\}}\biggr)^{-2\kappa}\,.
	\label{e.random.quenched.row}
\end{equation}
The sum in~\eqref{e.random.quenched.row} runs over every integer child scale~$k\leq m$, including negative scales. The same concentration gauge evaluated at~$t^{\nf d2-\gamma}$ gives the power in~\eqref{e.random.quenched.tail}, after absorbing its prefactor into the scale constant~$C(d,\gamma)$. The random scale~$\S$ appears in the condition~$3^m\geq\max\{1,\S,\X_0\}$ and in the ratio on the right side of~\eqref{e.random.quenched.row}. The tail of~$\X_0$ itself satisfies the concentration bound~\eqref{e.random.quenched.tail}.

\smallskip

Let~$\qq=\mathcal Q(\shom)$ and use the adapted error~$\mathcal E_{s,2}(\cus_m^\qq;\a,\ahom)$ of Step 4 of the preceding proof. The reference-block comparison and the limit in Theorem~\ref{t.algebraic.convergence} give~$|\shom|\,|\shom^{-1}|\leq\Pi$. The block-to-error comparison and the adapted-cube subdivision in the proof of~\cite[Theorem~B]{AK.HC} therefore yield
\begin{equation}
	\mathcal E_{s_0,2}(\cus_m^\qq;\a,\ahom)^2\leq C\Pi^{\nf12(1+2s_0)}\biggl(\frac{3^m}{\max\{1,\S,\X_0\}}\biggr)^{-2\kappa}\qquad\bigl(3^m\geq\max\{1,\S,\X_0\}\bigr)\,.
	\label{e.random.adapted.square}
\end{equation}
Indeed, the Euclidean weights~$3^{-2s_0(m-k)}$ are bounded by the row weights~$3^{-\rho(m-k)}$. The adapted boundary-layer weights~$3^{j-k}$ are summable against them because~$1-2s_0=\nf{1-\gamma}{2}>0$. Taking square roots before summing gives
\begin{equation}
	\sum_{m=n}^{\infty}\mathcal E_{s_0,2}(\cus_m^\qq;\a,\ahom)\leq C\Pi^{\nf14(1+2s_0)}\biggl(\frac{3^n}{\max\{1,\S,\X_0\}}\biggr)^{-\kappa}\qquad\bigl(3^n\geq\max\{1,\S,\X_0\}\bigr)\,.
	\label{e.random.adapted.sum}
\end{equation}
For every output order~$s\in[s_0,\nf{1}{2})$, the definition of the error gives the uniform comparison
\begin{equation}
	\mathcal E_{s,2}(\cus_m^\qq;\a,\ahom)\leq(1-3^{-\nf12})^{-\nf12}\mathcal E_{s_0,2}(\cus_m^\qq;\a,\ahom)\,.
	\label{e.random.order.comparison}
\end{equation}
Here~$3^{-2s(m-k)}\leq3^{-2s_0(m-k)}$ and~$\nf{1-3^{-2s}}{1-3^{-2s_0}}\leq(1-3^{-\nf12})^{-1}$. Thus only~$s_0$, not the output order, enters the choice of scale.

\smallskip

Define
\begin{equation}
	\X\coloneqq C(2+\Pi)^C\max\{1,\S,\X_0\}\,.
	\label{e.random.final.scale}
\end{equation}
We choose the exponent~$C(d,\gamma)$ so that~$C\kappa\geq\nf{d+3}{2}$, and the multiplicative constant sufficiently large for the estimates below and the fixed-order thresholds in part~\emph{(b)}. In particular,~$\X\geq\S$. The radius is fixed using~$s_0$; the output orders, domains, data, and approximation exponents will enter only the deterministic arguments above this radius. Theorem~\ref{t.algebraic.convergence} gives~$3^{m_0}\leq3(2+\Pi K_{\Psi_{\S}})^{C(d,\gamma)}$. For any~$t\geq1$,
\begin{equation*}
	\P\bigl[\max\{1,\S,\X_0\}\geq2t\bigr]\leq\P[\X_0>t]+\P[\S>t]\,.
\end{equation*}
Apply this inequality at the polynomial normalizer in~\eqref{e.random.quenched.tail}, use the tail bound~\eqref{e.source.tail} for~$\S$, and weaken its argument by monotonicity. Together with~\eqref{e.random.final.scale}, this proves~\eqref{e.random.scale.tail}, after changing~$C$. The argument uses no independence between~$\S$ and~$\X_0$ and retains the tail gauge~$\Psi_{\S}$ evaluated at~$c_{\rm src}(d,\gamma)t$.

\smallskip

\emph{Step 2: The intrinsic Dirichlet estimate.} Fix~$s,U,\varepsilon,g$ as in part~\emph{(a)}. Choose~$m\in\N$ with~$9\leq\varepsilon3^m<27$ and put~$V\coloneqq(\varepsilon3^m\qq)^{-1}U$. The rounding bounds give
\begin{equation}
	\bigl|\qq^{-1}\overline\lambda^{-\nf12}\shom^{\nf12}\bigr|\leq2\qand\bigl|\overline\lambda^{\nf12}\shom^{-\nf12}\qq\bigr|\leq\frac32\,.
	\label{e.random.intrinsic.bounds}
\end{equation}
Consequently,~$V\subseteq B_{\nf{2}{9}}\subseteq\cu_0$. Define the transformed coefficient and functions by
\begin{equation*}
	\widetilde\a(y)\coloneqq\overline\lambda^{-1}\qq^{-1}\bigl(\a(3^m\qq y)-\khom\bigr)\qq^{-1}\,,\qquad\widetilde u(y)\coloneqq u^\varepsilon(\varepsilon3^m\qq y)\,,
\end{equation*}
and~$\widetilde h(y)\coloneqq\uhom(\varepsilon3^m\qq y)$,~$\widetilde g(y)\coloneqq g(\varepsilon3^m\qq y)$. Subtracting the constant anti-symmetric matrix does not change the weak equations. The symmetric reference coefficient~$\overline\lambda^{-1}\qq^{-1}\shom\qq^{-1}$ lies between~$\frac49\Id$ and~$4\Id$ by~\eqref{e.random.intrinsic.bounds}. Affine covariance and invariance under this subtraction give
\begin{equation}
	\mathcal E_{s,2}\bigl(\cu_0;\widetilde\a,\overline\lambda^{-1}\qq^{-1}\shom\qq^{-1}\bigr)=\mathcal E_{s,2}(\cus_m^\qq;\a,\ahom)\leq1\,.
	\label{e.random.transformed.error}
\end{equation}
The last inequality follows from~\eqref{e.random.adapted.sum}--\eqref{e.random.final.scale}, since~$3^m\geq9\varepsilon^{-1}\geq9\X$.

\smallskip

Apply both clauses of~\cite[Proposition~5.3]{AK.HC} on~$V$, at the output order~$s$.\footnote{In the general-domain localization proof, the Whitney cubes are selected as maximal disjoint triadic cubes, with the interior buffer and parent boundary contact. This corrects the printed selection rule; the disjoint signed decomposition and the subsequent estimates then apply.} In the second clause, the initial~$h$ in the printed statement means the boundary datum~$g$. Its attainability clause and the weighted Dirichlet theory give existence and uniqueness for the stated boundary data. The general-domain embedding in~\cite{AK.HC} gives~$\widetilde u\in H^{1-s}(V)$. The weighted zero-boundary approximants to~$\widetilde u-\widetilde g$ converge in~$W^{1,1}(V)$, by Cauchy--Schwarz and local integrability of the inverse symmetric coefficient. \par
We verify the zero-boundary requirement~$\widetilde u-\widetilde h\in H^{1-s}_0(V)$ in the error clause of~\cite[Proposition~5.3]{AK.HC}. Put~$w=\widetilde u-\widetilde g$ and~$r=1-s\in(\nf{1}{2},1)$. Localize~$w$ by a smooth partition of unity and flatten each boundary chart to the half-space. Even reflection and mollification give common smooth approximants in~$W^{1,1}$ and~$H^r$. Their boundary restrictions converge to both traces, so the zero~$W^{1,1}$ trace is also the~$H^r$ trace; see~\cite[Theorem~2.3]{Mikh11} for the latter trace map. The zero-trace characterization~\cite[Theorem~2.9]{Mikh11}, applied on the half-space, puts each localized function in~$H^r_0$. Multiply the half-space approximants by a fixed smooth cutoff supported in the flattened chart and equal to one near the localized function's support, then pull them back and mollify inside~$V$. Bi-Lipschitz changes preserve~$H^r$ for~$0<r<1$, so these approximants converge in~$H^r(V)$. Summing over the partition proves~$w\in H^r_0(V)$, with no topological condition on~$\partial V$. Since~$\widetilde h-\widetilde g\in H^1_0(V)$, it follows that~$\widetilde u-\widetilde h\in H^{1-s}_0(V)$.

\smallskip

We may therefore apply the error clause. Its hatted negative norm tests against all smooth vector fields on~$V$, so it bounds the compact-test dual used here. On~$V$, replacing its positive norm and pairing by~\eqref{e.physical.fractional.norm} and the averaged pairing costs only a constant depending on~$V,s,d$. The constant-coefficient Dirichlet estimate~\cite[Theorem~0.5(b)]{JeKe} also gives~$\|\nabla\widetilde h\|_{\underline H^s(V)}\leq C(V,s,d)\|\nabla\widetilde g\|_{\underline H^s(V)}$. Since the factor~$1+\mathcal E_{s,2}$ is at most two, the two clauses give
\begin{equation}
	\|\nabla\widetilde u-\nabla\widetilde h\|_{\underline H^{-s}(V)}+\|\widetilde\a\nabla\widetilde u-\overline\lambda^{-1}\qq^{-1}\shom\qq^{-1}\nabla\widetilde h\|_{\underline H^{-s}(V)}\leq C(V,s,d)\mathcal E_{s,2}(\cus_m^\qq;\a,\ahom)\|\nabla\widetilde g\|_{\underline H^s(V)}\,.
	\label{e.random.intrinsic.dirichlet}
\end{equation}

\smallskip

For completeness, the affine Sobolev losses can be bounded using only~$\Pi$. The smallest singular value of~$\varepsilon3^m\qq$ is at least~$\nf{9}{2}$, and its largest is at most~$54\Pi^{\nf12}$. Changing variables in~\eqref{e.physical.fractional.norm}, and then in the averaged dual pairing, gives the following bounds, with~$C$ depending only on~$d$:
\begin{equation}
	\|F(\varepsilon3^m\qq\,\cdot)\|_{\underline H^s(V)}\leq C\Pi^{\nf14(d+1)}\|F\|_{\underline H^s(U)}\,,\quad\|F\|_{\underline H^{-s}(U)}\leq C\Pi^{\nf14(d+1)}\|F(\varepsilon3^m\qq\,\cdot)\|_{\underline H^{-s}(V)}\,.
	\label{e.random.affine.norms}
\end{equation}
In the squared positive norm, the mass term has factor~$\det(\varepsilon3^m\qq)^{\nf2d s}$, while the double integral has factor at most~$\nf{|\varepsilon3^m\qq|^{d+2s}}{\det(\varepsilon3^m\qq)}$. Both are bounded by~$C(d)\Pi^{\nf12(d+1)}$. The determinant cancels exactly in the averaged pairing, proving the second inequality from the first.

\smallskip

The physical gradient and centered flux, pulled back to~$V$, are respectively~$\nf{\overline\lambda^{\nf12}}{\varepsilon3^m}$ times the transformed gradient error multiplied by~$\overline\lambda^{-\nf12}\shom^{\nf12}\qq^{-1}$, and the transformed flux error multiplied by~$\overline\lambda^{\nf12}\shom^{-\nf12}\qq$. These matrices have norms at most two and~$\nf{3}{2}$. The boundary gradient satisfies
\begin{equation*}
	\nabla\widetilde g=\frac{\varepsilon3^m}{\overline\lambda^{\nf12}}\qq\overline\lambda^{\nf12}\shom^{-\nf12}(\shom^{\nf12}\nabla g)(\varepsilon3^m\qq\,\cdot)\,.
\end{equation*}
The boundary-gradient factor is the reciprocal of the scalar factor in the pulled-back errors, so they cancel. Only the bounded matrix factors and the affine Sobolev losses remain. \mbox{Applying}~\eqref{e.random.affine.norms}~to~the~boundary~datum and to the dual tests gives
\begin{multline}
	\|\shom^{\nf12}(\nabla u^\varepsilon-\nabla\uhom)\|_{\underline H^{-s}(U)}+\|\shom^{-\nf12}((\a^\varepsilon-\khom)\nabla u^\varepsilon-\shom\nabla\uhom)\|_{\underline H^{-s}(U)}
	\\\leq C_0(U,s,d)\Pi^{\nf12(d+1)}\mathcal E_{s,2}(\cus_m^\qq;\a,\ahom)\|\shom^{\nf12}\nabla g\|_{\underline H^s(U)}\,.
	\label{e.random.dirichlet.before.scale}
\end{multline}
By~\eqref{e.random.intrinsic.bounds}, the domain~$V$ is obtained from~$\overline\lambda^{\nf12}\shom^{-\nf12}U$ by a linear map whose norm and inverse norm are bounded by constants depending only on~$d$. The factor~$\varepsilon3^m$ lies in~$[9,27)$, so all domain constants above have exactly the dependence asserted for~$C_0$.

\smallskip

Finally,~\eqref{e.random.adapted.sum} and~\eqref{e.random.order.comparison}, with~$\nf{1+2s_0}{4}\leq\nf{1}{2}$, imply
\begin{equation*}
	\Pi^{\nf12(d+1)}\mathcal E_{s,2}(\cus_m^\qq;\a,\ahom)\leq C\Pi^{\nf12(d+2)}\bigl(\varepsilon\max\{1,\S,\X_0\}\bigr)^\kappa\leq(\varepsilon\X)^\kappa\,.
\end{equation*}
The last inequality is ensured by~\eqref{e.random.final.scale}, after increasing~$C$. This proves~\eqref{e.random.dirichlet}. The event and radius were chosen before~$s,U,\varepsilon,g$; all these parameters enter only the deterministic argument.
\end{proof}

\subsection{Correctors, Liouville, and large-scale regularity}
\label{ss.random.correctors}

\begin{proof}[Proof of part~\emph{(b)} of Theorem~\ref{t.random.homogenization}]
We use the fixed order~$s_0=\nf{1+\gamma}{4}$ and the same radius~$\X$ as in part~\emph{(a)}. The key input is the summability of~$\mathcal E_{s_0,2}(\cus_j^\qq;\a,\ahom)$ over~$j$. Consider solutions on growing cubes with affine boundary data~$e\cdot x$. Their gradient differences are controlled by the errors at the intervening scales; summability therefore makes the gradients converge locally to~$e+\nabla\phi_e$. The same comparison gives quantitative bounds on the corrector and its flux. For a general harmonic function, we compare approximations by these corrected affine functions at neighboring scales. Summing the differences of their slopes gives one slope valid at every intermediate radius, yielding large-scale regularity. For entire solutions with the growth allowed in the theorem, the Liouville argument makes the corrected affine representation exact. Uniqueness for a prescribed slope then identifies the translated correctors.

\smallskip

To apply the deterministic results of~\cite{AK.HC}, we work in coordinates adapted to the homogenized matrix. The change of variables~$x=\qq y$, subtraction of~$\khom$, and multiplication of the coefficient by~$\overline\lambda^{-1}$ give the reference coefficient~$\overline\lambda^{-1}\qq^{-1}\shom\qq^{-1}$. Its ellipticity bounds are~$\nf{4}{9}$ and~$4$ by~\eqref{e.random.intrinsic.bounds}, so the constants in those results depend only on~$d,\gamma$, and, when specified, the approximation exponent. We keep track of the changes in the norms when returning to the original coordinates.

\smallskip

\emph{Step 1: The summable error and the energy estimate.} Choose the constant in~\eqref{e.random.final.scale} large enough that~\eqref{e.random.adapted.sum} gives
\begin{equation}
	\sum_{j=n}^{\infty}\mathcal E_{s_0,2}(\cus_j^\qq;\a,\ahom)\leq c(d,\gamma)\Pi^{-\nf12}\left(\frac{3^n}{\X}\right)^{-\kappa}\qquad(3^n\geq\X)\,,
	\label{e.random.good.tail}
\end{equation}
where~$c(d,\gamma)>0$ is below the thresholds of~\cite[Propositions~5.10 and~5.12]{AK.HC} at approximation exponent~$\nf{1}{2}$. This is possible because the prefactor in~\eqref{e.random.adapted.sum} is at most~$C\Pi^{\nf12}$ and the exponent in~\eqref{e.random.final.scale} satisfies~$C\kappa\geq\nf{d+3}{2}$. This single choice of~$\X$ is independent of the later exponent~$\theta$.

\smallskip

The summable-error clause of~\cite[Proposition~5.10]{AK.HC} gives, for all~$m\geq n$ with~$3^n\geq\X$ and all finite-energy weak solutions on~$\cus_m^\qq$,
\begin{equation}
	\|\s^{\nf12}\nabla u\|_{\underline L^2(\cus_n^\qq)}\leq C\|\s^{\nf12}\nabla u\|_{\underline L^2(\cus_m^\qq)}\,.
	\label{e.random.cube.energy}
\end{equation}
Here and below finite energy means membership in~$H^1_\s$; this is the solution class in the cited proposition. The normalized energy under the affine change is exactly~$\overline\lambda^{-1}$ times the physical normalized energy, so no power of the reference ellipticity ratio enters~\eqref{e.random.cube.energy}. The bounds~\eqref{e.random.intrinsic.bounds} show that an ellipsoid~$E_r$ and the adapted cubes containing or contained in it have comparable radii and volumes, with constants depending only on~$d$. If~$\nf{R}{r}$ exceeds a sufficiently large constant depending only on~$d$, insert such cubes between~$E_r$ and~$E_R$ and apply~\eqref{e.random.cube.energy}. For the remaining ratios, use inclusion and the volume ratio. This proves~\eqref{e.random.energy}, in fact for every finite-energy weak solution in~$H^1_\s(E_R)$. We will use this stronger intermediate form for differences involving correctors.

\smallskip

\emph{Step 2: The corrector family and the endpoint norm.} For~$m\in\N$ and~$e\in\Rd$, let~$w(\cdot,\cus_m^\qq,e)$ be the weak solution with affine boundary data~$e\cdot x$. This is the finite-volume Dirichlet corrector of~\cite{AK.HC}, on the adapted cube. It is linear in~$e$. The telescope in the proof of~\cite[Proposition~5.12]{AK.HC}, using~\eqref{e.random.good.tail}, gives convergence of these gradients in~$L^2_{\s,\mathrm{loc}}$. Subtracting the average on one fixed cube before passing to the limit fixes the additive constants. Define~$\phi_e$, modulo constants, by
\begin{equation}
	e+\nabla\phi_e\coloneqq\lim_{m\to\infty}\nabla w(\cdot,\cus_m^\qq,e)\qquad\text{in }L^2_{\s,\mathrm{loc}}(\Rd)\,.
	\label{e.random.corrector.limit}
\end{equation}
The coarse-grained Poincar\'e estimate for the harmonic differences supplies the corresponding potentials in~$H^1_{\s,\mathrm{loc}}$. Passing to the weak equation is valid by~\eqref{e.qualitative.ellipticity} and Cauchy--Schwarz. The family of gradients is linear in~$e$. More precisely, the same telescope, with two scales reserved for the interior estimates, gives
\begin{equation}
	\|\s^{\nf12}\nabla(e\cdot x+\phi_e-w(\cdot,\cus_{n+2}^\qq,e))\|_{\underline L^2(\cus_n^\qq)}\leq C|\shom^{\nf12}e|\sum_{j=n+2}^{\infty}\mathcal E_{s_0,2}(\cus_j^\qq;\a,\ahom)\,.
	\label{e.random.corrector.telescope}
\end{equation}
All estimates in this step hold for~$3^n\geq\X$. By~\eqref{e.random.good.tail}, the total error is below the fixed threshold in the corrector energy estimate of that proposition. It yields
\begin{equation}
	C^{-1}|\shom^{\nf12}e|\leq\|\s^{\nf12}(e+\nabla\phi_e)\|_{\underline L^2(\cus_n^\qq)}\leq C|\shom^{\nf12}e|\,.
	\label{e.random.corrector.energy}
\end{equation}
The finite-volume~$L^2$ comparison with affine data, together with~\eqref{e.random.corrector.telescope} and the coarse-grained Poincar\'e inequality applied to the difference of the two harmonic functions, also gives
\begin{equation}
	\overline\lambda^{\nf12}3^{-n}\|\phi_e-(\phi_e)_{\cus_n^\qq}\|_{\underline L^2(\cus_n^\qq)}\leq C|\shom^{\nf12}e|\sum_{j=n}^{\infty}\mathcal E_{s_0,2}(\cus_j^\qq;\a,\ahom)\,.
	\label{e.random.corrector.oscillation}
\end{equation}
In this comparison the finite-volume function is taken on~$\cus_{n+2}^\qq$, its~$L^2$ error is restricted to~$\cus_n^\qq$, and its additive constant is removed by subtracting the average. Thus~\eqref{e.random.corrector.oscillation} does not require a boundary estimate for the infinite-volume corrector.

\smallskip

To obtain the stated endpoint norm, we first compare positive test norms on the unit cube. The fundamental theorem of calculus on line segments and integration in the displacement give
\begin{equation}
	\|\psi\|_{\underline H^{s_0}(\cu_0)}^2\leq\|\psi\|_{L^2(\cu_0)}^2+\frac{C(d)}{1-s_0}\|\nabla\psi\|_{L^2(\cu_0)}^2\leq C(d)\|\psi\|_{\underline H^1(\cu_0)}^2\,.
	\label{e.random.endpoint.tests}
\end{equation}
Indeed, the squared translation difference is bounded by~$|h|^2$ times the integral of the squared gradient along the segment; its radial integral is~$\int_0^{\sqrt d}r^{1-2s_0}\,dr$. Since~$s_0<\nf{1}{2}$, the last constant in~\eqref{e.random.endpoint.tests} depends only on~$d$. Applying this inequality to compactly supported tests, and then dilating, gives
\begin{equation}
	3^{-n}\|F\|_{\underline H^{-1}(\cu_n)}\leq C(d)3^{-ns_0}\|F\|_{\underline H^{-s_0}(\cu_n)}\,.
	\label{e.random.endpoint.scaling}
\end{equation}
The two factors are exact: the normalized positive test norms scale as~$3^{-n}$ and~$3^{-ns_0}$, while the averaged pairings are invariant. Apply the weak-norm clause of~\cite[Proposition~5.12]{AK.HC} after dilating the normalized coefficient to the unit cube, where the positive norms and pairings have no volume factors. Only this unit-cube version of the unaveraged hatted estimate of~\cite{AK.HC} is used; its all-test norm bounds our compact-test norm there. The rescaled corrector is~$3^{-n}\phi_e(3^n\qq\,\cdot)$ for slope~$\qq e$: shifting the finite-volume sequence in~\eqref{e.random.corrector.limit} gives precisely this gradient limit. Thus the family has not changed, and~\eqref{e.random.endpoint.scaling} applies to its gradient and centered flux errors.

\smallskip

Returning through~$\qq$ costs at most~$C\Pi^{\nf12}$ in this endpoint dual norm. In fact, change of variables in~\eqref{e.physical.endpoint.norm} gives
\begin{equation*}
	\|\psi(\qq\,\cdot)\|_{\underline H^1(\cu_n)}\leq C(d)\Pi^{\nf12}\|\psi\|_{\underline H^1(\cus_n^\qq)}\,;
\end{equation*}
the squared mass and gradient factors are~$(\det\qq)^{\nf2d}$ and at most~$|\qq|^2$, respectively, and both are bounded by~$C\Pi$. The averaged pairing again cancels the determinant. The matrix factors for the gradient and centered flux are those in Step 2 of the preceding proof, without the scalar dilation, and the slope transforms to~$\qq e$. Since~$\overline\lambda^{\nf12}|\qq e|\leq\frac32|\shom^{\nf12}e|$, the absolute coefficient factors cancel. We obtain
\begin{multline}
	3^{-n}\|\shom^{\nf12}\nabla\phi_e\|_{\underline H^{-1}(\cus_n^\qq)}+3^{-n}\|\shom^{-\nf12}((\a-\khom)(e+\nabla\phi_e)-\shom e)\|_{\underline H^{-1}(\cus_n^\qq)}
	\\\leq C\Pi^{\nf12}|\shom^{\nf12}e|\sum_{j=n}^{\infty}\mathcal E_{s_0,2}(\cus_j^\qq;\a,\ahom)\,.
	\label{e.random.corrector.cube.endpoint}
\end{multline}
For~$r\geq\X$, choose~$n$ with~$4r\leq3^n<12r$, so~$E_r\subseteq\cus_n^\qq$ and the volumes are comparable with constants depending only on~$d$. A compactly supported~$H^1$ test on~$E_r$ extends by zero to the cube with a comparable normalized norm. Restriction of~\eqref{e.random.corrector.cube.endpoint} to~$E_r$, followed by~\eqref{e.random.good.tail}, proves~\eqref{e.random.corrector}. In particular, both inverse-radius factors in the theorem come from~\eqref{e.random.endpoint.scaling}.

\smallskip

\emph{Step 3: One slope for all intermediate radii.} Fix~$\theta\in(0,1)$. Constants in this step may also depend on~$\theta$. Set
\begin{equation*}
	\eta\coloneqq\frac{1+\max\{\theta,\frac{1}{2}\}}2\in\left(\max\left\{\theta,\frac{1}{2}\right\},1\right)\,.
\end{equation*}
Choose~$A_\theta\geq1$, depending only on~$d,\gamma,\theta$, so that~$c(d,\gamma)A_\theta^{-\kappa}$ is below the~$\eta$-dependent thresholds in~\cite[Propositions~5.10 and~5.12]{AK.HC}. By~\eqref{e.random.good.tail}, these propositions then apply on every scale~$3^n\geq A_\theta\X$. The finite-volume sequence defining the correctors is still~\eqref{e.random.corrector.limit}; changing~$\eta$ does not change its limit.

\smallskip

We first show that, for~$3^n\geq A_\theta\X$,~$m\geq n$, and a finite-energy weak solution~$u$ on~$\cus_m^\qq$, there is one~$e\in\Rd$ such that
\begin{equation}
	\|\s^{\nf12}(\nabla u-e-\nabla\phi_e)\|_{\underline L^2(\cus_k^\qq)}\leq C(d,\gamma,\theta)3^{-\eta(m-k)}\|\s^{\nf12}\nabla u\|_{\underline L^2(\cus_m^\qq)}\qquad(n\leq k\leq m)\,.
	\label{e.random.intermediate.slope}
\end{equation}
We choose slopes separately at each scale, then compare neighboring slopes and telescope their differences to retain one slope for all intermediate radii. For each~$j\in[n,m]\cap\Z$, the corrector approximation clause of~\cite[Proposition~5.12]{AK.HC} gives a slope~$e_j$ whose error on~$\cus_j^\qq$ is at most~$C3^{-\eta(m-j)}\|\s^{\nf12}\nabla u\|_{\underline L^2(\cus_m^\qq)}$. The gap is~$m-j$; the corresponding index in~\cite{AK.HC} is misprinted. By linearity of the corrector gradients, the lower bound in~\eqref{e.random.corrector.energy}, the triangle inequality, and restriction from~$\cus_{j+1}^\qq$ to~$\cus_j^\qq$,
\begin{equation*}
	|\shom^{\nf12}(e_j-e_{j+1})|\leq C3^{-\eta(m-j)}\|\s^{\nf12}\nabla u\|_{\underline L^2(\cus_m^\qq)}\qquad(n\leq j<m)\,.
\end{equation*}
Sum these increments from~$j=n$ to~$k-1$ and use the upper bound in~\eqref{e.random.corrector.energy} on~$\cus_k^\qq$. The geometric sum gives
\begin{equation*}
	\|\s^{\nf12}(e_n+\nabla\phi_{e_n}-e_k-\nabla\phi_{e_k})\|_{\underline L^2(\cus_k^\qq)}\leq C3^{-\eta(m-k)}\|\s^{\nf12}\nabla u\|_{\underline L^2(\cus_m^\qq)}\,.
\end{equation*}
Combining this with the error bound for~$e_k$ proves~\eqref{e.random.intermediate.slope} with~$e=e_n$ for every~$k\in[n,m]$.

\smallskip

Increase~$A_\theta$ by a factor depending only on~$d$ to pass between cubes and ellipsoids. The same inclusions as in Step 1 give a single slope~$e$ such that
\begin{equation}
	\|\s^{\nf12}(\nabla u-e-\nabla\phi_e)\|_{\underline L^2(E_r)}\leq C(d,\gamma,\theta)\left(\frac{r}{R}\right)^\eta\|\s^{\nf12}\nabla u\|_{\underline L^2(E_R)}\qquad(A_\theta\X\leq r\leq R)\,.
	\label{e.random.regularity.above.cutoff}
\end{equation}
When the outer radius is too close to~$A_\theta\X$ to insert two different adapted scales, use~$e=0$ and~\eqref{e.random.energy}; the bounded ratio is absorbed into~$C(d,\gamma,\theta)$. For the other outer radii, the chosen slope also satisfies~$|\shom^{\nf12}e|\leq C(d,\gamma,\theta)\|\s^{\nf12}\nabla u\|_{\underline L^2(E_R)}$, by the corrector energy comparison and~\eqref{e.random.intermediate.slope}. This controls the radii comparable to~$R$ by the energy estimate and inclusion.

\smallskip

It remains to reach the original radius~$\X$. If~$R\geq A_\theta\X$ and~$\X\leq r<A_\theta\X$, apply the finite-energy form of~\eqref{e.random.energy} to~$u-e\cdot x-\phi_e$, then use~\eqref{e.random.regularity.above.cutoff} at~$A_\theta\X$. This gives
\begin{align}
	\|\s^{\nf12}(\nabla u-e-\nabla\phi_e)\|_{\underline L^2(E_r)}&\leq C(d,\gamma,\theta)\left(\frac{A_\theta\X}{R}\right)^\eta\|\s^{\nf12}\nabla u\|_{\underline L^2(E_R)}
	\notag\\&\leq C(d,\gamma,\theta)A_\theta^\theta\left(\frac{r}{R}\right)^\theta\|\s^{\nf12}\nabla u\|_{\underline L^2(E_R)}\,.
	\label{e.random.regularity.bridge}
\end{align}
The second inequality uses~$\eta>\theta$,~$A_\theta\X\leq R$, and~$r\geq\X$. If~$R<A_\theta\X$, take~$e=0$ and use~\eqref{e.random.energy}, since~$(\nf{r}{R})^\theta\geq A_\theta^{-\theta}$. Absorbing~$A_\theta^\theta$ into the constant proves~\eqref{e.random.regularity} with the asserted order~$\forall\theta\,\exists e\,\forall r$ and the original radius~$\X$.

\smallskip

\emph{Step 4: Liouville and uniqueness of the normalized slope.} Let~$v$ satisfy~\eqref{e.random.liouville.growth}. The coarse-grained Caccioppoli estimate on successive adapted cubes, with~\eqref{e.random.good.tail}, gives
\begin{equation}
	\|\s^{\nf12}\nabla v\|_{\underline L^2(\cus_n^\qq)}\leq C\overline\lambda^{\nf12}3^{-n}\|v-(v)_{\cus_{n+1}^\qq}\|_{\underline L^2(\cus_{n+1}^\qq)}=o(3^{\theta n})\,.
	\label{e.random.liouville.gradient}
\end{equation}
The last equality follows by enclosing~$\cus_{n+1}^\qq$ in a Euclidean ball of radius~$C|\qq|3^n$; the fixed volume ratio does not affect the growth exponent. Since~$\eta>\theta$, the gradient-growth Liouville clause of~\cite[Proposition~5.12]{AK.HC}, applied above~$A_\theta\X$, identifies~$\nabla v$ with~$e+\nabla\phi_e$ for some~$e$. The construction is the same one in~\eqref{e.random.corrector.limit}. Equality of gradients on~$\Rd$ gives~$v=e\cdot x+\phi_e+c$ almost everywhere.

\smallskip

Conversely, let~$v=e\cdot x+\phi_e+c$. By~\eqref{e.random.corrector.energy} and the coarse-grained Poincar\'e inequality,
\begin{equation*}
	\|v-(v)_{\cus_n^\qq}\|_{\underline L^2(\cus_n^\qq)}\leq C\overline\lambda^{-\nf12}3^n|\shom^{\nf12}e|\qquad(3^n\geq\X)\,.
\end{equation*}
The difference of the averages on~$\cus_n^\qq$ and~$\cus_{n+1}^\qq$ is bounded by the same right side with~$n+1$ in place of~$n$. Summing these bounds from one fixed scale yields~$\|v\|_{\underline L^2(\cus_n^\qq)}=O(3^n)$, with a constant depending on the fixed solution and coefficient field. Enclosing~$B_r$ in an adapted cube of comparable side length gives~$\|v\|_{\underline L^2(B_r)}=O(r)$. This is~$o(r^{1+\theta})$ for every~$\theta>0$, proving the reverse inclusion.

\smallskip

We also record the uniqueness needed for translations. By~\eqref{e.random.corrector.oscillation}, the deviation of~$e\cdot x+\phi_e$ from its affine slope has~$o(3^n)$ oscillation on~$\cus_n^\qq$. If an entire harmonic function with the same sublinear deviation from slope~$e$ is represented by Liouville as~$e'\cdot x+\phi_{e'}+c$, then
\begin{equation*}
	\frac1{\sqrt{12}}|\qq(e-e')|=3^{-n}\|(e-e')\cdot x-((e-e')\cdot x)_{\cus_n^\qq}\|_{\underline L^2(\cus_n^\qq)}\longrightarrow0\,.
\end{equation*}
Thus~$e'=e$, and the two functions differ by a constant. This proves uniqueness of the gradient with a prescribed normalized slope; no condition on its expectation is used.

\smallskip

\emph{Step 5: The stationary family and the common event.} Intersect the full-probability event used above with all its integer translates. This is a countable intersection and is invariant under~$\Zd$-translations. On it construct~\eqref{e.random.corrector.limit} first for the coordinate slopes and extend linearly. For each~$z\in\Zd$, the construction for~$\a(\cdot+z)$ uses its own starting radius; no supremum over these radii is taken. For fixed~$z\in\Zd$, translation of~$e\cdot x+\phi_e$ is harmonic for~$\a(\cdot+z)$, has~$O(r)$ growth by Step 4, and has sublinear deviation from the same slope~$e$ by~\eqref{e.random.corrector.oscillation}. The last assertion follows by enclosing the translated adapted cube in one concentric cube of comparable side length and then letting its side length tend to infinity. Applying the normalized-slope uniqueness on the translated coefficient gives, with the coefficient dependence displayed,
\begin{equation}
	\nabla\phi_e(\a(\cdot+z),x)=\nabla\phi_e(\a,x+z)\qquad(z\in\Zd)\,.
	\label{e.random.corrector.stationarity}
\end{equation}
Define the gradients to be zero on the invariant null complement. Their linearity and the translation identity are then consistent there as well. The estimates at the origin use the single radius~$\X$ in~\eqref{e.random.final.scale}; neither this event nor the family is chosen anew for a Sobolev order, an approximation exponent, a solution, or a radius. All remaining quantifiers in Theorem~\ref{t.random.homogenization} are deterministic consequences on this event. This completes the proof.
\end{proof}

\begingroup
\small

\subsubsection*{\bf Acknowledgments}
S.A. and T.K. acknowledge support from the European Research Council (ERC) under the European Union's Horizon Europe research and innovation programme, grant agreement number 101200828. T.K. was supported by the Academy of Finland. A.L. was supported by the Fondation Sciences Math\'ematiques de Paris.

\bibliographystyle{alpha}
\bibliography{refs}

\begin{thebibliography}{CGLL90}

\bibitem[AK25]{AK.HC}
Scott Armstrong and Tuomo Kuusi.
\newblock Renormalization group and elliptic homogenization in high contrast.
\newblock {\em Invent. Math.}, 242(3):895--1086, 2025.

\bibitem[AK26]{AK.ICM}
Scott Armstrong and Tuomo Kuusi.
\newblock A coarse-graining theory for elliptic operators and homogenization in
  high contrast.
\newblock In {\em Proceedings of the International Congress of Mathematicians
  2026}, volume~5, pages 193--212. Society for Industrial and Applied
  Mathematics, 2026.

\bibitem[AKM17]{AKM1}
Scott Armstrong, Tuomo Kuusi, and Jean-Christophe Mourrat.
\newblock The additive structure of elliptic homogenization.
\newblock {\em Invent. Math.}, 208(3):999--1154, 2017.

\bibitem[AKM19]{AKMBook}
Scott Armstrong, Tuomo Kuusi, and Jean-Christophe Mourrat.
\newblock {\em Quantitative Stochastic Homogenization and Large-Scale
  Regularity}, volume 352 of {\em Grundlehren der mathematischen
  Wissenschaften}.
\newblock Springer, Cham, 2019.

\bibitem[AS16]{AS}
Scott~N. Armstrong and Charles~K. Smart.
\newblock Quantitative stochastic homogenization of convex integral
  functionals.
\newblock {\em Ann. Sci. \'{E}c. Norm. Sup\'{e}r.}, 49(2):423--481, 2016.

\bibitem[CGLL90]{Clerc}
J.~P. Clerc, G.~Giraud, J.~M. Laugier, and J.~M. Luck.
\newblock The electrical conductivity of binary disordered systems, percolation
  clusters, fractals and related models.
\newblock {\em Adv. Phys.}, 39(3):191--309, 1990.

\bibitem[GO11]{GO1}
Antoine Gloria and Felix Otto.
\newblock An optimal variance estimate in stochastic homogenization of discrete
  elliptic equations.
\newblock {\em Ann. Probab.}, 39(3):779--856, 2011.

\bibitem[GO12]{GO2}
Antoine Gloria and Felix Otto.
\newblock An optimal error estimate in stochastic homogenization of discrete
  elliptic equations.
\newblock {\em Ann. Appl. Probab.}, 22(1):1--28, 2012.

\bibitem[HSCB86]{HongEtAl1986}
D.~C. Hong, H.~E. Stanley, A.~Coniglio, and A.~Bunde.
\newblock Random-walk approach to the two-component random-conductor mixture:
  Perturbing away from the perfect random resistor network and random
  superconducting-network limits.
\newblock {\em Phys. Rev. B}, 33(7):4564--4573, 1986.

\bibitem[JK95]{JeKe}
David Jerison and Carlos~E. Kenig.
\newblock The inhomogeneous {Dirichlet} problem in {Lipschitz} domains.
\newblock {\em J. Funct. Anal.}, 130(1):161--219, 1995.

\bibitem[Mik11]{Mikh11}
Sergey~E. Mikhailov.
\newblock Traces, extensions and co-normal derivatives for elliptic systems on
  {Lipschitz} domains.
\newblock {\em J. Math. Anal. Appl.}, 378(1):324--342, 2011.

\bibitem[Str76]{Straley1976}
J.~P. Straley.
\newblock Critical phenomena in resistor networks.
\newblock {\em J. Phys. C: Solid State Phys.}, 9(5):783--795, 1976.

\bibitem[TLW90]{TobochnikLaingWilson1990}
Jan Tobochnik, David Laing, and Gary Wilson.
\newblock Random-walk calculation of conductivity in continuum percolation.
\newblock {\em Phys. Rev. A}, 41(6):3052--3058, 1990.

\end{thebibliography}

\end{document}